\documentclass[11pt]{article}

\usepackage{amsfonts, amsmath,amssymb,array,latexsym, hyperref, amscd}

\usepackage{fancyhdr,a4wide}
\usepackage{amsthm}
\usepackage{mathtools} 
\usepackage{subcaption}
\usepackage{xcolor}

\newtheorem{theorem}{Theorem}[section]
\newtheorem{lemma}[theorem]{Lemma}
\newtheorem{corollary}[theorem]{Corollary}
\newtheorem{proposition}[theorem]{Proposition}

\newtheorem{conjecture}[theorem]{Conjecture}

\newtheorem{remark}[theorem]{Remark}
\newtheorem{assumption}[theorem]{Assumption}

\numberwithin{equation}{section}

\newcommand{\pr}[1]{\mathbb P\left(#1\right)}
\newcommand{\bk}[1]{\left \langle #1 \right \rangle}

\newcommand{\eps}{\varepsilon}
\newcommand{\norm}[1]{\left|\left|#1\right|\right|}
\newcommand{\lr}[1]{\left(#1\right)}
\newcommand{\abs}[1]{\left|#1\right|}
\newcommand{\set}[1]{\left\{#1\right\}}
\newcommand{\E}[1]{\mathbb E\left[#1\right]}

\def\Var{\mathrm{Var}}  
\def\Cov{\mathrm{Cov}} 
\newcommand{\x}{\times}
\newcommand{\R}{\mathbb R}

\newcommand{\gives}{\rightarrow}

\newcommand{\mO}{\mathcal O}
\newcommand{\mA}{\mathcal A}
\newcommand{\mF}{\mathcal F}
\newcommand{\mN}{\mathcal N}
\newcommand{\mL}{\mathcal L}

\newcommand{\mG}{\mathcal G}
\newcommand{\amO}{\overrightarrow{\mO}}

\newcommand{\Z}{\mathbb Z}
\newcommand{\N}{\mathbb N}
\newcommand{\bkl}[1]{\bk{#1}_{(\ell)}}
\newcommand{\bkkl}[1]{\bk{#1}_{K^{(\ell)}}}
\newcommand{\bkal}[1]{\bk{#1}_{\kappa^{(\ell)}}}
\newcommand{\twomat}[4]{\lr{\begin{array}{cc} #1 & #2 \\ #3 & #4 \end{array}}}
\newcommand{\wkappa}{\widehat{\kappa}}
\newcommand{\kappal}{\kappa^{(\ell)}}

\newcommand{\Disq}[1]{\frac{\partial^2}{\partial #1^2}}
\newcommand{\DDi}[2]{\frac{\partial^2}{\partial #1\partial #2}}

\newcommand{\K}[3]{K_{(#1#2)}^{(#3)}}
\newcommand{\Kell}[2]{K_{(#1#2)}^{(\ell)}}
\newcommand{\kell}[2]{\kappa_{(#1)(#2)}^{(\ell)}}
\newcommand{\zz}{\sigma^2}
\newcommand{\oz}{\sigma\sigma'dz}
\newcommand{\oo}{(\sigma'dz)^2}
\newcommand{\tz}{(\sigma\sigma'' (dz)^2 + \sigma\sigma' d^2z)}
\newcommand{\tzs}{\sigma\sigma'' (dz)^2 + \sigma\sigma' d^2z}
\renewcommand{\to}{(\sigma'\sigma'' (dz)^3 + (\sigma')^2dz d^2z)}
\newcommand{\tos}{\sigma'\sigma'' (dz)^3 + (\sigma')^2dz d^2z}
\renewcommand{\tt}{(\sigma'' (dz)^2 + \sigma' d^2z)^2}

\title{Random Fully Connected Neural Networks as Perturbatively Solvable Hierarchies}
\author{Boris Hanin\footnote{BH gratefully acknowledges support by the NSF through  DMS-2143754, DMS-1855684, and DMS-2133806 as well support from an ONR MURI on Foundations of Deep Learning.}\\
Department of Operations Research\\
and Financial Engineering\\
Princeton University\\
bhanin@princeton.edu}
\begin{document}
\maketitle
\begin{abstract}
    This article considers fully connected neural networks with Gaussian random weights and biases as well as $L$ hidden layers, each of width proportional to a large parameter $n$. For polynomially bounded non-linearities
    we give sharp estimates in powers of $1/n$ for the joint cumulants of the network output and its derivatives. Moreover, we show that network cumulants form a perturbatively solvable hierarchy in powers of $1/n$ in that $k$-th order cumulants in one layer have recursions that depend to leading order in $1/n$ only on $j$-th order cumulants at the previous layer with $j\leq k$. By solving a variety of such recursions, however, we find that the depth-to-width ratio $L/n$ plays the role of an effective network depth, controlling both the scale of fluctuations at individual neurons and the size of inter-neuron correlations. Thus, while the cumulant recursions we derive form a hierarchy in powers of $1/n$, contributions of order $1/n^k$ often grow like $L^k$ and are hence non-negligible at positive $L/n$. We use this to study a somewhat simplified version of the exploding and vanishing gradient problem, proving that this particular variant  occurs if and only if $L/n$ is large. Several key ideas in this article were first developed at a physics level of rigor in a recent monograph of Daniel A. Roberts, Sho Yaida, and the author. This article not only makes these ideas mathematically precise but also significantly extends them, opening the way to obtaining corrections to all orders in $1/n$.
\end{abstract}

\section{Introduction}
We live in an era of big data and cheap computation. This has led to remarkable progress in domains ranging from self-driving cars \cite{krizhevsky2012imagenet} to automatic drug discovery \cite{jumper2021highly} and machine translation \cite{brown2020language}. Underpinning many of these exciting practical developments is a class of computational models called neural networks. While they were originally developed in the 1940's and 1950's \cite{hebb1949,rosenblatt1958perceptron}, the complexity of state-of-the-art neural nets is unprecedented. And yet, despite their empirical utility, a theoretical understanding of how they work and how to make them better is nascent. In fact, it is sometimes said that neural networks are simply too complicated to allow for a rigorous understanding of their key features. 

This article adds to the growing body of literature to the contrary. Namely, in the simplest important setting of fully connected networks, we develop a flexible set of probabilistic tools for studying correlation functions of \textit{random fully connected neural networks} at finite width.

A random fully connected neural network, defined precisely in \S \ref{S:what}, is a random field whose law is determined by a fixed and typically non-linear function $\sigma:\R\gives \R$ as well as two structural parameters: a depth $L\in \N$ and a width $n\in \N$. For a fixed depth $L$, in the limit of infinite width $n\gives\infty$ random neural networks converge in distribution to Gaussian processes (see Theorem \ref{T:iwl}). To study finite width effects we therefore consider the higher cumulants for the distribution of the outputs. Our approach is recursive in the network depth $L$ and perturbatively in the reciprocal $1/n$ of the network width. Our main results are:

\begin{itemize}
\item We give in Theorem \ref{T:cumulants} sharp estimates in powers of $1/n$ for the the size of all joint cumulants of networks outputs and their derivatives. These can be viewed as quantitative results refinements of the  central limit theorem at fixed $L$ and large $n$ (cf Theorems \ref{T:iwl} and Theorem \ref{T:pert-exp}).

\item We derive exact recursions with respect to depth that describe cumulants at one layer in terms of cumulants at the previous layer (see e.g. Corollary \ref{C:cumulants-vals} and Proposition \ref{P:4pt-deriv-recs-tanh}). These recursions take the form of perturbatively solvable hierarchies in the sense that the recursion for the $k$-th cumulant at layer $\ell+1$ involves, at leading order in $1/n$ only the $j$-th cumulants at layer $\ell$ with $j\leq k$ (see Corollaries \ref{cor:hier} and \ref{C:cumulants-vals}). This is distinct from but similar in spirit  to the remarkable article \cite{huang2020dynamics}, which derived a dynamical perturbative hierarchy for the so-called neural tangent kernel. As in our work, the perturbative parameter is $1/n$. However, the emphasis in \cite{huang2020dynamics} is primarily on training dynamics while ours is on understanding  the effect of depth. 

\item Solving some special cases of the cumulant recursions to leader order in $1/n$ reveals that it is the depth-to-width ratio $L/n$, which we term the \textit{effective network depth}, rather than the apparent depth $L$, that is a more informative measure of neural network depth and complexity. Mathematically, this suggests a non-trivial double scaling limit for random neural networks in which 
\[
n,L\gives \infty\qquad \text{and}\qquad L/n\gives \xi\in [0,\infty).
\]
For non-linear networks this scaling limit has only started to be considered \cite{hanin2018neural,hanin2020products, roberts2022principles, li2022neural}. Even in the very special case of product of $L$ iid random $n\times n$  matrices (sometimes called deep linear networks) the simultaneous large $n,L$ regime has revealed a range of interesting and not fully understood properties (see e.g. references in  \cite{akemann2012universal,akemann2020universality} as well as \cite{gorin2018gaussian,hanin2020products,hanin2021non,liu2018lyapunov}). In contrast to the $\xi=0$ regime typically considered in previous work on neural networks (c.f. e.g. \cite{du2017gradient,du2018gradient,jacot2018neural,liu2020toward}), we show that when $\xi>0$ our double scaling limit is capable of exhibiting non-Gaussian and non-linear effects. We find the following effects to leading order in $1/n$:  \\
\begin{itemize}
\item We prove in Corollary \ref{C:tanh-cumulants} that for any non-linearity $\sigma$ from the $K_*=0$ universality class (defined in \S \ref{S:tanh-univ}) and for $k=2,3,4$, the $2k^{th}$ cumulants of the output of a random neural network with non-linearity $\sigma$ grow like $(L/n)^{k/2-1}$. This implies, for instance, that both the correlations between neurons and the fluctuations of a single neuron grow like the effective depth $L/n$ at large $L,n$ (see Remark \ref{R:tanh-cumulants} and just after Theorem \ref{T:derivs-cumulants}). Since the components of the output of a depth $0$ neural network are simply iid Gaussians, we see that the effective depth $L/n$ can be thought of as a measure of how close a random fully connected network is to a Gaussian process. Related questions were considered in \cite{basteri2022quantitative,eldan2021non}.


\item We show in Theorems \ref{T:derivs-cumulants} and \ref{T:evgp} that, for random neural networks initialized as in practice (see \S \ref{S:crit}), the variance of the gradient of the network output with respect to either its input or a trainable parameter in its first layer grows like $L/n$ at large $L$. As explained in \S \ref{S:evgp-results}, this gives the first mathematical characterization for fairly general fully connected networks of the so-called exploding and vanishing gradient problem (EVGP). Previously, this problem was solved rigorously in the important special case of random ReLU networks \cite{hanin2018neural,hanin2020products} and solved at a physics level of rigor by a somewhat different but related set of ideas in the monograph \cite{roberts2022principles} of Roberts, Yaida, and the author.\\
 
\end{itemize}

\end{itemize}

The remainder of the introduction is structured as follows. We begin by giving in \S \ref{S:what}  the precise definition of random fully connected neural networks. We then formulate and motivate the main question taken up in this article in \S \ref{S:main-q}. 


\section{Background and Motivation}\label{S:background-motivation}

\subsection{What is a (Random) Fully Connected Neural Network?}\label{S:what}
Neural networks are parameterized families of functions. The simplest kind of networks are called fully connected. Each such network is specified by its \textit{architecture}, which consists of an input dimension $n_0\in \Z_+$, a network depth $L\in \Z_+$, hidden layer widths $n_1,\ldots, n_L\in \Z_+$, an output dimension $n_{L+1}\in \Z_{+}$, and a non-linearity $\sigma:\R\gives\R$. The functions computed by a fully connected network with a given architecture are all maps from $\R^{n_0}$ to $\R^{n_{L+1}}$ that associate to each  network input $x_\alpha\in \R^{n_0}$ an output $z_\alpha^{_{(L+1)}}\in \R^{n_{L+1}}$ through a sequence of intermediate representations $z_\alpha^{_{(\ell)}}\in \R^{n_\ell}$ as follows:
\begin{equation}\label{E:z-def}
z_{\alpha}^{(\ell+1)}:=\begin{cases} b^{(\ell+1)}+W^{(\ell+1)}\sigma(z_{\alpha}^{(\ell)}),&\quad \ell\geq 1\\
b^{(1)}+W^{(1)}x_{\alpha},&\quad \ell=0
\end{cases},\qquad W^{(\ell+1)}\in \R^{n_{\ell+1}\times n_{\ell}},\, b^{(\ell+1)}\in \R^{n_{\ell+1}}.
\end{equation}
In this recursion, the univariate function $\sigma$ applied to a vector $z_\alpha^{_{(\ell)}}\in \R^{n_\ell}$ is short-hand for applying it separately to each component. The entries of the matrices $W^{(\ell)}$ and the components of the vectors $b^{(\ell)}$ are called the \textit{weights and biases} in layer $\ell$, respectively. One typically refers to 
\[
z_{\alpha}^{(\ell)} =\lr{z_{1;\alpha}^{(\ell)},\ldots, z_{n_{\ell};\alpha}^{(\ell)}}\in \R^{n_{\ell}}
\]
as the \textit{vector of pre-activations at layer $\ell$} corresponding to the input $x_\alpha$. 
The most popular choices of $\sigma$ in practice include $\mathrm{ReLU}(t):=\max\set{0,t}$ as well the hyperbolic tangent and their variations. We will analyze these cases in detail later (see \S \ref{S:univ}), but for our general results make only the following mild assumption
\begin{assumption}\label{A:sigma-prop}
There exists $r\geq 1$ so that the $r$-th derivative of $\sigma$ exists almost everywhere and grows at most polynomially:
\[
\exists k\geq 1 \text{ s.t. }\mathrm{sup}_{x\in \R} \abs{(1+\abs{x})^{-k}\frac{d^r}{dx^r}\sigma(x)}<\infty.
\]
\end{assumption}
The primary objects of study in this article are \textit{random fully connected neural networks}, obtained by choosing network weights and biases to be independent centered Gaussians:
\begin{equation}\label{E:Wb-def}
W_{ij}^{(\ell)}\sim\mN(0,C_W/n_{\ell-1}),\qquad b_i^{(\ell)}\sim \mN(0,C_b)\qquad \text{independent}.
\end{equation}
Here $C_b\geq 0,C_W>0$ are fixed constants. The $1/n_{\ell-1}$ scaling in the weight variance ensures that the moments of the outputs $z_\alpha^{_{(L+1)}}$ remain uniformly bounded as $n_1,\ldots, n_L,\gives \infty$ (see e.g. Theorem \ref{T:iwl} and \eqref{E:cond-gauss-vals}). While we carry out a substantial portion of our analysis for any choice of $C_b,C_W$, as we will see in \S \ref{S:crit}, there often exist distinguished $\sigma$-dependent settings of $C_b,C_W$ for which random neural networks at infinite width $n_1,\ldots, n_L\gives \infty$ are well-behaved at large depth $L$.

\subsection{Statement and Motivation for Questions Addressed in this Article}\label{S:main-q}
The main problem we take up in the present article is to characterize the finite dimensional distributions of a random neural network $x_\alpha\mapsto z_\alpha^{_{(L+1)}}$ (and its derivatives with respect to $x_\alpha$) in the regime where the input dimension $n_0$ is arbitrary with the inputs $x_\alpha$ satisfying
\[
\frac{1}{n_0}\norm{x_\alpha}_2^2 < \infty,
\]
the output dimension $n_{L+1}$ is fixed, and the hidden layer widths $n_\ell$ are large but finite:
\begin{equation}\label{E:n-def}
\exists c,C>0\text{ s.t. }\qquad c n \leq n_1,\ldots, n_L\leq Cn,\qquad \qquad n\gg 1.
\end{equation}
Our approach will be to describe the random field $x_\alpha\mapsto z_\alpha^{_{(L+1)}}$ perturbatively in $1/n$ and recursively in $L$. 
Before proceeding to the technical statements of our results, we pause to address below the following motivational questions:
\begin{itemize}
\item Why study random neural networks? 
\item Why treat $1/n$ as a perturbative parameter? 
\item Why study networks at finite width? 
\item What role does $\sigma$ play? 
\end{itemize}
The primary use of a neural network with a fixed architecture is to find a setting of its parameters $W^{{(\ell)}}, b^{{(\ell)}}$ giving rise to a mapping $x_\alpha\mapsto z_\alpha^{_{(L+1)}}$ as in \eqref{E:z-def} that matches a dataset $\set{(x_i,f(x_i))}$ of values for some otherwise unknown function $f:\R^{n_0}\gives \R^{n_{L+1}}$. The optimization procedure for finding these weights and biases is almost always some variant of gradient descent starting with $W^{{(\ell)}}, b^{{(\ell)}}$ drawn at random from the distribution \eqref{E:Wb-def}. Thus, studying the properties of random neural networks gives direct insights into the starting conditions for neural network optimization. For instance, understanding the behavior of neural networks at the start of training gives a principled way to set optimization hyperparameters (e.g. the variances $C_b,C_W$ and the step size used for gradient descent). We refer the interested reader our discussion of criticality in \S \ref{S:crit-univ} as well as to \S 3.2 in \cite{roberts2022principles} the articles articles \cite{he2022understanding,martens2021rapid,schoenholz2016deep,zhang2022deep} for more on this point.

Next, to address why $1/n$ can reasonably be treated perturbatively, we recall that networks used in practice often have a very large number of parameters. This is reflected in the fact that both the layer widths $n_\ell$ and the network depth $L$ are often big. Thus, it is sensible to first understand various limits of neural networks in which the number of parameters tends to infinity. The most well-studied regime of this type (though not the only option cf eg \cite{mei2018mean,rotskoff2018neural,sirignano2020mean,sirignano2021mean,yang2021tensoriv}) is the infinite width limit, also known as the \textit{NTK regime}. By definition, this regime is accessed by fixing the depth $L$, the input and output dimensions $n_0,n_{L+1}$, and the non-linearity $\sigma$, the initialization scheme \eqref{E:Wb-def} and considering the limit when $n_1,\ldots, n_L\gives \infty$. From the random matrix theory point of view this is the free probability regime. In view of the relation \eqref{E:n-def} the NTK regime is obtained by taking $n\gives \infty$ at fixed $L$ and has two salient features: \\
\begin{itemize}
	\item At the start of training, neural networks converge to a Gaussian process (see Theorem \ref{T:iwl} below) \cite{lee2017deep,matthews2018gaussian,neal1996priors,novak2018bayesian,yang2019scaling,yang2019tensori,yang2018deep}. 
	\item For the purposes of optimization of a squared loss, the network can be replaced by \textit{its linearization at the start of training} (see \cite{chizat2018note,du2017gradient,jacot2018neural,lee2019wide,liu2020toward}). \\
\end{itemize}

Taken together, these two points show that at least at infinite width and finite depth, it is the structure of the network at initialization that determines not only the start of training but really the entire training trajectory. However, the infinite width limit is too rigid to capture the ability of real work networks to learn data-dependent features (see e.g. \cite{hanin2019finite,huang2020dynamics,roberts2022principles,yang2021tensoriv}). Only finite width networks can capture these effects! Since the starting point for our analysis of random neural networks at finite width is their infinite width behavior, we take this opportunity to record the following result about random neural networks in the NTK regime. 






\begin{theorem}[Random Networks at fixed $L$ and Infinite Width are Gaussian Processes]\label{T:iwl}
    Fix a non-negative integer $r\geq 0$, and suppose $\sigma:\R\gives \R$ is $r$-times differentiable and that its $r$-th derivative is polynomially bounded: 
    \[
    \exists k\geq 1\text{ s.t. } \sup_{x\in \R}\abs{(1+\abs{x})^{-k}\frac{d^r}{dx^r}\sigma(x)}<\infty. 
    \]
    Then the finite-dimensional distributions of the stochastic process $x_\alpha\mapsto z_\alpha^{(L+1)}$ and its derivatives of order up to $r$ converge to those of a centered Gaussian process with $n_{L+1}$ iid components. The limiting variance of each component 
    \[
    K_{\alpha\beta}^{(L+1)}:=\lim_{n_1,\ldots, n_L\gives \infty}\Cov\lr{z_{i;\alpha}^{(L+1)}, z_{i;\beta}^{(L+1)}},\qquad x_\alpha,x_\beta\in\R^{n_0},
    \]
satisfies the recursion
    \begin{align}
       \label{E:K-rec} K_{\alpha\beta}^{(\ell+1)} = \begin{cases}
        C_b + C_W \bk{\sigma(z_\alpha)\sigma(z_\beta)}_{K^{(\ell)}},&\quad \ell \geq 1\\
        C_b + \frac{C_W}{n_0} \sum_{j=1}^{n_0}x_{j;\alpha}x_{j;\beta},&\quad \ell=0
        \end{cases}.
    \end{align}
\end{theorem}






In the statement of Theorem \ref{T:iwl} and henceforth we reserve the symbol $\bk{f(z_\alpha,z_\beta)}_{\kappa}$ to denote the expectation of $f(z_\alpha,z_\beta)$ with respect to the Gaussian distribution
\[
    \lr{z_\alpha,z_\beta} \sim \mathcal N\lr{0,\twomat{\kappa_{\alpha\alpha}}{\kappa_{\alpha\beta}}{\kappa_{\alpha\beta}}{\kappa_{\beta\beta}}},
\]
where $\kappa_{\alpha\beta}=\kappa(x_\alpha,x_\beta)$ is a given covariance function. The conclusion in Theorem \ref{T:iwl} is not new, having been obtained many times and under a variety of different assumptions (including for more general architectures) \cite{hanin2021deep,lee2017deep,matthews2018gaussian,poole2016exponential,yang2019tensori}. We refer the interested reader to \cite{hanin2021deep} for a discussion of prior work and note only that convergence of the derivatives of the field $z_{\alpha}^{_{(L+1)}}$ to its Gaussian limit does not seem to have been previously considered. We give a short proof that includes convergence of derivatives along the lines of the arguments in \cite{hanin2021deep,lee2017deep} in Appendix \S \ref{A:iwl}. 

Let us remark that the relation between $K^{{(\ell+1)}}$ to $K^{(\ell)}$ supplied by the recursion \eqref{E:K-rec} is in general non-linear due to the presence of $\sigma$ and depends crucially on the weight and bias variances $C_W,C_b$. Given $\sigma$, finding a choice of $C_b,C_W$ so that this recursion is well-behaved (e.g. not exponentially growing or decaying) at large $\ell$ is an important practical matter \cite{hubara2017quantized,lee2017deep,park2019effect}. We explain this in \S \ref{S:crit}, where we also point out that the different possible large $\ell$ limits suggest the existence of universality classes of random neural networks. 

Finally, as we alluded to above, the tremendous simplification of random neural networks in the NTK regime comes at a steep cost in terms of the descriptive power of the resulting model in the sense that such models cannot learn data-dependent features. We again refer to reader to articles such as \cite{chizat2018global,mei2018mean, rotskoff2018neural,sirignano2020mean,sirignano2021mean,yang2021tensoriv} in which a different initialization scheme leads to infinite width limits capable of feature learning. Since we do not treat feature learning in this article, we will not elaborate further on this point, referring the interested reader to \cite{chizat2018note,hanin2019finite,huang2020dynamics,roberts2022principles, woodworth2020kernel, yang2021tensoriv} instead.

Finally, before formulating the new results in this article, we point the interested to a final tranche of   physics-style references such as \cite{andreassen2020asymptotics, cohen2021learning,dyer2019asymptotics,koch2018mutual,lewkowycz2020large,naveh2021predicting,naveh2021self,seroussi2021separation}, which analyze a variety of questions related to either very wide or infinitely wide networks. 

\section{Results}\label{S:results}

In this section we give precise formulations of our results on random neural networks. We start in \S \ref{S:precise-results} by stating a structural result, Theorem \ref{T:cumulants}, on the size of the joint cumulants of a random neural networks and its derivatives in powers of $1/n$. We then extend this order of magnitude estimate to obtain in Theorem \ref{T:pert-exp} a general prescription for obtaining series expansions in powers of $1/n$ for expectation values of functions of a random depth $\ell+1$ neural network in terms of those of a random depth $\ell$ network. An application of Theorem \ref{T:pert-exp} yields explicit recursions for the $4^{th},6^{th}, 8^{th}$ joint cumulants for components $z_{i;\alpha}^{_{(\ell)}}$ of the vector of pre-activations at layer $\ell$ corresponding to a fixed network input $x_\alpha.$ These recursions are recorded in Corollary \ref{C:cumulants-vals}. See \S \ref{S:proof-overview} for an overview of the proof of Theorems \ref{T:cumulants} and \ref{T:pert-exp} as well as Corollary \ref{C:cumulants-vals}.

After stating these results, we describe in \S \ref{S:crit-univ} the important notion of criticality in random neural networks, which corresponds to choosing values of the weight and bias variance parameters $C_b,C_W$ in \eqref{E:Wb-def} depending on $\sigma$ in such a way that the infinite width covariance $K_{\alpha\beta}^{_{(\ell)}}$ (and in fact the higher cumulants at finite width as well) is well-behaved at large $\ell$. The variety of resulting large $\ell$ behaviors determines universality classes of random neural networks, as we explain in \S \ref{S:univ}. 

We then turn in \S \ref{S:relu-univ} and \S \ref{S:tanh-univ} to solving to the $2k$-th cumulant recursions for $k=1,2,3,4$ from Corollary \ref{C:cumulants-vals} in random networks tuned to criticality with non-linearities either from the universality class of ReLU or of tanh. We will see that, at large network depth $L$ and leading order in $1/n$, these cumulants depend only on the effective network depth $L/n$. We formalize our belief that this is a universal phenomenon in Conjecture \ref{C:resummation}. We then solve in \S \ref{S:evgp-results} a simplified version of the so-called exploding and vanishing gradient problem (EVGP), which concerns the empirical variance of parameter gradients in a random neural network. More precisely, for non-linearities in the $K_*=0$ universality class in prove (see Theorem \ref{T:evgp}) that to leading order in $1/n$ the empirical variance over weights in the first layer of network gradients grows linearity in the effective network depth $L/n$. These results depend on Theorem \ref{T:derivs-cumulants}, which gives for non-linearities in the $K_*=0$ universality class asymptotics at large $\ell$ for joint fourth cumulants between the values of partial derivatives of the output of a random neural network. 


\subsection{Cumulants of Random Neural Networks at Finite Width}\label{S:precise-results}
Since in the infinite $n$ limit, the field $z_{\alpha}^{_{(L+1)}}$ is Gaussian (see Theorem \ref{T:iwl}), it is natural to measure perturbations around this regime by considering the behavior of the cumulants of $z_{\alpha}^{_{(L+1)}}$ and its derivatives. Let us therefore agree that, given random variables $X_1,\ldots, X_k$ with finite moments defined on the same probability space, we will denote their mixed cumulant by 
\begin{equation}\label{E:cumulant-def}
\kappa\lr{X_1,\ldots, X_k}:= i^k\frac{\partial^k}{\partial t_1\cdots \partial t_k}\bigg|_{t=0}\log \E{\exp\left[-i(t_1X_1+\cdots + t_kX_k)\right]}.    
\end{equation}
Thus, for example, $\kappa(X_1)=\E{X_1}$ and $\kappa(X_1,X_2)=\Cov(X_1,X_2).$ We refer the reader to \S \ref{S:cumulant-background} for background on cumulants. Our first result, Theorem \ref{T:cumulants}, gives estimates on the order in $1/n$ of the cumulants of $z_{\alpha}^{_{(L+1)}}$ and its derivatives. To state it, we fix a finite collection
\[
\set{x_\alpha,\, \alpha\in \mA}\subseteq \R^{n_0},
\]
of $\abs{\mA}$ distinct network inputs. Moreover, we fix a collection of $p$ directional derivatives:
\begin{equation}\label{E:der-deriv-def}
D=\lr{d_1,\ldots, d_p},\qquad d_j:=\nabla_{v_j}=\sum_{i=1}^{n_0} v_{ij}\partial_{x_i}.    
\end{equation}
and for any multi-index $J=\lr{j_1,\ldots,j_{p}}\in \N^{p}$ set
\[
 D_\alpha^{J}:= d_1^{j_1}\cdots d_m^{j_{p}}\bigg|_{x=x_\alpha}
\]
for the corresponding differential operator of order $\abs{J}:=j_1+\cdots + j_p.$






\begin{theorem}[order of magnitude for cumulants of random neural networks]\label{T:cumulants}
Fix $r,L\geq 1$ and suppose that $\sigma:\R\gives \R$ satisfies Assumption \eqref{A:sigma-prop} with this value of $r$. Suppose further that one of the following two conditions holds:
\begin{itemize}
\item $\sigma$ is smooth
\item the limiting covariance matrix
\begin{equation}\label{E:pd-cov}
\lr{\lim_{n\gives \infty}\Cov\lr{D_{\alpha_1}^{J_1}z_{1;\alpha_1}^{(\ell)}, D_{\alpha_2}^{J_2}z_{1;\alpha_2}^{(\ell)}}}_{\substack{\abs{J_1},\abs{J_2}\leq r\\ \alpha_1,\alpha_2\in \mA}}
\end{equation}
of derivatives of order at most $r$ in the directional derivatives $d_1,\ldots,d_p$ of the scalar field $z_{1;\alpha}^{(\ell)}$ is strictly positive definite in the infinite width limit for all $\ell\leq L$.
\end{itemize}
Then, for each $k,\ell\geq 1$, as $n\gives \infty$
\begin{equation}\label{E:kappa-est}
\kappa\lr{D_{\alpha_1}^{J_1}z_{i_1;\alpha_1}^{(\ell)},\ldots, D_{\alpha_k}^{J_k}z_{i_k;\alpha_{k}}^{(\ell)}}= \begin{cases}0,&\quad k\text{ odd}\\
O(n^{-\frac{k}{2}+1}),&\quad k\text{ even}\end{cases},
\end{equation}
where the implicit constant in the error term depends on $k$, the inputs $x_{\alpha_1},\ldots, x_{\alpha_k},$ the multi-indices $J_1,\ldots, J_k$, the weight and bias variances $C_b,C_W,$ the non-linearity $\sigma$, and the layer index $\ell$.
\end{theorem}






We prove Theorem \ref{T:cumulants} in \S \ref{S:cumulants-pf}. At a physics level of rigor, Theorem \ref{T:cumulants} with $k=4$ and no derivatives was already derived in the breakthrough work of Yaida \cite{yaida2019non}. In fact, Yaida's original article went much further: it obtained a recursive formula with respect to $\ell$ for the fourth cumulant $\kappa(z_{i_1;\alpha_1}^{_{(\ell)}},\ldots,z_{i_4;\alpha_4}^{_{(\ell)}})$ at layer $\ell$ in terms of the second and fourth cumulants at layer $\ell-1$. This is analogous to the recursion \eqref{E:K-rec} for the infinite width covariance $K_{\alpha_1\alpha_2}^{_{(\ell)}}$. This theme was then picked up and significantly expanded upon in the physics monograph \cite{roberts2022principles}, which computes, among other things, at order $1/n$ the leading corrections to the field $z_{\alpha}^{_{(\ell)}}$ and its derivatives with respect to both $x_\alpha$ and model parameters. We will reproduce some of these recursions and obtain new ones of a similar flavor below. 

Compared to this prior work the main novelty of Theorem \ref{T:cumulants} is two-fold. First, it gives sharp estimates in powers of $1/n$ for cumulants of all orders (the sharpness can already be seen when $\ell=2$). Second, it treats in a uniform way the cumulants for not only the values but also all derivatives of $z_{\alpha}^{_{(\ell)}}$. 

In order to put Theorem \ref{T:cumulants} into context, we take this opportunity to make an important remark. Namely, it is no accident that the order of magnitude $O(n^{-k+1})$ for the $2k$-th cumulant in Theorem \ref{T:cumulants} is the same as that of the $k$-th cumulant of an average of $n$ iid random variables. To see why, let us denote by $\mF^{(\ell)}$ the sigma algebra generated by the weights and biases in layers up to and including $\ell$. Since we've assumed weights and biases to be Gaussian and independent for different neurons, we find that conditional on $\mF^{(\ell)}$ the vectors 
\[
z_{i;\mA}^{(\ell+1)} :=\lr{z_{i;\alpha}^{(\ell+1)},\, \alpha\in \mA}
\]
are independent centered Gaussians with the covariance 
\begin{equation}\label{E:big-sig-def}
\Sigma_{\alpha\alpha'}^{(\ell)}:= \Cov\lr{z_{i;\alpha}^{(\ell+1)},z_{i;\alpha'}^{(\ell+1)}~|~\mF^{(\ell)}} = C_b+ \frac{C_W}{n_{\ell}}\sum_{j=1}^{n_{\ell}}\sigma\lr{z_{j;\alpha}^{(\ell)}} \sigma\lr{z_{j;\alpha'}^{(\ell)}}.
\end{equation}
Put another way, we have the following equality in distribution:
\begin{equation}\label{E:cond-gauss-vals}
\lr{z_{i;\alpha}^{(\ell+1)}}_{i=1}^{n_{\ell+1}}~\stackrel{d}{=}~ \lr{\lr{\Sigma^{ (\ell)}}^{1/2} Z_i}_{i=1}^{n_{\ell+1}},    
\end{equation}
where $ Z_i$ are i.i.d. standard Gaussians and are independent of the random covariance matrix
\[
\Sigma^{(\ell)} := \lr{\Sigma_{\alpha\alpha'}^{(\ell)}}_{\substack{\alpha,\alpha'\in \mA}}.
\]
Inspecting \eqref{E:big-sig-def} shows this conditional covariance has the structure of an average of $n_\ell$ identically distributed random matrices, each depending only on the pre-activations $z_{i;\mA}^{_{(\ell)}}$ of a single neuron at layer $\ell$. We will refer to such random variables as \textit{collective observables}. The key point is that while the $z_{i;\mA}^{_{(\ell)}}$ are not independent for different $i$ at finite width, we will show that they are sufficiently weakly correlated that the cumulants of $\Sigma_{\alpha\alpha'}^{_{(\ell)}}$ have the same order in $n$ as if they were exactly independent (see Theorem \ref{T:smooth-cumulants} and Lemma \ref{L:moment-bounds}). This reasoning applies equally well to derivatives of $z_{i;\alpha}^{_{(\ell+1)}}$ and explains the order of magnitude of the estimates in Theorem \ref{T:cumulants}.

Collective observables such as $\Sigma_{\alpha\alpha'}^{_{(\ell)}}$ are a convenient book-keeping device for studying the fully distribution of
\[
D^{\leq r}z_{i;\mA}^{(\ell+1)}:=\lr{D_{\alpha}^Jz_{i;\alpha}^{(\ell+1)},\quad \alpha \in \mA,\, J\in \N^{p},\, \abs{J}\leq r}.
\]
Indeed, the conditional Gaussian structure \eqref{E:cond-gauss-vals} means the cumulants of $D^{\leq r}z_{i;\mA}^{_{(\ell+1)}}$ are easily expressible in terms of those of $D_{\alpha}^{J}D_{\alpha'}^{J'}\Sigma_{\alpha\alpha'}^{_{(\ell)}}$. For example, using Wick's Theorem and the multi-linearity of cumulants reveals that the fourth cumulant 
\begin{align*}
\kappa\lr{D_{\alpha_1}^{J_1}z_{i_1;\alpha_1}^{(\ell+1)}, D_{\alpha_2}^{J_2} z_{i_2;\alpha_2}^{(\ell+1)}, D_{\alpha_3}^{J_3} z_{i_3;\alpha_3}^{(\ell+1)}, D_{\alpha_4}^{J_4} z_{i_4;\alpha_4}^{(\ell+1)}}
\end{align*}
equals
\[
\delta_{i_1i_2}\delta_{i_3i_4}\kappa_{(\alpha_1\alpha_2)(\alpha_3\alpha_4)}^{(J_1J_2)(J_3J_4),(\ell)}+ \delta_{i_1i_3}\delta_{i_2i_4}\kappa_{(\alpha_1\alpha_3)(\alpha_2\alpha_4)}^{(J_1J_3)(J_2J_4),(\ell)}+ \delta_{i_1i_4}\delta_{i_2i_3}\kappa_{(\alpha_1\alpha_4)(\alpha_2\alpha_3)}^{(J_1J_4)(J_2J_3),(\ell)},
\]
where we've abbreviated
\begin{equation}\label{E:fancy-cumulant-def}
\kappa_{(\alpha_1\alpha_1'),\ldots, (\alpha_k\alpha_k')}^{{(J_1J_1'),\ldots, (J_kJ_k'),(\ell)}}:=\kappa\lr{D_{\alpha_1}^{J_1}D_{\alpha_1'}^{J_1'}\Sigma_{\alpha_1\alpha_1'}^{(\ell)},\ldots, D_{\alpha_k}^{J_k}D_{\alpha_k'}^{J_k'}\Sigma_{\alpha_k\alpha_k'}^{(\ell)}}.    
\end{equation}
The general pattern is that odd mixed cumulants in the $z_\alpha^{(\ell)}$ and its derivatives vanish by symmetry and $2k$-th cumulants are finite sums of cumulants $\kappa_{(\alpha_1\alpha_1'),\ldots, (\alpha_k\alpha_k')}^{_{(J_1J_1'),\ldots, (J_kJ_k'),(\ell)}}$. In Corollary \ref{C:cumulants-vals} and Proposition \ref{P:k4-deriv-recs} below we will obtain a range of recursions for these cumulants and thereby implicitly for the cumulants of the original field $z_{\alpha}^{_{(\ell)}}$ as well. 

To obtain such recursions note that, at first glance, the estimate \eqref{E:kappa-est} only gives the order of magnitude in powers of $1/n$ of the cumulants of $\kappa(D_{\alpha_1}^{J_1}z_{i_1;\alpha_1}^{(\ell)},\ldots, D_{\alpha_k}^{J_k}z_{i_k;\alpha_{k}}^{_{(\ell)}})$ but does not seem to provide information about their structural dependence on the remaining model parameters $C_b,C_W,\sigma, \ell$. However, such information can be obtained by combining \eqref{E:kappa-est} with an additional argument. To state the result, let us agree to write
\begin{equation}\label{E:kappa-def}
    \kappa_{\alpha_1\alpha_2}^{(\ell)}:=\Cov\lr{z_{i;\alpha_1}^{(\ell)},z_{i;\alpha_2}^{(\ell)}},\qquad\qquad    \kappa_{\alpha_1\alpha_2}^{J_1J_2,(\ell)}:=D_{\alpha_1}^{J_1}D_{\alpha_2}^{J_2}\kappa_{\alpha_1\alpha_2}^{(\ell)}.
\end{equation}
Physicists might refer to $\kappa_{\alpha_1\alpha_2}^{_{(\ell)}}$ as a dressed two point function. Moreover, let us denote by $\bk{\cdot}_{\kappa^{(\ell)}}$ the expectation with respect to a collection of centered jointly Gaussian random vectors 
\[
D_{\mA}^{\leq r} z_{i} = \lr{D_{\alpha}^Jz_{i;\alpha},\,\alpha\in \mA,\, \abs{J}\leq r}
\]
with the same covariance
\[
\Cov\lr{D_{\alpha_1}^{J_1} z_{i_1;\alpha_1},\, D_{\alpha_2}^{J_2} z_{i_2;\alpha_2}}:= \Cov\lr{D_{\alpha_1}^{J_1} z_{i_1;\alpha_1}^{(\ell)},\, D_{\alpha_2}^{J_2} z_{i_2;\alpha_2}^{(\ell)}}=\delta_{i_1i_2} \kappa_{\alpha_1\alpha_2}^{J_1J_2,(\ell)}
\]
as the true vectors of derivatives $D_{\mA}^{\leq r} z_{i;\mA}^{_{(\ell)}}$ in each component separately but zero covariance for different $i$. 





\begin{theorem}[perturbative expansions of expectations of observables at finite width]\label{T:pert-exp}
Fix $r\geq 0$ and suppose that $f$ is both a continuous function and a tempered distribution. Then for any $q_*\geq 0$ we have
\begin{align}
\notag &  \E{f\lr{D^{\leq r}z_{1,\mA}^{(\ell+1)},\ldots, D^{\leq r}z_{m,\mA}^{(\ell+1)}}}\\ 
\notag &\quad  =\sum_{q=0}^{2q_*}\frac{(-1)^q}{2^qq!}  \bigg\langle\mathbb E \bigg[\bigg(\sum_{\substack{\abs{J},\abs{J'}\leq r\\ \alpha,\alpha'\in \mA}}\Delta_{\alpha\alpha'}^{JJ',(\ell)} \sum_{j=1}^m \partial_{D_{\alpha}^{J}z_{j;\alpha}}\partial_{D_{\alpha'}^{J'}z_{j;\alpha'}}\bigg)^q\bigg]
 f\lr{D_{\mA}^{\leq r}z_{1},\ldots, D_{\mA}^{\leq r}z_{m}}\bigg \rangle_{\kappa^{(\ell+1)}}\\
\label{E:pert-exp}  &\quad +O(n^{-q_*-1}),
\end{align}
where the sum is over multi-indices $J,J'\in \N^{p}$ of order at most $r$, we've set
\begin{equation}\label{E:Delta-def}
\Delta_{\alpha\alpha'}^{JJ',(\ell)}:=D_{\alpha}^JD_{\alpha'}^{J'}\Sigma_{\alpha\alpha'}^{(\ell)}-\E{D_{\alpha}^JD_{\alpha'}^{J'}\Sigma_{\alpha\alpha'}^{(\ell)}},    
\end{equation}
and the derivatives $\partial_{D_{\alpha}^{J}z_{j;\alpha}}$ are interpreted in the weak sense if $f$ is not differentiable.
\end{theorem}






We prove Theorem \ref{T:pert-exp} in \S \ref{S:cumulants-recs-pf} and give the main idea of the proof in \S \ref{S:proof-overview}. By substituting various polynomials in $z_\alpha^{_{(\ell+1)}}$ and its derivatives for $f$ into the perturbative expansion \eqref{E:pert-exp}, it is now straightforward in principle to deduce recursions for the cumulants $\kappa_{(\alpha_1\alpha_1'),\ldots, (\alpha_k\alpha_k')}^{_{(J_1J_1'),\ldots, (J_kJ_k'),(\ell+1)}}$ at layer $\ell+1$ in terms of objects of the same type at layer $\ell$. In particular, we have the following

\begin{corollary}[cumulants in random neural networks form a hierarchy to leading order in $1/n$]\label{cor:hier}
With the assumptions of Theorem \ref{T:cumulants},the mixed cumulant 
\[
\kappa\lr{D_{\alpha_1}^{J_1}z_{i_1;\alpha_1}^{(\ell+1)},\ldots, D_{\alpha_{2k}}^{J_{2k}}z_{i_{2k};\alpha_{2k}}^{(\ell+1)}}
\]
equals
\begin{equation}\label{eq:hier}
\sum_{\substack{j\leq k\\
J_i',\,i=1,\ldots, 2j\\ \abs{J_1'}+\cdots + \abs{J_{2j}'}\\
\quad \leq \abs{J_1}+\cdots + \abs{J_{2k}}}} C(J_i',K_{\alpha_{i}\alpha_{i'}}^{(\ell)} i,i'=1,\ldots, 2j)~\kappa\lr{D_{\alpha_1}^{J_1'}z_{1;\alpha_1}^{(\ell)},\ldots, D_{\alpha_{2j}}^{J_{2j}'}z_{2j;\alpha_{2j}}^{(\ell)}} + O\lr{n^{-k}},
\end{equation}
where the sum if over multi-indices $J_i'$, the constants $C(J_i',K_{\alpha_i\alpha_j}^{(\ell)} i,j=1,\ldots, 2k)$ depend only on the multi-indices $J_i'$ and the infinite width covariance $K^{(\ell)}$, while the implicit constant in the error term depends on $k$, the inputs $x_{\alpha_1},\ldots, x_{\alpha_k},$ the multi-indices $J_1,\ldots, J_k$, the weight and bias variances $C_b,C_W,$ the non-linearity $\sigma$, and the layer index $\ell$.

\end{corollary}

We do not know how to efficiently compute the coefficients $C$ in the recursion \eqref{eq:hier} for arbitrary $k$. Instead, we will compute them by hand for small values of $k$ in the two settings:
\begin{itemize}
\item We fix a single input $x_\alpha\in \R^{n_0}$ and obtain a recursion for the cumulants:
\begin{equation}\label{E:kalpha-def}
\kappa_{2k;\alpha}^{(\ell)}:= \kappa_{(\alpha\alpha)\cdots(\alpha\alpha)}^{(00)\cdots (00),(\ell)}=\frac{1}{(2k-1)!!}\kappa\bigg(\underbrace{z_{i;\alpha}^{(\ell+1)},\ldots, z_{i;\alpha}^{(\ell+1)}}_{2k\text{ times}}\bigg)    
\end{equation}
when $k=2,3,4$ (the pairs $(\alpha\alpha)$ and $(00)$ both appear $k$ times). See Corollary \ref{C:cumulants-vals}.
\item We again fix a single input $x_{\alpha}\in \R^{n_0}$ and consider any two partial derivatives 
\[
d_j:=\sum_{i=1}^{n_0} c_{i,j} \partial_{x_{j}}\bigg|_{x=x_{\alpha}},\qquad j=1,2.
\]
We obtain recursions for 
\begin{align}
\label{E:k4-deriv-def}\kappa_{(i_1i_1')(i_2i_2');\alpha}^{(\ell)}:=\kappa\lr{d_{i_1}d_{i_1'}\Delta_{\alpha\alpha}^{(\ell)}, d_{i_2}d_{i_2'}\Delta_{\alpha\alpha}^{(\ell)}}.
\end{align}
These cumulants are the building blocks for understanding the fourth mixed cumulants of $z_{i;\alpha}^{_{(\ell+1)}},d_1z_{i;\alpha}^{_{(\ell+1)}},d_2z_{i;\alpha}^{_{(\ell+1)}}$. For example, 
\[
\kappa\lr{d_{1}z_{i;\alpha}^{(\ell+1)},d_{1}z_{i;\alpha}^{(\ell+1)},d_{2}z_{i;\alpha}^{(\ell+1)},d_{2}z_{i;\alpha}^{(\ell+1)}} = \kappa_{(11)(22);\alpha}^{(\ell)}+2\kappa_{(12)(12);\alpha}^{(\ell)}.
\]
This will be done in the course of proving Theorem \ref{T:evgp}. We refer the interested reader to Proposition \ref{P:k4-deriv-recs} for the precise recursions and to Proposition \ref{P:4pt-deriv-recs-tanh} their solutions for non-linearities from the $K_*=0$ universality class (including $\tanh$).
\end{itemize}
In order to facilitate a compact form for the recursions described in first bullet point, let us write
\begin{equation}\label{E:T-def}
T_{i,j;\alpha}^{(\ell)}:=C_W^j\bk{\partial_{z}^i \left\{\lr{\sigma^2(z)-\bk{\sigma^2(z)}_{K_{\alpha\alpha}^{(\ell)}}}^j\right\}}_{K_{\alpha\alpha}^{(\ell)}},
\end{equation}
with the derivatives interpreted in the weak sense if $\sigma$ is not sufficiently smooth and where we remind the reader of our standing notation
\[
\bk{f(z)}_{K} = \int_{-\infty}^\infty f(zK^{1/2}) e^{-\frac{z^2}{2}}\frac{dz}{\sqrt{2\pi}},\qquad K\geq 0.
\]






\begin{corollary}\label{C:cumulants-vals}
Fix $r\geq 1$ and suppose that $\sigma:\R\gives \R$ satisfies Assumption \eqref{A:sigma-prop} with this value of $r$. Consider a depth $L$ random neural network with input dimension $n_0,$ hidden layer widths $n_1,\ldots, n_L$, output dimension $n_{L+1}$ and non-linearity $\sigma$. Fix $x_\alpha\in \R^{n_0}$ and define
\[
\chi_{||;\alpha}^{(\ell)}:=\frac{1}{2}T_{2,1;\alpha}^{(\ell)}=\frac{C_W}{2}\bk{\partial_z^2 \sigma(z)^2}_{K_{\alpha\alpha}^{(\ell)}},
\]
where the second derivative is interpreted in the weak sense if $\sigma$ is not twice differentiable. For each $\ell=1,\ldots, L$, in the notation of \eqref{E:kalpha-def}, the fourth cumulant satisfies
\begin{align}
\label{E:k4-rec}\kappa_{4;\alpha}^{(\ell+1)} &= \frac{T_{0,2;\alpha}^{(\ell)}}{n_{\ell}} +  \lr{\chi_{||;\alpha}^{(\ell)}}^2\kappa_{4,\alpha}^{(\ell)}  + O(n^{-2}).
\end{align}
Further, the $6$-th cumulant satisfies
\begin{align}
\label{E:k6-rec}\kappa_{6;\alpha}^{(\ell+1)} &= \frac{T_{0,3;\alpha}}{n_{\ell}^2}+\frac{3T_{2,2;\alpha}^{(\ell)}}{2n_{\ell}} \chi_{||;\alpha}^{(\ell)} \kappa_{4;\alpha}^{(\ell)}-\frac{3T_{4,1;\alpha}^{(\ell)}}{8}\lr{\chi_{||;\alpha}^{(\ell)}\kappa_{4;\alpha}^{(\ell)}}^2+\lr{\chi_{||;\alpha}^{(\ell)}}^3\kappa_{6;\alpha}^{(\ell)}+O(n^{-3}).
\end{align}
Finally, the $8$-th cumulant satisfies:
\begin{align}
\notag \kappa_{8;\alpha}^{(\ell+1)} &= \frac{1}{n_\ell^3}\lr{T_{0,4;\alpha}^{(\ell)}-3\lr{T_{0,2;\alpha}^{(\ell)}}^2}\\
&\notag +\frac{1}{n_\ell^2}\left[2T_{2,3;\alpha}^{(\ell)}\chi_{||;\alpha}^{(\ell)} - 12T_{0,2;\alpha}^{(\ell)}\lr{\chi_{||;\alpha}^{(\ell)}}^2+\frac{3}{2}\lr{T_{2,2;\alpha}^{(\ell)}}^2-\frac{3}{2}T_{4,1;\alpha}^{(\ell)}T_{0,2;\alpha}^{(\ell)}\right]\kappa_{4;\alpha}^{(\ell)}\\
\notag&-\frac{1}{n_\ell}\left[2T_{2,2;\alpha}^{(\ell)}T_{4,1;\alpha}^{(\ell)}\chi_{||;\alpha}^{(\ell)}-\frac{1}{2}T_{4,2;\alpha}^{(\ell)}\lr{\chi_{||;\alpha}^{(\ell)}}^2+\lr{\chi_{||;\alpha}^{(\ell)}}^4\right]\lr{\kappa_{4;\alpha}^{(\ell)}}^2\\
\notag&+\frac{1}{n_\ell}\left[5T_{0,2;\alpha}^{(\ell)}T_{4,1;\alpha}^{(\ell)}\chi_{||;\alpha}^{(\ell)}+12T_{2,2;\alpha}^{(\ell)}\lr{\chi_{||;\alpha}^{(\ell)}}^2\right]\kappa_{6;\alpha}^{(\ell)}\\
\notag &+\frac{3}{32} \lr{T_{4,1;\alpha}^{(\ell)}}^2 \lr{\chi_{||;\alpha}^{(\ell)}}^2 \lr{\kappa_{4;\alpha}^{(\ell)}}^3- \frac{1}{2}\lr{\chi_{||;\alpha}^{(\ell)}}^3 T_{4,1;\alpha}^{(\ell)} \kappa_{4;\alpha}^{(\ell)}\kappa_{6;\alpha}^{(\ell)}\\
\label{E:k8-rec}&+\lr{\chi_{||;\alpha}^{(\ell)}}^4 \kappa_{8;\alpha}^{(\ell)} + O(n^{-4}).
\end{align}
The initial condition for the recursions \eqref{E:k4-rec}-\eqref{E:k8-rec} is that $\kappa_{2k;\alpha}^{_{(1)}}=0$ for all $k\geq 2$. 
\end{corollary}
\begin{remark}
Note that for $k=2,3,4$, we therefore see that to leading order in $1/n$ the recursions for $\kappa_{2k;\alpha}^{_{(\ell+1)}}$ depends only on $\kappa_{2j;\alpha}^{_{(\ell)}}$ for $j\leq k$. This allows us to interpret \eqref{E:k4-rec} - \eqref{E:k8-rec} as a forming the start of hierarchy in powers of $1/n$ describing the cumulants of the output of a random neural network.
\end{remark}






\noindent Let us briefly compare Corollary \ref{C:cumulants-vals} to  results in prior work:
\begin{itemize}
    \item In the special case when $\sigma$ is $1-$homogeneous (i.e. is linear, ReLU, leaky ReLU, etc, see \eqref{E:1-homog-def}), the full distribution of a neuron pre-activation $z_{i;\alpha}^{_{(\ell)}}$ can be worked out in closed form. Namely, as we explain in \S \ref{S:relu-univ} and Appendix \ref{A:relu-exact}, is has the same distribution as a Gaussian with an independent random variance given by a product of independent weighted chi-squared random variables. This was first pointed out in \cite{hanin2018neural,hanin2020products} and described in the language of special functions (namely Meijer G functions) in \cite{zavatone2021exact}. For such non-linearities obtaining the recursions \eqref{E:k4-rec}-\eqref{E:k8-rec} is not new. 
    \item The breakthrough work of Yaida \cite{yaida2019non} was the first to obtain, at a physics level of rigor, the recursion \eqref{E:k4-rec} and probe its solutions at large $\ell$. 
    \item The ideas of Yaida in \cite{yaida2019non} then seeded the development in the monograph \cite{roberts2022principles} of Roberts, Yaida, and the author a much richer analysis, producing at a physics level of rigor many recursions similar in flavor to \eqref{E:k4-rec}-\eqref{E:k8-rec} that describe the the behavior of objects such as network derivatives at the start of training, the NTK at the start of training, and even the change in the NTK and the resulting output of a \textit{trained} network. Many of these results go far beyond what we are currently capable of doing mathematically. 
    \item The analysis in \cite{roberts2022principles} never required studying cumulants $\kappa_{2k;\alpha}^{_{(\ell)}}$ for $k\geq 3$, and while the techniques developed there can certainly be used to obtain the recursions \eqref{E:k6-rec} and \eqref{E:k8-rec} we take a rather different approach in this article that produces such recursions more directly. 
\end{itemize}
 The functional $\chi_{||;\alpha}^{_{(\ell)}}$ plays a fundamental role in the recursive description of random neural networks 
supplied by Corollary \ref{C:cumulants-vals}, whose proof is in \S \ref{S:cumulants-recs-pf}. In the following section we explain a principled procedure, called \textit{tuning to criticality}, that reveals its origin (as well as that of a similar object we denote $\chi_{\perp;\alpha}^{_{(\ell)}}$) and explains how to choose $C_b,C_W$ so that these functionals are approximately equal to $1$ at large $\ell$. As we will see, such a choice will ensure that the recursions in Corollary \ref{C:cumulants-vals} and their infinite width counterpart \eqref{E:K-rec} have well-behaved solutions at large $\ell$. We will then return in \S \ref{S:relu-univ} and \S \ref{S:tanh-univ} to solving the recursions from Corollary \ref{C:cumulants-vals} in random networks tuned to criticality (see \eqref{E:1-homog-limit} and Corollary \ref{C:tanh-cumulants}).








\section{Criticality and Universality in Random Neural Networks at Large Depth}\label{S:crit-univ}
In the previous section we presented two kinds of results about the structure of random neural networks at large but finite width. The first, Theorem \ref{T:cumulants}, concerned the order of magnitude for cumulants of the output of such a random network and its derivatives. The second, Theorem \ref{T:pert-exp} and Corollary \ref{C:cumulants-vals}, spelled out recursions with respect to the layer index $\ell$ that describe, to leading order in $1/n$, network cumulants at layer $\ell+1$ in terms of those at layer $\ell$. Our purpose in forthcoming sections is to analyze these recursions at large $\ell$ and to apply this analysis to obtain results about the structure of gradients in deep fully connected networks. Before doing this, we must take a step back and ask: for which $\sigma, C_b,C_W$ are the recursions \eqref{E:K-rec} describing the infinite width covariance $K^{(\ell)}$ well-behaved at large $\ell$? 

In section \S \ref{S:crit}, we recall a more or less canonical answer to this question whose roots are in the early articles \cite{poole2016exponential,schoenholz2016deep} and that was recently spelled out in the generality presented here in \cite{roberts2022principles}. This procedure, called tuning to criticality, prescribes combinations of $\sigma, C_b,C_W$ for which $K^{(\ell)}$ is indeed well-behaved at large $\ell$. As we shall see below, the term criticality is meant to be evocative of it's use in the analysis of 2d systems in statistical mechanics in that tuning to criticality consists of choosing $C_b,C_W$ so that the infinite width covariance function $K^{{(\ell)}}$ is as close to constant as a function $\ell$ as possible. 

At a high level, there are two reasons to ask that $K^{(\ell)}$ be slowly varying as a function of $\ell$. First, it arguably only makes sense to study perturbative corrections in $1/n$ recursively in $\ell$ if the limiting $n\gives \infty$ covariance structure does not change too rapidly between consecutive layers. Second, and perhaps more importantly, as explained and thoroughly validated in \cite{park2019effect,schoenholz2016deep}, deep fully connected networks (without residual connections \cite{he2015delving}, batch normalization \cite{ioffe2017batch}, etc) are numerically stable enough for gradient-based optimization to succeed only if they are tuned to criticality.

We discuss in \S \ref{S:univ} how considerations underlying criticality naturally give rise to a notion of universality classes for random neural networks. Even the correct definition of universality is still not fully understood. Unlike in random matrix theory, universality for random neural networks depends not on the statistics of the individual weights and biases (though this is also an interesting direction to consider e.g. \cite{hanin2021deep}) but rather on the effect of the non-linearity $\sigma$ on the behavior of the infinite width covariance $K^{(\ell)}$ at large values of the depth $\ell$. 

Before giving the details, we take this opportunity to emphasize, as we have elsewhere, that the definitions of criticality and universality, the approach to solving the recursions for $\kappa_{2k;\alpha}^{_{(\ell)}}$ from Corollary \ref{C:cumulants-vals}, and the resulting lessons learned about the role of the effective network depth $L/n$ closely follow the ideas developed in the monograph \cite{roberts2022principles}. Though we pursue them in a somewhat different way, the author would nonetheless like to acknowledge that his co-authors Dan Roberts and Sho Yaida in the book deserve significant credit.






\subsection{Tuning to Criticality}\label{S:crit} 
As originally explained in \cite{poole2016exponential,schoenholz2016deep} and recently spelled out in a definitive way in \cite{roberts2022principles}, tuning a neural network to criticality means seeking choices of $(C_b,C_W)$ that lead to critical fixed points of the form $(K_*,K_*,K_*)$ for the recursion \eqref{E:K-rec}, viewed as a dynamical system describing  $(K_{\alpha\alpha}^{_{(\ell)}},K_{\beta\beta}^{_{(\ell)}},K_{\alpha\beta}^{_{(\ell)}})$ with time parameter $\ell$. Specifically, criticality requires
\begin{align}
\label{E:K*-def}\tag{$*$}\exists K_*\geq 0\qquad\text{s.t.}\qquad &K_* = C_b + C_W \bk{\sigma^2(z)}_{K_*}\\
\label{E:para-cond}\tag{$||$}\forall \ell\geq 1\qquad &\frac{\partial K_{\alpha\alpha}^{(\ell)}}{\partial K_{\alpha\alpha}^{(1)}}\bigg|_{K_{\alpha\alpha}^{(1)}=K_*}=1\\
\label{E:perp-cond}\tag{$\perp$}\forall \ell\geq 1\qquad &\frac{\partial K_{\alpha\beta}^{(\ell)}}{\partial K_{\alpha\beta}^{(1)}}\bigg|_{K_{\alpha\alpha}^{(1)}= K_{\alpha\alpha}^{(1)} = K_{\alpha\beta}^{(1)} =K_*}=1,
\end{align}
where
\[
K_{\alpha\alpha}^{(1)} = C_b+C_WK_{\alpha\beta}^{(0)},\qquad K_{\alpha\beta}^{(0)}:=\frac{1}{n_0}\sum_{j=1}^{n_0}x_{j;\alpha}x_{j;\beta},\qquad x_\alpha,x_\beta \in \R^{n_0}.
\]
The intuitive meaning of these conditions  is as follows. Due to Theorem \ref{T:iwl}, the first guarantees the existence of a fixed point $K_*$ for of the recursion
\begin{equation}\label{E:K-rec-diag}
K_{\alpha\alpha}^{(\ell+1)}= C_b+C_W\bk{\sigma^2(z)}_{K_{\alpha\alpha}^{(\ell)}}
\end{equation}
of the infinite width variance. In particular, \eqref{E:K*-def} implies
\[
K_{\alpha\alpha}^{(1)}=C_b+\frac{C_W}{n_0}\norm{x_\alpha}^2=K_*\qquad \Longrightarrow \qquad K_{\alpha\alpha}^{(\ell)} =\lim_{n\gives \infty}\Var\left[z_{i;\alpha}^{(\ell)}\right]= K_*\quad \forall \ell\geq 1.
\]
Thus, if a network is tuned to criticality, there is a critical radius
\[
K_{\text{crit}}^2:=n_0C_W^{-1}(K_*-C_b)
\]
such that for inputs $x_\alpha$ on the sphere of radius $K_{\text{crit}}$ the variance of $z_{i;\alpha}^{_{(\ell)}}$ is independent of $\ell$ in the infinite width limit. In non-critical networks, we expect this variance to either grow or decay exponentially in $\ell$, leading to numerical instabilities. The second condition \eqref{E:para-cond} considers the infinite width limit of the variance of $z_{i;\alpha}^{_{(\ell)}}$ for an input $x_\alpha$ for which  $K_{\alpha\alpha}^{_{(1)}}$ is close to $K_*$. Specifically, condition \eqref{E:para-cond} requires for all $\ell \geq 1$ that 
\[
K_{\alpha\alpha}^{(1)}=\Var[z_{i;\alpha}^{(1)}] = K_* + \delta K \qquad \Longrightarrow \qquad K_{\alpha\alpha}^{(\ell)}=\lim_{n\gives \infty}\Var\left[z_{i;\alpha}^{(\ell)}\right]= K_*+\delta K + o(\delta K).
\]
This guarantees that the fixed point $K_*$ of the recursion \eqref{E:K-rec-diag} is critical and hence that for inputs near the sphere of radius $K_{\text{crit}}$ the variance of the resulting pre-activations $z_{i;\alpha}^{_{(\ell)}}$ is approximately constant in $\ell$ at large $n$. The final condition \eqref{E:perp-cond} considers instead the covariance between two inputs on the sphere of radius $K_{\text{crit}}$. Namely, given two nearby network inputs $x_{\alpha},x_\beta\in \R^{n_0}$ with
\[
K_{\alpha\alpha}^{(1)} = K_{\beta\beta}^{(1)}= K_*, \qquad K_{\alpha\beta}^{(1)} = C_b+\frac{C_W}{n_0}\sum_{j=1}^{n_0}x_{j;\alpha}x_{j;\beta}= K_* - \delta K,
\]
the third condition asks that
\[
K_{\alpha\beta}^{(\ell)}=\lim_{n\gives \infty}\Cov\lr{z_{i;\alpha}^{(\ell)},z_{i;\beta}^{(\ell)}}= K_*-\delta K + o(\delta K),\quad \forall \ell.
\]
This ensures that the covariance between pre-activations $z_{i;\alpha}^{_{(\ell)}}$ and $z_{i;\beta}^{_{(\ell)}}$ corresponding to two nearby inputs on the $K_{\text{crit}}$-sphere are approximately independent of $\ell$ at large $n$. A simple computation directly from the recursion \eqref{E:K-rec} shows that 
\begin{align}
\label{E:chi-perp-def}\chi_{||}(K)&:= \frac{\partial K_{\alpha\alpha}^{(\ell+1)}}{\partial K_{\alpha\alpha}^{(\ell)}} \bigg|_{K_{\alpha\alpha}^{(\ell)}= K} = \frac{C_W}{2}\bk{\partial_z^2 (\sigma^2(z))}_{K}\\
\label{E:chi-para-def} \chi_{\perp}(K)&:=\frac{\partial K_{\alpha\beta}^{(\ell+1)}}{\partial K_{\alpha\beta}^{(\ell)}}\bigg|_{K_{\alpha\alpha}^{(\ell)}= K_{\beta\beta}^{(\ell)}= K_{\alpha\beta}^{(\ell)}= K}= C_W\bk{(\partial_z\sigma(z))^2}_{K}.
\end{align}
Hence, all together, tuning to criticality requires
\begin{equation}\label{E:criticality}
\boxed{K_*\geq 0\text{ s.t. } K_* = C_b + C_W \bk{\sigma^2(z)}_{K_*}\qquad\text{and}\qquad \chi_{||}(K_*)=\chi_{\perp}(K_*)=1.}
\end{equation}






\subsection{Universality Classes of Random Neural Networks: Two Examples In Search of a General Definition}\label{S:univ}
We now turn to discussing the notion of universality classes for random neural networks. To start, recall from Theorem \ref{T:iwl} that the behavior at large depth $\ell$ of random fully connected neural networks at infinite width is completely specified by the asymptotics of the limiting covariance function $K^{{(\ell)}}$. Observe, moreover, that the coefficients in the recursions for $k=2,3,4$ of the cumulants $\kappa_{2k;\alpha}^{_{(\ell+1)}}$ from Corollary \ref{C:cumulants-vals}, which by Theorem \ref{T:cumulants} determine the behavior of random neural networks at finite width to the first four orders in $1/n$, are completely determined by $\sigma$, the infinite width covariance $K^{(\ell)}$, and cumulants $\kappa_{2j;\alpha}^{_{(\ell)}},\, j\leq k$. It is therefore in terms of the large $\ell$ behavior of $K^{{(\ell)}}$ that we should hope to define universality classes of random neural networks at large depth. At present it is not clear what the correct general definition of such a universality class should be. We content ourselves instead with studying two important classes of examples.

\subsubsection{The Universality Class of ReLU}\label{S:relu-univ}
The most popular non-linearities used in practice are positively homogeneous of degree $1$, i.e. have the form
\begin{equation}\label{E:1-homog-def}
	\sigma(t)=(a_-{\bf 1}_{\set{t<0}} + a_+{\bf 1}_{\set{t>0}})t,\qquad a_-,a_+\in \R,\quad a_-\neq a_+,\quad a_-^2+a_+^2\neq 0.
\end{equation}
Such non-linearities include the ReLU ($a_-=0, a_+=1$) and the leaky ReLU ($a_-\in (0,1),a_+=1$). A direct computation, left to the reader, shows that criticality is achieved if and only if
\[
K_*\geq 0 \text{ is arbitrary}\quad \text{and}\quad C_b=0,\, C_W= \frac{2}{a_+^2+a_-^2}.
\]
Thus, the first property of the ReLU universality class is that setting $(C_b,C_W)=(0,2/(a_+^2+a_-^2))$ allows all non-negative $K_*$ to satisfy ($*$). In fact, at criticality, a simple symmetrization argument shows that the variance of neuron pre-activations is preserved exactly \textit{even at finite width}
\begin{equation}\label{E:fin-width-var}
\Var\left[z_{i;\alpha}^{(\ell)}\right] = \Var\left[z_{i;\alpha}^{(1)}\right]= \frac{C_W}{n_0}\norm{x_\alpha}^2\qquad \forall \ell,n_0,\ldots, n_\ell \geq 1,\, x_\alpha\in \R^{n_0}    
\end{equation}
and, relatedly, that we have
\[
\chi_{||;\alpha}^{(\ell)}:=\chi_{||}(K_{\alpha\alpha}^{(\ell)})=1=\chi_{\perp;}(K_{\alpha\alpha}^{(\ell)})=: \chi_{\perp;\alpha}^{(\ell)},\qquad \forall \ell\geq 1,\, x_\alpha\in \R^{n_0}.
\]
The remarkable property \eqref{E:fin-width-var} is much stronger than the criticality condition $(*)$, which requires only that this condition holds for \textit{some} value of $n_0^{-1}\norm{x_\alpha}^2$ and only in the limit when $n\gives \infty$. It implies that the cumulant recursions from Corollary \ref{C:cumulants-vals} for $1-$homogeneous non-linearities have constant coefficients and are therefore particularly simple to solve. For instance, we find at leading order in $1/n$
\begin{align*}
\kappa_{4;\alpha}^{(\ell+1)} &= \frac{C_W^2}{n_{\ell}}\left[\bk{\sigma(z)^4}_{K_{\alpha\alpha}^{(\ell)}}-\bk{\sigma(z)^2}_{K_{\alpha\alpha}^{(\ell)}}^2\right] +  \lr{\chi_{||;\alpha}^{(\ell)}}^2\kappa_{4;\alpha}^{(\ell)}\\
&= \lr{\frac{2}{(a_+^2+a_-^2)n_0}\norm{x_\alpha}^2}^2\lr{6\frac{a_+^4+a_-^4}{(a_+^2+a_-^2)^2}-1}\sum_{\ell'=1}^\ell \frac{1}{n_{\ell'}},    
\end{align*}
which shows that while $\kappa_{4,\alpha}^{_{(\ell)}}$ is suppressed by one power of $1/n$ relative to the infinite width variance $K_{\alpha\alpha}^{_{(\ell)}}$, it also grows one order faster in $\ell$. This illustrates an important and general theme: depth amplifies finite width effects. It is the effective depth $\ell/n$ of neurons at layer $\ell$ that measures the distance to the infinite width Gaussian regime.

Moreover, in the special setting of $1$-homogeneous non-linearities there is a simple method for obtaining the full distribution of the pre-activation vector $z_{\alpha}^{_{(\ell)}}$ at a single input at any finite values of $n_0,\ldots, n_\ell$. This was first pointed out in \cite{hanin2018neural,hanin2020products,zavatone2021exact} and is briefly reviewed in Appendix \ref{A:relu-exact}. A key takeaway is that if we take the hidden layer widths $n_1=\cdots=n_L=n$, then we have following convergence in distribution to product of independent normal and log-normal random variables:
\begin{equation}\label{E:1-homog-limit}
\lim_{\substack{n,L\gives \infty \\ L/n \gives \xi\in [0,\infty)}} z_{i;\alpha}^{(L)}\quad \stackrel{d}{=}\quad \lr{\frac{2\norm{x_\alpha}^2}{(a_+^2+a_-^2)n_0}}^{1/2} Z_1  \exp\left[-\mu(\xi,a_+,a_-)+\sigma(\xi,a_+,a_-)Z_2\right],    
\end{equation}
where
\[
\mu(\xi,a_+,a_-) = \sigma^2(\xi,a_+,a_-) := \frac{\xi}{4}\lr{6\frac{a_+^4+a_-^4}{(a_+^2+a_+^2)^2}-1},\quad Z_1,Z_2\sim \mN(0,1)\text{ iid}.
\]
The convergence \eqref{E:1-homog-limit} reveals that for a fixed input the distribution of the output of a random with $1-$homogeneous non-linearities at large depth and width depends in a simple way on the limiting effective network depth $\xi$. This bolsters the claim that they are all part of the same universality class. It also means that increasing the network depth $L$ drives it away from the infinite width Gaussian behavior observed at $\xi=0$ and that the outputs of \textit{deep and wide} networks are not well-approximated by a Gaussian at all, unless $\xi$ is infinitesimal, in which case the log-normal term $\exp\left[-\mu(\xi,a_+,a_-)+\sigma(\xi,a_+,a_-)Z_2\right]$ is negligible. 

Prior work \cite{hanin2019finite,hanin2020products,hanin2018start} of the author shows that when $\sigma=\mathrm{ReLU}$ (or any other $1$-homogeneous non-linearity), the distribution at large $n,L$ of not only the network output $z_{i;\alpha}^{_{(L+1)}}$ but also is derivatives with respect to inputs $x_\alpha$ and model parameters (e.g. weights and biases) depends only on the effective depth $L/n.$ We will return to this theme in our discussion of exploding and vanishing gradients (Theorem \ref{T:evgp}). Finally, we note that it has also been observed that log-normal random variables describe the structure of gradients in residual networks, even during/after training  \cite{li2021future}.

To complete our discussion of the ReLU universality class, we make two final remarks. First, a direct computation (reviewed briefly in Proposition \ref{P:1homog-angle-prop} of Appendix \ref{A:relu-exact}) shows that at criticality for any non-zero inputs $x_{\alpha_1},x_{\alpha_2}\in \R^{n_0}$ with the same norm  we have 
\begin{equation}\label{E:1-homog-corr}
\lim_{n\gives \infty}\mathrm{Corr}\lr{z_{i;\alpha_1}^{(\ell)},z_{i;\alpha_2}^{(\ell)}} =1 - \frac{2(a_+-a_-)^2}{3\pi(a_+^2+a_-^2)}\ell^{-2}(1+o(1)).
\end{equation}
The power law exponent $2$ that appears in this estimate is common to all $1-$homogeneous non-linearities and is another reason to believe they fall into the same universality class. In contrast, this exponent equals one for non-linearities in the $K_*=0$ universality class presented below. The estimate \eqref{E:1-homog-corr} suggests that in order to define a double scaling limit $n,L\gives \infty$ and $L/n\gives \xi$ in which the entire field $x_\alpha\mapsto z_{\alpha}^{_{(L+1)}}$ is non-degenerate (rather than just its value at a single input) one must rescale distances in the input space to prevent the collapse of correlations otherwise guaranteed in \eqref{E:1-homog-corr}. We leave this as an interesting direction for future work.

\subsubsection{The Universality Class of Hyperbolic Tangent}\label{S:tanh-univ}
The second class of non-linearities we study is what \cite{roberts2022principles} termed the $K_*=0$ universality class, which we take to mean non-linearities $\sigma$ such that 
\begin{itemize}
    \item $\sigma$ is a smooth, odd function satisfying Assumption \ref{A:sigma-prop}.
    \item $\sigma$ satisfies
    \begin{equation}\label{E:sigma-taylor}
    \sigma_1\sigma_3 <0,\qquad \sigma_j:=\frac{1}{j!}\frac{d^j}{dt^j}\bigg|_{t=0} \sigma(t).    
    \end{equation}
    \item $K_*=0$ is the unique fixed point of  equation \eqref{E:K*-def}.
    \item At criticality, for every non-zero network input $x_{\alpha}\in \R^{n_0}$ and each $\delta\in (0,1)$ we have as $L
    \gives \infty$ that
    \begin{equation}\label{E:Kstar0-decay}
    K_{\alpha\alpha}^{(L)} =\frac{1}{aL}\lr{1+O(L^{-1+\delta})},    
    \end{equation}
    where the implicit constant depends on $\delta$ and $x_\alpha$ and we've set
    \[
    a:=-6\frac{\sigma_3}{\sigma_1}.
    \]
    This specific value of $a$, which is positive by \eqref{E:sigma-taylor}, is the only possible candidate for decay of the form \eqref{E:Kstar0-decay} that is consistent with the recursion \eqref{E:K-rec}.
\end{itemize}
Some remarks are in order. First, if $K_*=0$ is the unique fixed point for \eqref{E:K*-def}, then a simple computation shows that criticality is achieved if and only if 
\begin{equation}\label{E:Kstar0-def}
    K_*=0,\qquad C_b=0,\qquad C_W =\sigma_1^{-2}.
\end{equation}
Next, our definition of the $K_*=0$ universality class does not make apparent whether it is empty. As we will see in Proposition \ref{P:tanh-crit}, however, the $K_*=0$ universality class is in fact quite large and contains for example the hyperbolic tangent and more generally any non-linearity that is $\tanh$-like in the sense that is smooth with $\sigma_1\neq 0$, has the opposite sign as its second derivative 
\begin{equation*}
\text{for almost every } z,~ \mathrm{sgn}\lr{\sigma(z)\sigma''(z)}=-1,
\end{equation*} 
is sub-linear
\begin{equation*}
 \exists C>0\, \text{ s.t. }\forall z\in \R\, \abs{\sigma(z)}\leq \abs{\sigma_1z},
\end{equation*}
and is controlled by its first few non-zero Taylor series coefficients at $0$:
\begin{equation*}
\exists C\geq 0\text{ s.t. }\forall z\geq 0,\quad \sigma_1z +\sigma_3z^3\leq \sigma(z)\leq  \sigma_1z +\sigma_3z^3+ Cz^4.
\end{equation*}
Further, by definition, for the $K_*=0$ universality class, the infinite width variance $K_{\alpha\alpha}^{_{(\ell)}}$ of neuron pre-activations $z_{i;\alpha}^{(\ell)}$ is qualitatively different from that of $1-$homogeneous non-linearities. Indeed, $K_{\alpha\alpha}^{_{(L)}}$ depends on $L$, decaying polynomially to $0$. Moreover, at large $L$, the value of $K_{\alpha\alpha}^{_{(L)}}$ is independent of the initial condition $K_{\alpha\alpha}^{_{(0)}}$ to leading order in $L$. As a final remark let us point out that searching for non-linearities $\sigma$ so that $K_*=0$ at criticality is quite natural. Indeed, for any $\sigma$ that is twice differentiable, we have
\[
\chi_{||}(K) = \chi_{\perp}(K) + C_W\bk{\sigma(z)\sigma''(z)}_{K}
\]
Hence, if $K>0$, then 
\[
\chi_{||}(K) = 1,\, \chi_{\perp}(K)=1\qquad  \Longrightarrow \qquad \bk{\sigma(z)\sigma''(z)}_{K} = 0.
\]
But if $\sigma$ is a sigmoidal function such as $\tanh$, then $\sigma(z)\sigma''(z)<0$ for all $z\neq 0$. Hence, $\bk{\sigma(z)\sigma''(z)}_{K} = 0$ can only occur when $K=0$. 

As in the monograph \cite{roberts2022principles}, let us now probe the role of network depth by studying the large $L$ behavior of the cumulants $\kappa_{2k;\alpha}^{_{(L)}},\, k=2,3,4,$ in networks with non-linearities from the $K_*=0$ universality class tuned to criticality. Note that in \eqref{E:Kstar0-decay} the limiting behavior of the variance $K_{\alpha\alpha}^{_{(L)}}$ depends (mildly) on the non-linearity $\sigma$ in terms of is first few Taylor coefficients at $0$. As we are about to see, however, the behavior of the higher cumulants $\kappa_{2k;\alpha}^{_{(L)}},\, k=2,3,4,$ when normalized by the appropriate power of $K_{\alpha\alpha}^{_{(L)}}$, is independent of $\sigma$ at leading order in $n$ and $L$ and depends only on universal constants and the effective network depth $L/n$. 
\begin{corollary}\label{C:tanh-cumulants}
Suppose $\sigma$ is a non-linearity from the $K_*=0$ universality class, and consider a depth $L$ random neural network with input dimension $n_0$, output dimension $n_{L+1}$, hidden layer widths satisfying
\[
n_1,\ldots, n_L=n\gg 1,
\]
and non-linearity $\sigma$ that has been tuned to criticality as in \eqref{E:Kstar0-def}. Write $\xi=L/n$ and define the normalized cumulants
\[
\widehat{\kappa}_{2k;\alpha}^{(L)}:=\frac{\kappa_{2k;\alpha}^{(\ell)}}{(K_{2;\alpha}^{(\ell)})^k}.
\]
We have for $k=2,3,4$ that
\begin{align}
 \label{E:kappa2k-tanh}   \widehat{\kappa}_{2k;\alpha}^{(L)}&=   C_{2k}\xi^{k-1} \lr{1+O(L^{-1})} + O(n^{-k}),
\end{align}
where
\[
C_4 = \frac{2}{3},\qquad C_6 = \frac{28}{15},\qquad C_8 = \frac{8756}{315}
\]
are some positive universal constants. The implicit constants in error terms $O(L^{-1})$ depend on $\sigma,C_b,C_W$ and constants in $O(n^{-j}),\, j=2,3,4$ may depend in addition on $L$. 
\end{corollary}
\begin{remark}\label{R:tanh-cumulants}
Combining estimate \eqref{E:kappa2k-tanh} for $k=2$ with the definition \eqref{E:Kstar0-decay} of the $K_*=0$ universality class and  \eqref{E:kalpha-def} yields that in the setting of Corollary \ref{C:cumulants-vals} we have to first order in $1/n$ and to leading order in $1/L$ that
\begin{align*}
    \widehat{\Var}\left[\lr{z_{1;\alpha}^{(\ell+1)}}^2\right] &= 2(1+\xi),\qquad \mathrm{Corr}\lr{\lr{z_{1;\alpha}^{(\ell+1)}}^2,\, \lr{z_{2;\alpha}^{(\ell+1)}}^2}=\frac{2}{3}\xi,\qquad \xi:=\frac{L}{n},
\end{align*}
where $\widehat{\Var}[X]=\Var[X]/\E{X^2}$. Thus, both the fluctuations of a single neuron pre-activation and the correlation between different neurons is controlled to first order in $1/n,1/L$ by the effective network depth $\xi$. 
\end{remark}
We prove Corollary \ref{C:cumulants-vals} in \S \ref{S:tanh-rec-solutions}. The formula \eqref{E:kappa2k-tanh}  derived in \cite{yaida2019non} for $k=2$ at a physics level of rigor. Moreover, Corollary \ref{C:cumulants-vals} represents a partial analog of the convergence result \eqref{E:1-homog-limit}, which lends credence to the following 

\begin{conjecture}[Existence of Double Scaling Limit for Random Neural Networks]\label{C:resummation}
Consider a random depth $L$ neural network with input dimension $n_0$, hidden layer widths 
\[
 n_1,\ldots, n_L=n\gg 1,
\]
output dimension $n_{L+1}$ and non-linearity $\sigma$. Suppose further that this network is tuned to criticality in the sense that \eqref{E:criticality} is satisfied. Fix a non-zero network input $x_\alpha\in \R^{n_0}$ and write $\xi=L/n$. For each $k\geq 1$ there exists $C_{2k}>0$ depending on the universality class of $\sigma$ so that
\[
\frac{\kappa_{2k;\alpha}^{(L)}}{\lr{K_{2;\alpha}^{(L)}}^{k}} = C_{2k}\xi^{k-1}+O\lr{\xi^{k}}.
\]
Moreover, for each $\xi\in [0,\infty)$ there exists a probability distribution $\mathbb P_{\xi,\sigma}$ on $\R$, depending only on $\xi$ and $\sigma$, such that in the double scaling limit
\[
n,L\gives \infty,\qquad \frac{L}{n}\gives \xi,
\]
the random variable $z_{i;\alpha}^{(\ell)}$ converges in distribution to a random variable with law  $\mathbb P_{\xi,\sigma}$. 
\end{conjecture}





\subsection{Gradients in Random Neural Networks}\label{S:grad-results}

We presented in \S \ref{S:results} and \S \ref{S:crit-univ} a range of results about random fully connected neural networks. The high-level takeaway was that such networks are succinctly described when $n,L$ are large by the effective network depth $\xi=L/n$. We saw that when $\xi=0$ the network outputs are a Gaussian process (Theorem \ref{T:iwl}). When $\xi>0$, however, a different picture emerges. Namely, the higher order cumulants of the network output grow like powers of $\xi$ (see Corollary \ref{C:cumulants-vals} and Conjecture \ref{C:resummation}). In this section we continue to consider a random neural network $x_\alpha\mapsto z_\alpha^{_{(L+1)}}$ and study for any $\ell$ the partial derivatives 
\[
\partial_{x_{j;\alpha}} z_{i;\alpha}^{(\ell)},\qquad j=0,\ldots, n_0,
\]
where by definition when $j=0$ we set
\[
\partial_{x_{0;\alpha}} z_{i;\alpha}^{(\ell)}:=z_{i;\alpha}^{(\ell)}.
\]
In \S \ref{S:K-recs} - \S\ref{S:S2-recs} we obtain recursions with respect to $\ell$ for both the infinite width covariances
    \[
    K_{(ij)}^{(\ell)}:=\lim_{n\gives \infty }\Cov\lr{\frac{\partial z_{i;\alpha}^{(\ell+1)}}{\partial x_{i;\alpha}},\frac{\partial z_{i;\alpha}^{(\ell+1)}}{\partial x_{j;\alpha}}},\quad i,j=1,2.
    \]
and the fourth cumulants
    \begin{align*}
        \kappa_{(ii)(jj)}^{(\ell)}&:=(1-\frac{2}{3}\delta_{ij})\kappa\lr{\frac{\partial z_{i;\alpha}^{(\ell+1)}}{\partial x_{i;\alpha}},\frac{\partial z_{i;\alpha}^{(\ell+1)}}{\partial x_{i;\alpha}},\frac{\partial z_{i;\alpha}^{(\ell+1)}}{\partial x_{j;\alpha}},\frac{\partial z_{i;\alpha}^{(\ell+1)}}{\partial x_{j;\alpha}}},\qquad i,j=0,1,2.\\
        &=\kappa\lr{\frac{\partial z_{1;\alpha}^{(\ell+1)}}{\partial x_{i;\alpha}},\frac{\partial z_{1;\alpha}^{(\ell+1)}}{\partial x_{i;\alpha}},\frac{\partial z_{2;\alpha}^{(\ell+1)}}{\partial x_{j;\alpha}},\frac{\partial z_{2;\alpha}^{(\ell+1)}}{\partial x_{j;\alpha}}}.
    \end{align*}
These recursions are valid for any non-linearity and any choice of $C_b,C_W$. Then, in \S \ref{S:tanh-rec-solutions} we solve these recursions in the context of a random neural network with a non-linearity from the $K_*=0$ universality class tuned to criticality. We record several of these solutions in the following result. 
\begin{theorem}[Second and Fourth Joint Cumulants for Values and Derivatives of Random Neural Networks]\label{T:derivs-cumulants}
    Fix $r\geq 1$ and suppose that $\sigma:\R\gives\R$ satisfies Assumption \ref{A:sigma-prop} with this value of $r$. Fix also an $x_\alpha\neq 0 \in \R^{n_0}$. Consider a random neural network $z_\alpha^{_{(L+1)}}$ with input dimension $n_0$, output dimension $n_{L+1}$, hidden layer widths $n_1,\ldots, n_L=n$, and non-linearity $\sigma$ belonging to the $K_*=0$ universality class that has been tuned to criticality as in \eqref{E:Kstar0-def}. We have for any $\delta\in (0,1)$
    \begin{align}
    \label{E:K00-form}K_{(00)}^{(\ell+1)} &= \frac{1}{a\ell}+O_\delta(\ell^{-2+\delta})\\
    \label{E:K10-form}K_{(10)}^{(\ell+1)} &=\frac{C_We^{-2\gamma}x_{1;\alpha}}{n_0\ell^2}\lr{1+O(\ell^{-1})}\\
    \label{E:K11-form}K_{(11)}^{(\ell+1)} & = \frac{C_We^{-\gamma}}{n_0\ell}\lr{1+O(\ell^{-1})} + O\lr{\frac{(x_{1;\alpha})^2}{n_0^2\ell^3}}\\
   \label{E:K12-form} K_{(12)}^{(\ell+1)} &=O\lr{\frac{x_{1;\alpha}x_{2;\alpha}}{n_0^2\ell^3}},
\end{align}
where $\gamma$ is the Euler-Mascheroni constant and the implicit constant in the $O_\delta(\ell^{-2+\delta})$ error depends on $\delta$. Moreover, denoting by $\xi = L/n$ the effective network depth, we have
\begin{align}
 \label{E:k1100-formula}     \frac{\kappa_{(11)(00)}^{(\ell+1)}}{K_{(11)}^{(\ell)}K_{(00)}^{(\ell)}} &=-\frac{1}{3}\xi (1+O(\ell^{-1}))+O(n^{-2})\\
  \label{E:k1111-formula}    \frac{\kappa_{(11)(11)}^{(\ell+1)}}{(K_{(11)}^{(\ell)})^2} & = \frac{8}{3}\xi(1+O(\ell^{-1})) +O(n^{-2})\\
 \label{E:k1122-formula}       \frac{\kappa_{(11)(22)}^{(\ell+1)}}{K_{(11)}^{(\ell)}K_{(22)}^{(\ell)}}&=\frac{2}{3}\xi(1+O(\ell^{-1}))+O(n^{-2}) ,
\end{align}
where the implicit constants in the $O(\ell^{-1})$ error terms depend only on $x_\alpha,\sigma$ and the implicit constants in the $O(n^{-2})$ error terms may depend in addition on $\ell$.
\end{theorem}
Theorem \ref{T:derivs-cumulants} reveals several interesting phenomena. First, we find from \eqref{E:K00-form} and \eqref{E:K10-form} that at infinite width values and derivatives become uncorrelated at large $L$:
\[
\lim_{L\gives \infty}\lim_{n\gives \infty}\mathrm{Corr}\lr{z_{1;\alpha}^{(L+1)},\, \partial_{x_{1;\alpha}}z_{1;\alpha}^{(L+1)}} = 0.
\]
Moreover, at finite width, the effective network depth $\xi$ controls the distribution of gradients both in terms of their fluctuations at a single neuron and their correlations between neurons. For instance, to leading order in $n,L$:
\begin{align*}
\frac{\Var\left[\lr{\partial_{x_{1;\alpha}} z_{1;\alpha}^{(L+1)}}^2\right]}{\E{\lr{\partial_{x_{1;\alpha}} z_{1;\alpha}^{(L+1)}}^2}}&= \frac{8}{3}\xi,\qquad \mathrm{Corr}\lr{\lr{\frac{\partial z_{1;\alpha}^{(L+1)}}{\partial x_{1;\alpha}}}^2,\lr{\frac{\partial z_{1;\alpha}^{(L+1)}}{\partial x_{2;\alpha}}}^2}= \frac{2}{3}\xi.
\end{align*}
In the next section, we will apply such estimates to understand a simple version of the so-called exploding and vanishing gradient problem.

\subsection{Exploding and Vanishing Gradient Problem for First Layer Weights}\label{S:evgp-results} On its face, Theorem \ref{T:derivs-cumulants} concerns only derivatives of $z_{i;\alpha}^{_{(L+1)}}$ with respect to the network input $x_\alpha$. However, it is the derivatives of $z_{i;\alpha}^{_{(L+1)}}$ with respect to the model parameters (weights and biases) that are arguably more important. To explain this more precisely, consider a fully connected neural network $x_\alpha\mapsto z_{\alpha}^{_{(L+1)}}$ with input dimension $n_0$, output dimension $n_{L+1}$, and hidden layer widths $n_1,\ldots, n_L$. For any $m\in \N$ let us set
\[
[m]:=\set{1,\ldots, m}
\]
and denote by 
\[
\theta=\lr{\theta_\mu,\, \mu\in [\#\text{params}]}=\lr{W_{ij}^{(\ell)},b_i^{(\ell)},\, \ell\in [L+1],\, i\in [n_\ell],\, j\in [n_{\ell-1}]}
\]
the vector of its trainable parameters, where we've set
\[
\#\text{params} := \#\text{weights and biases}= \sum_{\ell=1}^{L+1} n_{\ell}\lr{1 +n_{\ell-1}}.
\]
As we briefly discussed in \S \ref{S:main-q} the typical use case for a network is to optimize its parameters $\theta$ by some variant of gradient descent
\begin{equation}\label{E:param-gd}
\theta_\mu(t+1) = \theta_\mu(t) - \eta  \partial_\mu \mathcal L(\theta(t)),\qquad \mu\in [\#\text{params}]
\end{equation}
starting from a random initialization $\theta(0)$ drawn from the distribution \eqref{E:Wb-def}. The goal of this procedure is to minimize a loss 
\[
\mL(\theta)= \mL(z_{\mA_{\text{train}}}^{(L+1)}(\theta)),\qquad z_{\mA_{\text{train}}}^{(L+1)}:=\lr{z_\alpha^{(L+1)},\, \alpha\in \mA_{\text{train}}},
\] 
which often depends only on the network output and measures for each $\theta$ how well the resulting neural net function $x_\alpha\mapsto z_{\alpha}^{(L+1)}(\theta)$ matches given training data $\set{(x_\alpha,f(x_\alpha)),\, \alpha \in \mA_{\text{train}}}$ for a function $f:\R^{n_0}\gives \R^{n_{L+1}}$.  For example, we might have 
\[
\mL(\theta)=\frac{1}{\abs{\mA_{\text{train}}}} \sum_{\alpha\in \mA_{\text{train}}} \norm{f(x_\alpha)-z_{\alpha}^{(L+1)}}_2^2,
\]
though many different losses are used in practice. The parameter $\eta$ is called the learning rate, or step size. An important and well-known \cite{bengio1994learning,hanin2018neural,hochreiter1998vanishing} numerical stability issue that comes up in the course of using a first order method such as \eqref{E:param-gd} to optimize the parameters of a deep neural network is called the exploding and vanishing gradient problem (EVGP). Informally, the EVGP occurs when the update \eqref{E:param-gd} is numerically unstable:
\[
\mathrm{EVGP}\qquad \longleftrightarrow \qquad \text{for many parameters }\theta_\mu\quad \frac{\abs{\eta  \partial_\mu \mathcal L(\theta)}}{\abs{\theta_\mu}}\approx 0\text{ or }\infty.
\]
The presence of the EVGP causes optimization to break down since if $\abs{\partial_\mu\mL(\theta)}/\abs{\theta_\mu}\approx 0\text{ or }\infty$, then the relative change in the parameter $\theta_\mu$ is either too small to be useful or so large that it amounts to a large random step in the space of parameters and therefore is unlikely to decrease the loss. 

It has long been known that sufficiently deep fully connected neural networks are prone to suffer from the EVGP \cite{bengio1994learning,hochreiter1998vanishing,hochreiter2001gradient}. Indeed, many important practical innovations such as residual connections \cite{he2015delving} were invented in part to address such numerical issues, though by now they are seen as useful for many other reasons. To see why deep enough neural nets tend to have numerically unstable gradients note that since the loss only depends on the network outputs on a finite training set $z_{\mA_{\text{train}}}^{_{(L+1)}} $, we may use the chain rule to write
\[
\partial_{\mu}\mL(\theta) = \sum_{\alpha\in \mA_{\text{train}}} J_{\theta_\mu} z_{\alpha}^{(L+1)} J_{z_{\alpha^{(L+1)}}}\mL(z_{\mA_{\text{train}}}^{(L+1)}) ,
\]
where for any function $f$ we denote by $J_xf(x)$ its Jacobian with respect to $x.$ Moreover, suppose that $\theta_\mu$ is a weight or bias in layer $\ell_0$. Then, again by the chain rule, the Jacobian of the network output with respect to $\mu$ equals
\[
J_{\theta_\mu} z_{\alpha}^{(L+1)} = J_{\theta_\mu} z_{\alpha}^{(\ell_0)}J_{z_{\alpha}^{(\ell_0)}} z_{\alpha}^{(L+1)}.
\]
Both of these terms may have large fluctuations. Perhaps most importantly, the second term can be rewritten as a product
\[
 J_{z_{\alpha}^{(\ell_0)}} z_{\alpha}^{(L+1)}=\prod_{\ell=\ell_0}^{L} J_{z_{\alpha}^{(\ell)}} z_{\alpha}^{(\ell+1)},
\]
of $L-\ell_0+1$ layer-to-layer Jacobian matrices (which are random at the start of training), with each term having the following explicit representation: 
\[
J_{z_{\alpha}^{(\ell)}} z_{\alpha}^{(\ell+1)} = D^{(\ell+1)}W^{{(\ell+1)}},\qquad D^{(\ell+1)}:=\mathrm{Diag}\lr{\sigma'\lr{z_{i;\alpha}^{(\ell+1)}},\, i=1,\ldots, n_{\ell+1}}.
\]
Except in the simple case when $\sigma$ is the identity and we require $W^{{(\ell+1)}}$ to be orthogonal, the top singular value of each layer-to-layer Jacobian will tend to deviate from $1$, causing the norms of their products to either grow or decay exponentially with $L$. This poor conditioning for large matrix products is a key culprit behind the EVGP. 

Often, though not always, the gradients $\partial_\mu\mL$ have the largest fluctuations at the start of training. Thus, the EVGP is primarily a property of random neural networks, and the techniques in this article can be used to shed light on when it occurs. At the start of training the entries of the Jacobian $J_{z_{\alpha}^{(\ell_0)}} z_{\alpha}^{_{(L+1)}}$ are precisely the derivatives of a random neural network with depth $L-\ell_0+1$ with respect to its input. One can therefore hope to use Theorems \ref{T:cumulants} and \ref{T:pert-exp} to address the following more precise formulation of the EVGP:
\begin{align*}
\text{EVGP}\quad \longleftrightarrow \quad &\text{for a typical random draw of weights and biases the size of the}\\
&\text{relative Jacobians of the network output}\\
& \qquad\qquad \qquad \qquad \qquad |J_{\theta_\mu} z_{j;\alpha}^{(L+1)}|/\abs{\theta_\mu}\\
& \text{has a large variance over $\mu$ (i.e. over network parameters)}.
\end{align*}
This characterization of the EVGP says in essence that if we can choose only a single learning rate $\eta$ for a group of otherwise indistinguishable parameters, such as the weights in a single layer, then it make sense to set
\[
\eta= \lr{\text{average relative gradient } |J_{\theta_\mu} z_{j;\alpha}^{(L+1)}|/\abs{\theta_\mu}}^{-1}.
\]
The EVGP will then occur when the size of relative gradients have a large variance across parameters in a typical random draw of weights and biases, so that for some parameters our setting of $\eta$ is too small whereas for others it is too large. 

We will leave the study of the EVGP in this  general form as pertaining to future work and consider here a special case. Specifically, we compute in Theorem \ref{T:evgp} below the average of the empirical variance of the squared gradients $|J_{\theta_\mu}z_{j;\alpha}^{_{(L+1)}}|^2$ when $\theta_\mu=W_{ij}^{_{(1)}}$ varies over all weights in the first layer. By construction, the typical size of $\abs{\theta_\mu}$ is the same for all such $\mu$. The EVGP will thus occur for weights in the first layer if and only if the empirical fluctuations over $\mu$ of the raw Jacobians $|J_{{W_{ij}^{(1)}}} z_{j;\alpha}^{_{(L+1)}}|$ is large. To state the precise result we fix $x_{\alpha}\neq 0\in \R^{n_0}$ and $q\in [n_{L+1}]$. We then write
\[
\mathrm{Grad~Mean}^{(1)}:=\frac{1}{n_{0}n_{1}}\sum_{\substack{1\leq i\leq n_{1}\\ 1\leq j\leq n_{0}}} \lr{\partial_{W_{ij}^{(1)}} z_{q;\alpha}^{(L+1)}}^2
\]
for the empirical mean of the squared gradients of $z_{q;\alpha}^{(L+1)}$ with respect to the weights $W_{ij}^{_{(1)}}$ in the first layer and analogously 
\[
\mathrm{Grad~Var}^{(1)}:= \frac{1}{n_{0}n_{1}}\sum_{\substack{1\leq i\leq n_{1}\\ 1\leq j\leq n_{0}}} \lr{\partial_{W_{ij}^{(1)}} z_{q;\alpha}^{(L+1)}}^4 - \lr{\mathrm{Grad~Mean}^{(1)}}^2
\]
for the empirical variance of the squared gradients.

\begin{theorem}[EVGP for First Layer Weights in Fully Connected Networks]\label{T:evgp}
Let $x_{\alpha}\mapsto z_{\alpha}^{(L+1)}$ be a random neural network with input dimension $n_0$, output dimension $n_{L+1}$, hidden layer widths 
\[
n_1,\ldots, n_L=n\gg 1,
\]
and non-linearity $\sigma$. Suppose also that
\begin{itemize}
	\item $\sigma$ satisfies Assumption. \ref{A:sigma-prop}
	\item $\sigma$ belongs to the $K_*=0$ universality class in the sense described in \S \ref{S:crit} as well as \S \ref{S:tanh-univ}.
	\item $C_b,C_W$ are tuned to criticality in the sense described as in \eqref{E:Kstar0-def}.
\end{itemize}
Then, for any non-zero $x_\alpha\in \R^{n_0}$ we have 
\[
\frac{\E{\mathrm{Grad~Var}^{(1)}}}{\E{\mathrm{Grad~Mean}^{(1)}}^2} = C_{\sigma,\alpha} \bigg(1+\frac{8}{3} \xi +O(L^{-1})+O(n^{-1})+O_L(n^{-2})\bigg),
\]
where $\xi = \frac{L}{n}$ is the effective network depth and 
\[
C_{\sigma,\alpha} := 3\frac{\bk{(\sigma')^4}_{K_{\alpha\alpha}^{(1)}}}{\bk{(\sigma')^2}_{K_{\alpha\alpha}^{(1)}}^2}\frac{n_0^{-1}\norm{x_\alpha}_4^4}{(n_{0}^{-1} \norm{x_\alpha}_2^2)^2} - 1
\]
is a strictly positive constant depending only on $x_\alpha,\sigma$, the implicit constants in the error terms $O(n^{-1}), O(L^{-1})$ depend only on $x_\alpha, \sigma$ whereas the implicit constant in the error term $O_L(n^{-2})$ depends on $x_\alpha, \sigma, L$. 
\end{theorem}
\begin{remark}
Theorem \ref{T:evgp} shows that, at least to leading order in $1/n$, the exploding and vanishing gradient problem occurs for first layer weights in a critically tuned random fully connected with non-linearity from the $K_*=0$ universality class if and only if the effective network depth $L/n$ is large. 
\end{remark}
\noindent We prove Theorem \ref{T:evgp} in \S \ref{S:evgp-pf}. To put this result into the context, let us make several remarks. First, the analog of Theorem \ref{T:evgp} for $1-$homogeneous non-linearities was derived in prior work \cite{hanin2018neural,hanin2020products}. As we've alluded to before, the behavior of a random fully connected neural network with such non-linearities can be studied by more specialized methods that reveal not only the leading order dependence in $1/n$ but actually the full dependence on $L/n$ plus errors of size $L/n^2$ (see Appendix \ref{A:relu-exact}). The conclusion, just as for non-linearities from the $K_*=0$ universality class, is that the EVGP occurs if and only if $L/n$ is large. 

Second, Theorem \ref{T:evgp} is the first time the EVGP has been characterized mathematically in random fully connected networks for a broad class of non-linearities beyond those that are $1-$homogeneous. However, as throughout, the author would like to acknowledge the prior work \cite{roberts2022principles} of Roberts, Yaida, and the author which also studies the EVGP and shows, at a physics level of rigor, that the variance \textit{over random initialization} of a single squared gradient $(\partial_\mu z_{q;\alpha}^{_{(L+1)}})^2$ also scales like $L/n$ at large $L$. This is conceptually different from showing that in a typical initialization the empirical variance of $(\partial_\mu z_{q;\alpha}^{_{(L+1)}})^2$ over $\mu$ in the first layer is large, which is what Theorem \ref{T:evgp} establishes. Nonetheless, the conclusions are the essentially the same and the basic techniques used to derive the results are similar.






\section{Overview of Proofs}\label{S:proof-overview}
In this section, we present the essential idea for how we analyze a random fully connected neural network $x_\alpha\mapsto z_{\alpha}^{(L+1)}$ at finite width. Our approach is based on the following structural properties:
\begin{itemize}
    \item The sequence of fields $z_{\alpha}^{(\ell)}$ is a Markov Chain with respect to $\ell$. 
    \item Conditional on the sigma algebra $\mF^{(\ell)}$ defined by $z_{\alpha}^{(\ell)}$ is a Gaussian field with independent components $z_{i;\alpha}^{(\ell+1)}$. See Lemma \ref{L:cond-indep}.
    \item The variance of 
    each component $z_{i;\alpha}^{(\ell+1)}$ depends on $z_\alpha^{(\ell)}$ only through random variables of the form
    \[
    \mO_{f}^{(\ell)}:=\frac{1}{n_\ell} \sum_{j=1}^{n_\ell}f(z_{i;\alpha}^{(\ell)}),
    \]
    which we refer to as collective observables. See \eqref{E:cond-gauss-vals}.
    \item Collective observables are self-averaging to a similar extent as if the random variables $f(z_{i;\alpha}^{(\ell)})$ were independent in the sense that for any $q\geq 0$, we have
    \begin{equation}\label{E:Delta-concentration}
    \E{\lr{\mO_{f}^{(\ell)}-\E{\mO_{f}^{(\ell)}}}^q}=O_q\lr{n^{-\lceil \frac{q}{2}\rceil}}.
    \end{equation}
    See Theorem \ref{T:smooth-cumulants} and Lemma \ref{L:moment-bounds}.
\end{itemize}
Let us briefly explain, mostly dispensing with rigor, how these four ideas come together to obtain a recursive description of the distribution of the field $z_\alpha^{(\ell+1)}$ in terms of that of $z_\alpha^{(\ell)}$. To keep the notation to a minimum, we fix a network input $x_\alpha$ and focus on describing the joint distribution of $z_{i;\alpha}^{(\ell+1)},\, i=1,\ldots, m$. Extensions to multiple inputs and derivatives proceed along very similar lines. Denoting by $\xi=\lr{\xi_1,\ldots,\xi_m}$ dual variables, consider the characteristic function
\[
p^{(\ell+1)}(\xi):=\E{\exp\left[-i\sum_{i=1}^m \xi_i z_{i;\alpha}^{(\ell+1)}\right]}.
\]
Conditioning on $z_\alpha^{(\ell)}$ and using \eqref{E:cond-gauss-vals} allows us to write
\[
p^{(\ell+1)}(\xi):=\E{\exp\left[-\frac{1}{2}\norm{\xi}^2 \Sigma_{\alpha\alpha}^{(\ell)}\right]},
\]
where we remind the reader that 
\[
\Sigma_{\alpha\alpha}^{(\ell)}=\Var\left[z_{i;\alpha}^{(\ell+1)}~\big|~ \mF^{(\ell)}\right] =C_b+\frac{C_W}{n_\ell}\sum_{j=1}^{n_\ell} \sigma(z_{j;\alpha}^{(\ell)})^2
\]
is a collective observable \textit{at the previous layer}. Writing
\[
\kappa_{\alpha\alpha}^{(\ell)}:=\E{\Sigma_{\alpha\alpha}^{(\ell)}},\qquad \Delta_{\alpha\alpha}^{(\ell)}:=\Sigma_{\alpha\alpha}^{(\ell)} - \E{\Sigma_{\alpha\alpha}^{(\ell)}},
\]
we find
\[
p^{(\ell+1)}(\xi):=\E{\exp\left[-\frac{1}{2}\norm{\xi}^2 \Delta_{\alpha\alpha}^{(\ell)}\right]}\exp\left[-\frac{1}{2}\norm{\xi}^2\kappa_{\alpha\alpha}^{(\ell)}\right].
\]
The second term is precisely the characteristic function of a centered $m$-dimensional Gaussian with iid components of variance $\kappa_{\alpha\alpha}^{(\ell)}$. Moreover, at least heuristically, the first term can be written
\[
\E{\exp\left[-\frac{1}{2}\norm{\xi}^2 \Delta_{\alpha\alpha}^{(\ell)}\right]} = \sum_{q\geq 0} \E{\lr{\Delta_{\alpha\alpha}^{(\ell)}}^q} \frac{(-1)^q}{2^qq!} \norm{\xi}^{2q}.
\]
The concentration estimates \eqref{E:Delta-concentration} ensure that this series converges. Moreover, since the Fourier transform turns polynomials into derivatives we have
\[
-\norm{\xi}^2 = \text{ Laplacian in the variables } z_{i;\alpha}^{(\ell+1)}.
\]
Hence, we obtain for any reasonable test function $f$ that
\[
\E{f(z_{i;\alpha}^{(\ell+1)},\, i=1,\ldots, m)} = \sum_{q=0}^\infty \frac{1}{2^qq!} \E{\lr{\Delta_{\alpha\alpha}^{(\ell)}}^q} \bk{\lr{\sum_{i=1}^m \partial_{z_{i;\alpha}}^2}^q f(z_{i;\alpha},\, i=1,\ldots, m)}_{\kappa_{\alpha\alpha}^{(\ell)}},
\]
where $(z_{i;\alpha},\, i=1,\ldots, m)$ is a vector of iid centered Gaussians with variance $\kappa_{\alpha\alpha}^{(\ell)}$. The concentration estimates \eqref{E:Delta-concentration} ensure that this expression is a power series in $1/n$. In particular,
\begin{align}
\label{E:pert-exp-outline}\E{f(z_{i;\alpha}^{(\ell+1)},\, i=1,\ldots, m)} &= \bk{ f(z_{i;\alpha},\, i=1,\ldots, m)}_{\kappa_{\alpha\alpha}^{(\ell)}} \\ \notag &+\frac{\E{(\Delta_{\alpha\alpha}^{(\ell)})^2}}{8} \bk{\lr{\sum_{i=1}^m \partial_{z_{i;\alpha}}^2}^2 f(z_{i;\alpha},\, i=1,\ldots, m)}_{\kappa_{\alpha\alpha}^{(\ell)}} + O(n^{-2}). 
\end{align}
This is the essence of Theorem \ref{T:pert-exp}. To derive usable recursions for cumulants of $z_{i;\alpha}^{(\ell+1)}$, note for instance that, in the notation of Corollary \ref{C:cumulants-vals},
\[
\kappa_{4;\alpha}^{(\ell)}:=\frac{1}{3}\kappa\lr{z_{i;\alpha}^{(\ell+1)},z_{i;\alpha}^{(\ell+1)},z_{i;\alpha}^{(\ell+1)},z_{i;\alpha}^{(\ell+1)}} =\E{(\Delta_{\alpha\alpha}^{(\ell)})^2}.
\]
Writing
\[
X_j:=\sigma(z_{j;\alpha}^{(\ell+1)})^2 - \E{\sigma(z_{j;\alpha}^{(\ell+1)})^2}
\]
we thus have
\[
\kappa_{\alpha\alpha}^{(\ell+1)} = \E{\lr{\Delta_{\alpha\alpha}^{(\ell)}}^2}=\frac{C_W^2}{n_\ell}\E{X_1^2} +C_W^2\lr{1-n_\ell^{-1}} \E{X_1X_2}.
\]
Applying the expansion \eqref{E:pert-exp-outline} to both these terms and a bit of algebra already yields 
\begin{align*}
\kappa_{\alpha\alpha}^{(\ell+1)} &= \E{\lr{\Delta_{\alpha\alpha}^{(\ell)}}^2}\\
&=\frac{C_W^2}{n_\ell}\lr{\bk{\sigma^4}_{\kappa_{\alpha\alpha}^{(\ell)}}-\bk{\sigma^2}_{\kappa_{\alpha\alpha}^{(\ell)}}^2}\\
&+C_W^2\lr{1-n_\ell^{-1}} \lr{\lr{\bk{\sigma^2}_{\kappa_{\alpha\alpha}^{(\ell)}}-\E{\sigma(z_{1;\alpha}^{(\ell)})^2}}^2 + \frac{1}{4}\bk{\partial^2 \sigma^2}_{\kappa_{\alpha\alpha}^{(\ell)}}^2\kappa_{4;\alpha}^{(\ell)}}+O(n^{-2}).    
\end{align*}
A short argument supplied in \S \ref{S:cumulants-recs-pf} shows that $\bk{\sigma^2}_{\kappa_{\alpha\alpha}^{(\ell)}}=\E{\sigma(z_{1;\alpha}^{(\ell)})^2}+O(n^{-1})$ and that we may replace $\kappa_{\alpha\alpha}^{(\ell)}$ by its infinite width limit  $K_{\alpha\alpha}^{(\ell)}$ in all remaining expectations at the cost of an $O(n^{-1})$ error. This already yields the recursion \eqref{E:k4-rec} of Corollary \ref{C:cumulants-vals}.

\section{Background}
\subsection{Properties of Cumulants}\label{S:cumulant-background}
Recall that, given random variables $X_1,\ldots, X_k$ on the same probability space, we denote their mixed cumulant by 
\[
\kappa\lr{X_1,\ldots, X_k}:= i^k\frac{\partial^k}{\partial t_1\cdots \partial t_k}\bigg|_{t=0}\log \E{\exp\left[-i(t_1X_1+\cdots + t_kX_k)\right]}.
\]
In the following result, we recall the key properties of these mixed cumulants that we will need. 
\begin{proposition}[See Theorem 2.3.1 in \cite{brillinger2001time}]\label{P:cumulant-props}
Mixed cumulants satisfy the following properties. 
\begin{enumerate}
    \item Suppose $X=\lr{X_1,\ldots,X_k}$ is a random vector with finite moments of all orders. Then
\begin{align}\label{E:conditional-cumulants}
    \kappa\lr{X_1,\ldots, X_k} = \sum_{\pi=\lr{\pi_1,\ldots \pi_b}} \kappa\lr{\kappa\lr{X_{\pi_1}~|~\mF},\ldots, \kappa\lr{X_{\pi_b}~|~\mF}},
\end{align}
where the sum is over all partitions $\pi$ of $[k]$ and for each $a=1,\ldots, b$ 
\[
X_{\pi_a}:=\lr{X_i,\, i\in \pi_a}.
\]
This is known as the law of total cumulance. See \cite{brillinger1969calculation}. 
\item Suppose $X=\lr{X_1,\ldots,X_k}$ is a random vector with finite moments of all orders. When $X$ can be partitioned into two independent subsets, the mixed cumulant $\kappa(X)$ vanishes. More precisely, suppose $I\subseteq [k]$ and that $I,I^c\neq \emptyset.$ Then
\begin{equation}\label{E:indep-cumulants}
    X_I:=\lr{X_i,\, i\in I} \perp X_{I^c}=\lr{X_i,\, i\not \in I} \qquad \Longrightarrow\qquad \kappa\lr{X_1,\ldots, X_k}=0.
\end{equation}
\item Mixed cumulants are multi-linear. More precisely if 
\[
\set{X_{i,j},\, 1\leq j\leq k,\, i\leq T_j}
\] 
are random variables with finite moments defined on the same probability space, then
\begin{equation}\label{E:linear-cumulants}
\kappa\lr{\sum_{i_1=1}^{T_1} a_{i_1,1}X_{i_1,1},\ldots, \sum_{i_k=1}^{T_k} a_{i_k,k}X_{i_k,k}} =\sum_{i_1=1}^{T_1} \cdots\sum_{i_k=1}^{T_k}a_{i_1,1}\cdots a_{i_k,k} \kappa\lr{X_{i_1,1},\ldots,  X_{i_k,k}}    
\end{equation}
for any $a_{i,j}\in \R$.
\item Suppose $X=\lr{X_1,\ldots, X_j}\sim \mN(0,\Sigma)$ is a centered Gaussian with covariance $\Sigma$. Then for any $i_1,\ldots,i_k\in [j]$
\begin{equation}\label{E:gaussian-cumulants}
\kappa\lr{X_{i_1},\ldots, X_{i_k}} = \begin{cases}\Sigma_{i_1i_2},&\quad k=2 \\ 0,&\quad \text{otherwise}\end{cases}  .
\end{equation}
\item Moments are polynomials in cumulants. Specifically, suppose $X=\lr{X_1,\ldots,X_k}$ is a random vector with finite moments of all orders. Then,
\begin{equation}\label{E:moments-cumulants}
    \E{X_1\cdots X_k}=\sum_{\pi=\lr{\pi_1,\ldots,\pi_b}} \prod_{a=1}^b \kappa\lr{X_{\pi_a}},
\end{equation}
where the sum is over all partitions of $[k]$ and for each $a\in [b]$ we've written
\[
X_{\pi_a}:=\lr{X_i,\,i\in \pi_a}.
\]
\item Cumulants are polynomials in moments. Specifically, 
\begin{equation}\label{E:cumulants-moments}
    \kappa\lr{X_1,\ldots, X_k}=\sum_{\pi=\lr{\pi_1,\ldots,\pi_b}} (-1)^{b-1}(b-1)!\prod_{a=1}^b\E{\prod_{i\in \pi_a}X_i},
\end{equation}
where the sum is over all partitions of $[k]$ and for each $a\in [b]$ we've written
\[
X_{\pi_a}:=\lr{X_i,\,i\in \pi_a}.
\]
\end{enumerate}
    
\end{proposition}

\subsection{Lemmata}
In this section, we collect two simple auxiliary results that we will need. The first is a Lemma for Solving Certain Recursions.
\begin{lemma}\label{L:rec-lemma}
Fix $C_1,C_2,\psi>0$ satisfying
\[
C_2\geq 1,\quad \psi\neq C_2+1
\]
as well as $*\in \set{\leq,\, \geq}$. Suppose also that for each $\ell \geq 0$ we have
\begin{equation}\label{E:a-rec}
a_{\ell+1} ~*~ \xi_\ell + (1-\zeta_\ell)a_\ell,\qquad \zeta_\ell\in [0,1]
\end{equation}
with $a_0\in \R$ given and that there exist $C_1',C_2'>0$ so that
\[
\abs{\xi_\ell - C_1\ell^{-\psi}}\leq C_1' \ell^{-1-\psi},\qquad \abs{\zeta_\ell - C_2\ell^{-1}}\leq C_2'\ell^{-2}.
\]
Then 
\begin{equation}\label{E:a-form}
a_{\ell+1} * \frac{\ell^{1-\psi}}{1-\psi+C_2}\lr{1+O(\ell^{-1})} + e^{-C_2\gamma}\ell^{-C_2}a_0\lr{1+O(\ell^{-1})}
\end{equation}
where $\gamma$ is the Euler-Mascheroni constant and the implied constants depend only $C_1,C_2,C_1',C_2'.$
\end{lemma}
\begin{proof}
By unfolding the recursion \eqref{E:a-rec} we find
\begin{align*}
    a_{\ell+1} ~*~ \sum_{\ell'=1}^\ell \xi_{\ell'}\prod_{\ell''=\ell'+1}^{\ell}\lr{1-\zeta_{\ell''}} + a_0 \prod_{\ell''=0}^{\ell}(1-\zeta_{\ell''}). 
\end{align*}
We have
\begin{align*}
    \prod_{\ell''=1}^{\ell}(1-\zeta_{\ell''})&=\exp\left[\sum_{\ell''=1}^\ell \log\lr{1-C_2(\ell'')^{-1} + O(\ell^{-2})}\right]\\
    &=\exp\left[O(\ell^{-1})+\sum_{\ell''=1}^\ell -C_2(\ell'')^{-1} \right]\\
    &=\exp\left[O(\ell^{-1})-C_2\log(\ell) -C_2\gamma \right]\\
    &=e^{-C_2\gamma}\ell^{-C_2}\lr{1+O(\ell^{-1})}.
\end{align*}
This gives the second term in \eqref{E:a-form}. For the first term, we write
\begin{align*}
    \sum_{\ell'=1}^{\ell}\xi_{\ell'}\prod_{\ell''=\ell'+1}^\ell \lr{1-\zeta_\ell} & = \sum_{\ell'=1}^{\ell}\xi_{\ell'}\exp\left[\sum_{\ell''=\ell'+1}^\ell\log \lr{1-\zeta_\ell}\right]\\
    & = \sum_{\ell'=1}^{\ell}\xi_{\ell'}\exp\left[\sum_{\ell''=\ell'+1}^\ell-C_2(\ell'')^{-1} + O((\ell'')^{-2})\right]\\
    & = \sum_{\ell'=1}^{\ell}C_1(\ell')^{-\psi}(1+O(\ell')^{-1})\exp\left[-C_2\log\lr{\frac{\ell}{\ell'}} + O((\ell')^{-1})\right]\\
    & = \ell^{-C_2}\sum_{\ell'=1}^{\ell}C_1(\ell')^{-\psi+C_2}(1+O(\ell')^{-1})\\
    &=\frac{C_1}{1+C_2-\psi}\ell^{1-\psi}\lr{1+O(\ell^{-1})}.
\end{align*}
This completes the proof of \eqref{E:a-form}.
\end{proof}
The second result we will need is the following simple Lemma about Gaussian Integrals.



\begin{lemma}\label{L:smooth-gauss}
Fix $r\geq 1$, a $r\times r$ matrix $\Sigma$ and measurable function $g:\R^r\gives \R$ that is polynomially bounded:
\[
\exists k\geq 1\text{ s.t. } \sup_{x\in \R^k} \abs{\lr{1+\norm{x}}^{-k}g(x)}<\infty.    
\]
If $X$ is a standard Gaussian random vector in $\R^r$, then the function
\begin{equation}\label{E:Sigma-exp}
\Sigma \mapsto \E{g\lr{\Sigma^{1/2}X}}    
\end{equation}
is smooth on the open set of strictly positive definite $k\times k$ matrices. Further, if $g$ is a smooth function and each of its derivatives is polynomially bounded, then the map \eqref{E:Sigma-exp} is extends to a smooth function on the closed set of positive semi-definite matrices and, in particular,
\begin{equation}\label{E:gauss-exp-deriv}
\frac{\partial }{\partial \Sigma_{ij}}  \E{g\lr{\Sigma^{1/2}X}} = \E{(\partial_{i}\partial_{j}g)(\Sigma^{1/2}X)}.    
\end{equation}
\end{lemma}
\begin{proof}
On the open set of strictly positive definite matrices, the Gaussian density
\[
\Sigma\mapsto \exp\left[-\frac{1}{2}x^T\Sigma^{-1}x -\frac{1}{2}\log \det (2\pi \Sigma)\right]
\]
is a smooth function of $\Sigma$ with derivatives that are polynomials in $x$ and the entries of $\Sigma,\Sigma^{-1}$. The assumption that $f$ is polynomially bounded shows that we may differentiate under the integral sign and see that that 
\[
\E{g(\Sigma^{1/2}X)} = \int_{\R^r} g(x)\exp\left[-\frac{1}{2}x^T\Sigma^{-1}x -\frac{1}{2}\log \det (2\pi \Sigma)\right]dx
\]
is indeed a smooth function of $\Sigma$. Suppose instead that $g$ is a smooth function and that it's derivatives are all polynomially bounded. Suppose first that $g$ is in fact a Schwartz function. Then, writing $\widehat{g}$ for its Fourier transform we have
\[
\E{g(\Sigma^{1/2}X)}=\int_{\R^r} \widehat{g}(\xi) \exp\left[-\frac{1}{2}\xi^T \Sigma \xi\right] d\xi. 
\]
Since $\widehat{g}$ is also Schwartz, we may differentiate under the integral sign to obtain 
\begin{equation}\label{E:schwartz-deriv}
\frac{\partial}{\partial \Sigma_{ij}} \E{g(\Sigma^{1/2}X)}  = -\int_{\R^r}\xi_{i}\xi_j\widehat{g}(\xi) \exp\left[-\frac{1}{2}\xi^T \Sigma \xi\right]d\xi = \E{\partial_{x_i}\partial_{x_j}\bigg|_{x=\Sigma^{1/2}X} g(x)}.
\end{equation}
Finally, if $g$ is not Schwartz but is smooth with all derivatives being polynomially bounded, we consider the convolution
\[
g_\epsilon(x):=(g*\psi_\epsilon)(x),\qquad \psi_\epsilon(y)= \exp\left[-\frac{\norm{y}^2}{2\epsilon}-\frac{1}{2}\log(2\pi \epsilon)\right].
\]
Then, $g_\epsilon$ is Schwartz for all $\epsilon>0$. Moreover, note that  $g_\epsilon(\Sigma^{1/2}x)$ is also polynomially bounded for any PSD matrix $\Sigma.$ Specifically, for any fixed PSD matrix $\Sigma_0$ we have for any $k\geq 1$ 
\begin{align}
&\notag \sup_{\epsilon\in [0,1]}\sup_{\substack{\norm{\Sigma-\Sigma_0}\leq 1\\ \Sigma\text{ PSD}}}\sup_{x\in \R^{r}}\abs{(1+\norm{x})^{-k} g_\epsilon(\Sigma^{1/2}x)}\\
&\notag\qquad =\sup_{\epsilon\in [0,1]}\sup_{\substack{\norm{\Sigma-\Sigma_0}\leq 1\\ \Sigma\text{ PSD}}}\sup_{x\in \R^{r}}\abs{(1+\norm{x})^{-k} \int_{\R^r} g(\Sigma^{1/2}(x-y)) \psi_\epsilon(y)dy}\\
\notag &\qquad \leq\sup_{\epsilon\in [0,1]}\sup_{\substack{\norm{\Sigma-\Sigma_0}\leq 1\\ \Sigma\text{ PSD}}}\sup_{x\in \R^{r}}\left\{ (1+\norm{x})^{-k} \int_{\R^r} \lr{1+\norm{\Sigma^{1/2}(x-y)}^k} \psi_\epsilon(y)dy\right\}\\
\label{E:conv-deriv}&\qquad <\infty,
\end{align}
Note that there exists $K>0$ depending only $k,r,\Sigma_0$ so that
\[
\sup_{\substack{\norm{\Sigma-\Sigma_0}\leq 1}}\norm{\Sigma^{1/2}(x-y)}^{k}\leq  K\lr{1+\norm{\Sigma_0^{1/2}}}^k(\norm{x}^{k}+\norm{y}^{k}).
\]
Hence, since
\[
\sup_{\epsilon\in [0,1]}\int_{\R^r} \norm{y}^k\psi_\epsilon(y)dy<\infty
\]
we find that 
\begin{align}
\label{E:conv-deriv-2}& \sup_{\epsilon\in [0,1]}\sup_{\substack{\norm{\Sigma-\Sigma_0}\leq 1\\ \Sigma\text{ PSD}}}\sup_{x\in \R^{r}}\abs{(1+\norm{x})^{-k}g_\epsilon(\Sigma^{1/2}x)}<\infty.
\end{align}
The estimate above allows us to use dominate convergence to see that for any PSD $\Sigma$  
\begin{equation}\label{E:conv-conv}
\E{g(\Sigma^{1/2}X)} = \lim_{\epsilon\gives 0} \E{g_\epsilon(\Sigma^{1/2}X)}.  
\end{equation}
To complete the proof we note that $g_\epsilon$ and $\partial_i\partial_j g_\epsilon$ are both Schwartz for any positive $\epsilon$. Moreover, $\partial_i\partial_j\partial_k \partial_m g_\epsilon$ satisfies \eqref{E:conv-deriv-2}. Hence, we conclude by applying \eqref{E:schwartz-deriv} that for any PSD matrix $\Sigma_0$ there exists $C>0$ so that
\begin{align*}
&\sup_{\substack{\norm{\Sigma-\Sigma_0}\leq 1\\ \Sigma\text{ PSD}}}\sup_{\epsilon\in [0,1]}\frac{\abs{\E{g_\epsilon(\Sigma^{1/2}X)} - \E{g_\epsilon(\Sigma_0^{1/2}X)}- \sum_{i,j=1}^r\E{(\partial_{i}\partial_{j}g_\epsilon)(\Sigma_0^{1/2}X)} \lr{\Sigma-\Sigma_0}_{ij}}}{\norm{\Sigma-\Sigma_0}^{2}}\\
&\qquad \leq\sup_{\epsilon\in [0,1]}\sup_{\norm{\Sigma-\Sigma_0}\leq 1}\sum_{i,j,k,m=1,\ldots,r} \abs{\E{(\partial_{i}\partial_{j}\partial_{k}\partial_{m})g_{\epsilon}(\Sigma^{1/2}X)}}\\
&\qquad \leq  C.
\end{align*}
Thus, if $\Sigma-\Sigma_0/\norm{\Sigma-\Sigma_0}\gives \Sigma_1$, we find by applying \eqref{E:conv-conv} to $\partial_{i}\partial_j g$ that
\begin{align*}
    \lim_{\Sigma\gives \Sigma_0}\frac{\E{g(\Sigma^{1/2}X)} - \E{g(\Sigma_0^{1/2}X)}}{\norm{\Sigma-\Sigma_0}} &= \lim_{\Sigma\gives \Sigma_0}\lim_{\epsilon\gives 0} \frac{\E{g_\epsilon(\Sigma^{1/2}X)} - \E{g_\epsilon(\Sigma_0^{1/2}X)}}{\norm{\Sigma-\Sigma_0}}\\
    &=\lim_{\Sigma\gives \Sigma_0}\lim_{\epsilon\gives 0}\left\{\sum_{i,j=1}^r \E{(\partial_i\partial_j g_\epsilon (\Sigma_0^{1/2}X))} \frac{(\Sigma-\Sigma_0)_{ij}}{\norm{\Sigma-\Sigma_0}}\right\}\\
    &=\sum_{i,j=1}^r\E{\partial_{i}\partial_j g(\Sigma_0^{1/2}X)}(\Sigma_1)_{ij}.
\end{align*}
This shows that \eqref{E:gauss-exp-deriv} holds for any $g$ that is smooth, completing the proof of Lemma \ref{L:smooth-gauss}.
\end{proof}

\section{Proof of Theorem \ref{T:cumulants}}\label{S:cumulants-pf}
Let us first recall the notation. We fix $r\geq 1$ and assume that $\sigma:\R\gives\R$ satisfies assumption \ref{A:sigma-prop} with this value of $r$. We also fix a finite collection $x_\mA=\set{x_\alpha,\alpha\in \mA}\subseteq \R^{n_0}$ of distinct network inputs and $p$ directional derivatives $d_1,\ldots,d_p$ as in \eqref{E:der-deriv-def}. We denote by 
\[
N(p,r) = \#\set{J=\lr{j_1,\ldots,j_{p}}\in \mathbb N^{p}~|~j_1+\cdots + j_{p}\leq r},
\]
which computes the number of possible derivatives of order at most $r$ in the $p$ directional derivatives $d_j$. We also denote by $\mF^{(\ell)}$ the sigma algebra generated by the weights and biases in layers up to and including $\ell$. The starting point for proving Theorem \ref{T:cumulants} is the following simple but fundamental observation. 
\begin{lemma}\label{L:cond-indep}
For each $\ell \geq 0$, conditional on $\mF^{(\ell)}$, 
\[
\set{\lr{D_\alpha^Jz_{i;\alpha}^{(\ell+1)},\, \alpha\in \mA,\, \abs{J}\leq r}}_{i=1}^{n_{\ell+1}}
\]
is a collection of $n_{\ell+1}$ iid centered Gaussians of dimension $N(p,r)$. 
\end{lemma}
\begin{proof}
The defining recursion \eqref{E:z-def} of a fully connected network yields for each $\alpha,J$ 
\begin{align}
\label{E:deriv-rec}D_\alpha^J z_{i;\alpha}^{(\ell+1)} &= D_\alpha^J\left\{b_i^{(\ell+1)}+\sum_{j=1}^{n_\ell} W_{ij}^{(\ell+1)}\sigma\lr{z_{j;\alpha}^{(\ell)}}\right\}= \delta_{\abs{J}=0}  b_i^{(\ell+1)} + \sum_{j=1}^{n_\ell} W_{ij}^{(\ell+1)} D_\alpha^J \sigma\lr{z_{j;\alpha}^{(\ell)}}.
\end{align}
Note that $D_\alpha^J \sigma(z_{j;\alpha}^{_{(\ell)}})$ are measurable with respect to $\mF^{(\ell)}$. The conclusion now follows since the weights $W_{ij}^{_{(\ell+1)}},\,j=1\ldots, n_{\ell}$ and bias $b_{i}^{_{(\ell+1)}}$ are centered Gaussians and are independent for different $i$. 
\end{proof}

\noindent Thus, the structure of $z_{\alpha}^{_{(\ell+1)}}$ and its derivatives is always that of a Gaussian field after conditioning on $\mF^{_{(\ell)}}$. To ease the notation in what comes given $f:\R^{\abs{\mA}\times N(n_0,r)}\gives \R$, let us abbreviate 
\[
f\lr{z_{j;\mA}^{(\ell)}} := f\lr{D_\alpha^Jz_{j;\alpha}^{(\ell)},\, \alpha\in \mA,\, \abs{J}\leq r},\qquad j\in [n_{\ell}].
\]
Next, we remind the reader that given  $f:\R^{\abs{\mA}\times N(n_0,r)}\gives \R$, which is  measurable and polynomially bounded, the corresponding collective observable $\mO_{f}^{(\ell)}$ at layer $\ell$ is
\[
\mO_{f}^{(\ell)} = \frac{1}{n_\ell}\sum_{j=1}^{n_{\ell}} f\lr{z_{j;\mA}^{(\ell)}}
\]
and that the statement \eqref{E:iwl-goal-11} in Proposition \ref{P:iwl} ensures
\begin{equation}\label{E:col-fin-avg}
\sup_{n\geq 1} \E{\abs{\mO_{f}^{(\ell)}}}<\infty.    
\end{equation}
Recall also our notation for the conditional covariance
\[
\Sigma_{\alpha_1\alpha_2}^{(\ell)}:=\Cov\lr{z_{i;\alpha_1}^{(\ell+1)},z_{i;\alpha_2}^{(\ell+1)}~|~\mF^{(\ell)}}=C_b+\frac{C_W}{n_\ell}\sum_{j=1}^{n_\ell}\sigma\lr{z_{j;\alpha_1}^{(\ell)}}\sigma\lr{z_{j;\alpha_2}^{(\ell)}}
\]
and note that both it and its derivatives
\begin{align*}
D_{\alpha_1}^{J_1}D_{\alpha_2}^{J_2}\Sigma_{\alpha_1\alpha_2}^{(\ell)}&=\Cov\lr{D_{\alpha_1}^{J_1}z_{i;\alpha_1}^{(\ell+1)},D_{\alpha_2}^{J_2}z_{i;\alpha_2}^{(\ell+1)}~|~\mF^{(\ell)}}\\
&=D_{\alpha_1}^{J_1}D_{\alpha_2}^{J_2}\lr{C_b+\frac{C_W}{n_\ell}\sum_{j=1}^{n_\ell}\sigma\lr{z_{j;\alpha_1}^{(\ell)}}\sigma\lr{z_{j;\alpha_2}^{(\ell)}}}    
\end{align*}
are collective observables at layer $\ell$. Our first application of Lemma \ref{L:cond-indep} is the following reduction of the study of cumulants of $D_\alpha^Jz_{i;\alpha}^{(\ell+1)}$ to the cumulants of certain collective observables.

\begin{proposition}\label{P:collective-reduce}
Fix $k,\ell \geq 1$ and $p$-dimensional multi-indices $J_1,\ldots, J_{k}$ with $\abs{J_i}\leq r.$ If $k$ is odd, then
\[
 \kappa\lr{D_{\alpha_1}^{J_1}z_{i_1;\alpha_1}^{(\ell+1)},\ldots, D_{\alpha_k}^{J_{k}}z_{i_{k};\alpha_{k}}^{(\ell+1)}} =0
\]
In contrast, if $k$ is even
\[
 \kappa\lr{D_{\alpha_1}^{J_1}z_{i_1;\alpha_1}^{(\ell+1)},\ldots, D_{\alpha_k}^{J_{k}}z_{i_{k};\alpha_{k}}^{(\ell+1)}}~ =~ \text{ finite sums of  } \kappa\lr{\mO_{f_1}^{(\ell)},\ldots, \mO_{f_{k/2}}^{(\ell)}},
\]
where $\mO_{f_j}^{(\ell)}$ are collective observables of the form
\begin{equation}\label{E:sigma-derivs}
D_{\alpha_1}^{J_1} D_{\alpha_2}^{J_2} \Sigma_{\alpha_1\alpha_2}^{(\ell)},\qquad \abs{J_1},\abs{J_2}\leq r. 
\end{equation}
\end{proposition}
\begin{proof}
Using \eqref{E:conditional-cumulants} and recalling that $\mF^{(\ell)} $ is the sigma algebra generated by weights and biases in layers up to and including $\ell$, we have that  $\kappa\lr{D_{\alpha_1}^{J_1}z_{i_1;\alpha_1}^{(\ell+1)},\ldots, D_{\alpha_{k}}^{J_{k}}z_{i_{k};\alpha_{k}}^{(\ell+1)}} $ equals
\begin{align}
\label{E:cond-kappa}   \sum_{\substack{\pi=\lr{\pi_1,\ldots, \pi_B}}} \kappa\lr{\kappa\lr{ \lr{D^J z^{(\ell+1)}}_{\pi_1}\big|\mathcal F^{(\ell)}},\ldots, \kappa\lr{ \lr{D^J z^{(\ell+1)}}_{\pi_B}\big|\mathcal F^{(\ell)}}},
\end{align}
where the sum is over partitions $\pi$ of $[k]$ and for $b=1,\ldots, B$ we've abbreviated
\[
\lr{D^J z^{(\ell+1)}}_{\pi_b}:=\lr{D_{\alpha_t}^{J_t}z_{i_t;\alpha_t}^{(\ell+1)},\quad t\in \pi_b}.
\]
By Lemma \ref{L:cond-indep}, $\{(D_\alpha^Jz_{i;\alpha}^{(\ell+1)},\, \alpha\in \mA,\, \abs{J}\leq d),\, i=1,\ldots, n_{\ell+1}\}$ are iid centered Gaussians conditional on $\mF^{(\ell)}$. Hence, by the properties \eqref{E:indep-cumulants} and \eqref{E:linear-cumulants} and \eqref{E:gaussian-cumulants} from Proposition \ref{P:cumulant-props}, in the sum over partitions above, a term is non-zero only if 
\[
\forall b\in [B],\qquad \abs{\pi_b}=2\quad\text{and}\quad i_{\pi_b(1)}=i_{\pi_b(2)}
\]
This proves that $\kappa\lr{D^{J_1}z_{i_1;\alpha_1}^{(\ell+1)},\ldots, D^{J_{k}}z_{i_{k};\alpha_{k}}^{(\ell+1)}} $ vanishes if $k$ is odd. To treat the case when $k$ is even observe that by \eqref{E:deriv-rec}
\begin{align*}
\kappa\lr{ D^{J_1} z_{i_1;\alpha_1}^{(\ell+1)},D^{J_2} z_{i_2;\alpha_2}^{(\ell+1)}\big|\mathcal F^{(\ell)}} &=\delta_{i_1i_2}D_{\alpha_1}^{J_1}D_{\alpha_2}^{J_2}\Sigma_{\alpha_1\alpha_2}^{(\ell)}.
\end{align*}
Substituting this into \eqref{E:cond-kappa} completes the proof. 
\end{proof}

When $k=2$, Proposition \ref{P:collective-reduce} and our assumption \eqref{A:sigma-prop} shows that for each $\ell\geq 1$, any $i_1,i_2\in [n_{\ell+1}],\, \alpha\in \mA,$ and multi-indices $J_1,J_2$ of order at most $d$, there exists a polynomially bounded function $f:\R^{\abs{\mA}\times N(n_0,d)}\gives \R$ for which 
\[
\kappa\lr{D_{\alpha_1}^{J_1}z_{i_1;\alpha_1}^{(\ell+1)},D_{\alpha_2}^{J_2}z_{i_2;\alpha_2}^{(\ell+1)}} = \E{\mO_f^{(\ell)}}
\]
In light of \eqref{E:col-fin-avg} this proves Theorem \ref{T:cumulants} when $k=2$. Further, since the cumulant of $2$ or more random variables is shift-invariant, we may assume for $k\geq 3$ that the collective observables  $D_{\alpha_1}^{J_1}D_{\alpha_2}^{J_2} \Sigma_{\alpha_1\alpha_2}^{(\ell)}$ in Proposition \ref{P:collective-reduce} are replaced by their zero mean versions:
\begin{equation}\label{E:gen-delt-def}
\Delta_{\alpha_1\alpha_2}^{J_1,J_2,(\ell)}:=D_{\alpha_1}^{J_1}D_{\alpha_2}^{J_2} \Sigma_{\alpha_1\alpha_2}^{(\ell)} - \E{D_{\alpha_1}^{J_1}D_{\alpha_2}^{J_2} \Sigma_{\alpha_1\alpha_2}^{(\ell)}}.    
\end{equation}
Hence, Theorem \ref{T:cumulants} is a special case of the following result. 
\begin{theorem}\label{T:smooth-cumulants}
    Fix $k,m\geq 1$. Consider any $m$-tuple $F=\lr{f_1,\ldots, f_m}$ consisting of measurable, functions
    \[
    f_i:\R^{\abs{\mA}\x N(n_0,r)}\gives \R,\quad i=1,\ldots,m\qquad
    \]
    that are polynomially bounded and satisfy
    \[
    \E{\mO_{f_i}^{(\ell)}}=\E{f_i\lr{z_{1;\mA}^{(\ell)}}}=0,\qquad i=1,\ldots,m.
    \]
    Define the $m-$tuple of collective observables
    \[
    \amO_{F}^{(\ell)} := \lr{\mO_{f_i}^{(\ell)},\, i=1,\ldots, m}.
    \]
    Consider further any measurable polynomially bounded functions
    \[
    g_j:\R^m\gives \R,\quad j=1,\ldots,k.
    \]
    which are smooth in a neighborhood of $0$. If $f_i$ and $\sigma$ are in fact smooth, then, for every $\ell\geq 1$ 
    \begin{equation}\label{E:smooth-cumulants}
    \sup_{n\geq 1}  \abs{n^{k-1}\kappa\lr{g_1\lr{\amO_{F}^{(\ell)}},\ldots, g_k\lr{\amO_{F}^{(\ell)}}}}<\infty
    \end{equation}
    Moreover, \eqref{E:smooth-cumulants} holds without the assumption that $f_i,\sigma$ are smooth provided that for each $\ell$ the vector of iterated directional derivatives $(D_\alpha^{J}z_{i;\alpha}^{(\ell)},\, \abs{J}\leq r,\alpha\in \mA)$ of order at most $r$ is non-degenerate in the sense of \eqref{E:pd-cov}.
\end{theorem}
\begin{proof}
Our starting point is a reduction of Theorem \ref{T:smooth-cumulants} to the case when $g_j$ are polynomials. This is related to a technique called the delta method in some parts of the mathematical statistics literature \cite{ver2012invented}.
\begin{proposition}[Polynomials are Enough for Theorem \ref{T:smooth-cumulants}]\label{P:poly-to-gen}
    Fix $m\geq 1$ and suppose that for each $n\geq 1$ we have an $m-$tuple $X_n=\lr{X_{n,1},\ldots, X_{n,m}}$ of mean $0$ random variables that possess bounded moments of all orders:
    \begin{equation}\label{E:mixed-moment-bounds}
\sup_{n\geq 1}\abs{\E{X_{n,1}^{q_1}\cdots X_{n,m}^{q_m}}}<\infty,\qquad \forall\, q_1,\ldots, q_m\geq 0.        
    \end{equation}
    Suppose for any given polynomials $p_1,\ldots, p_k$ in $m$ variables we have
    \begin{equation}\label{E:poly-cumulants}
        \sup_{n\geq 1}\abs{n^{k-1} \kappa\lr{p_1(X_n),\ldots, p_k(X_n)}}<\infty.
    \end{equation}
    Then, for any measurable, polynomially bounded functions $g_j:\R^m\gives \R,\, j=1,\ldots,k,$ which are smooth in some fixed neighborhood of $0$
     \begin{equation}\label{E:prop-smooth-cumulants}
    \sup_{n\geq 1}  \abs{n^{k-1}\kappa\lr{g_1\lr{X_n},\ldots, g_k\lr{X_n}}}<\infty.    
    \end{equation}
\end{proposition}
\begin{proof}
We begin with the following simple Lemma, which allows us to translate between the cumulants bounds \eqref{E:poly-cumulants} and high probability bounds. 
\begin{lemma}\label{L:moment-bounds}
For any $q\geq 1$ 
    \[
    \sup_{n\geq 1}\sup_{1\leq i\leq m}\abs{n^{\lceil \frac{q}{2}\rceil}\E{X_{i,n}^q}}<\infty.
    \]
\end{lemma}
\begin{proof}
We have by property \eqref{E:moments-cumulants} from Proposition \ref{P:cumulant-props} that
\begin{align*}
    \E{X_{i;n}^q} = \sum_{\substack{\pi=\lr{\pi_1,\ldots, \pi_B}\\\pi\in P(m)}} \prod_{b=1}^B \kappa\big({\underbrace{X_{i;n},\ldots, X_{i;n}}_{\abs{\pi_b}\text{ times}}}\big).
\end{align*}
Since by assumption $X_{i;n}$ has mean $0$, we have
\[
\kappa(X_{i;n})=\E{X_{i;n}}=0.
\]
Thus, the only partitions $\pi=\lr{\pi_1,\ldots, \pi_B}\in S(m)$ that give rise to non-zero terms in the expression above must have $B\leq \lfloor \frac{q}{2}\rfloor$. Moreover, for any such partition, we have 
\[
\left\lceil \frac{q}{2}\right\rceil = q-\left\lfloor\frac{q}{2}\right\rfloor = -\left\lfloor\frac{q}{2}\right\rfloor+\sum_{b=1}^B \abs{\pi_b} \leq \sum_{b=1}^B \lr{\abs{\pi_b} - 1}.
\]
Hence, we find 
\begin{align*}
    \sup_{n\geq 1}\abs{n^{\lceil\frac{q}{2} \rceil}\E{X_{i;n}^q}}& \leq  \sum_{\substack{\pi=\lr{\pi_1,\ldots, \pi_B}\\\pi\in P(m),\, \abs{\pi_b}\geq 2}} \prod_{b=1}^B \sup_{n\geq 1}\abs{n^{\abs{\pi_b}-1}\kappa\big({\underbrace{X_{i;n},\ldots, X_{i;n}}_{\abs{\pi_b}\text{ times}}}\big)}<\infty,
\end{align*}
where the final inequality follows from the assumption \eqref{E:poly-cumulants}.
\end{proof}
Applying Markov's inequality and Lemma \ref{L:moment-bounds} shows that for any $q\geq 1$ we have
\begin{equation}\label{E:localization}
\sup_{n\geq 1} n^{q}\pr{S_n^c}<\infty,\qquad S_n:=\set{\abs{X_{i;n}}\leq n^{-1/4},\qquad i=1,\ldots, m}.    
\end{equation}
This localization estimate allows us to replace each $g_i$ by its Taylor expansion around $0$. Indeed, note that 
\[
\kappa\lr{g_1(X_n),\ldots, g_k(X_n)} = P\lr{\E{g_1(X_{n})^{q_1}\cdots g_k(X_{n})^{q_k}},\quad q_1+\cdots+q_k\leq k}
\]
for some universal polynomial $P$ evaluated at the mixed moments of $X_n$ (the formula for this polynomial is given in \eqref{E:cumulants-moments} but is not important). Moreover, using the growth assumption \eqref{E:mixed-moment-bounds} on $X$ and the fact that $g_i$ are  polynomially bounded we find that
  \begin{equation}\label{E:mixed-moment-bounds-g}
\sup_{n\geq 1}\E{g(X_{n})^{q_1}\cdots g_k(X_{n})^{q_k}}<\infty,\qquad \forall\, q_1,\ldots, q_k\geq 1.        
    \end{equation}
This, in combination with the localization estimate \eqref{E:localization} applied with $q=k-1$ yields
\[
\kappa\lr{g_1(X_n),\ldots, g_k(X_n)} = P\lr{\E{{\bf 1}_{S_n}g_1(X_n)^{q_1}\cdots g_k(X_n)^{q_m}},\quad q_1+\cdots+q_m\leq k} + O(n^{-k+1}).
\]
Note that for $n$ sufficiently large, on the event $S_n$, the argument $X_n$ is any fixed neighborhood of $0$. Hence, we may write
\[
g_j(X_n)= p_j(X_n)+O(n^{-k+1}),
\]
where $p_j$ represents the $q-$th order Taylor expansion of $g_j$ around $0$ with $q$ sufficiently large and the constant in the error term is uniformly bounded. This yields
\[
\kappa\lr{g_1(X_n),\ldots, g_k(X_n)} = P\lr{\E{{\bf 1}_{S_n}p_1(X_n)^{q_1}\cdots p_k(X_n)^{q_m}},\quad q_1+\cdots+q_m\leq k} + O(n^{-k+1}).
\]
Finally, using the mixed moment estimates \eqref{E:mixed-moment-bounds} and the localization estimate \eqref{E:localization}, we conclude
\begin{align*}
\kappa\lr{g_1(X_n),\ldots, g_k(X_n)} &= P\lr{\E{p_1(X_n)^{q_1}\cdots p_k(X_n)^{q_m}},\quad q_1+\cdots+q_m\leq k} + O(n^{-k+1})\\
&= \kappa\lr{p_1(X_n),\ldots, p_k(X_n)}+O(n^{-k+1}).    
\end{align*}
Recalling \eqref{E:poly-cumulants} completes the proof.
\end{proof}

Proposition \ref{P:poly-to-gen} shows that, in establishing the conclusion \eqref{E:smooth-cumulants} of Theorem \ref{T:smooth-cumulants}, it is sufficient to assume that $g_j$ are polynomials. The remainder of the proof of Theorem \ref{T:smooth-cumulants} is by induction on $\ell$, starting with $\ell=1$. In view of Proposition \ref{P:poly-to-gen}, the following result establishes the base case. 

\begin{proposition}[Base Case: Theorem \ref{T:smooth-cumulants} holds for polynomials at $\ell=1$]\label{P:base-case-cumulants}
    Fix $k,m\geq 1$ and suppose $f_i,\, i=1,\ldots, m$ are as in the statement of Theorem \ref{T:smooth-cumulants}. Then, if $p_1,\ldots,p_k$ are any polynomials in $m$ variables, we have
    \[
    \sup_{n\geq 1} \abs{n ^{k-1}\kappa\lr{p_1\lr{\amO_{F}^{(1)}},\ldots, p_k\lr{ \amO_{F}^{(1)}}}}<\infty.
    \]
\end{proposition}
\begin{proof}
Since cumulants are multi-linear, we may and shall assume that $p_a$ are monomials:
\begin{equation}\label{E:gi-mono-def}
p_a(x)=x^{Q^{(a)}}:= x_1^{q_{1}^{(a)}}\cdots x_m^{q_m^{(a)}},\qquad x=\lr{x_1,\ldots, x_m},\quad Q^{(a)}=\lr{q_1^{(a)},\ldots, q_m^{(a)}}.    
\end{equation}
Recall that
\[
\mO_{f_i}^{(1)}:=n_1^{-1}\sum_{j=1}^{n_1}f_i\lr{z_{j;\mA}^{(1)}}.
\]
Therefore, writing $q^{(a)}:=q_1^{(a)}+\cdots + q_m^{(a)}$ we find
\begin{align*}
p_a\lr{\amO_{F}^{(1)}}= n_1^{-q^{(a)}}\sum_{J^{(a)}} f_{J^{(a)}},\qquad  f_{J^{(a)}}:= \prod_{i=1}^m \prod_{q=1}^{q_i^{(a)}} f_i\big(z_{j_{q;i}^{(a)};\mA}^{(1)}\big),
\end{align*}
where the sum is over tuples of multi-indices
\begin{equation}\label{E:Ja-def}
J^{(a)}=\lr{J_1^{(a)},\ldots, J_m^{(a)}},\qquad J_i^{(a)}=\lr{j_{q;i}^{(a)}\in [n_1],\, i\in [m],\, q\in [q_i^{(a)}]}.
\end{equation}
Hence, using that cumulants are multi-linear (see \eqref{E:linear-cumulants}), we obtain
\begin{align*}
    \kappa\lr{p_1\lr{\amO_{F}^{(1)}},\ldots, p_k\lr{\amO_{F}^{(1)}}} = n_1^{-(q^{(1)}+\cdots+q^{(k)})} \sum_{\substack{J^{(1)},\ldots, J^{(k)}}}\kappa\lr{f_{J^{(1)}},\ldots, f_{J^{(k)}}},
\end{align*}
where the sum extends over ordered collections $\lr{J^{(a)},\, 1\leq a \leq k}$ of multi-indices as in \eqref{E:Ja-def}. The expression on the right hand side can be interpreted as an average. Namely, we can think of the indices $j_{q;i}^{(a)}\in [n_1]$ are chosen uniformly from $[n_1]$ and independently for all $i,q,a$. Writing $\mathcal E$ for the average with respect to this distribution, we obtain
\begin{align*}
    \kappa\lr{p_1\lr{\amO_{F}^{(1)}},\ldots, p_k\lr{\amO_{F}^{(1)}}} =\mathcal E\left[\kappa\lr{f_{J^{(1)}},\ldots, f_{J^{(k)}}}\right].
\end{align*}
Our goal is to show that this average is small. To quantify this, let us associate to each collection $\lr{J^{(a)},\, a\in [k]}$ a graph 
\begin{equation}\label{E:graph-def}
\mathcal G\lr{J^{(a)},\, a\in [k]} = \lr{[k],\, E\lr{J^{(a)},\, a\in [k] }},    
\end{equation}
with vertex set $[k]$ and edge set defined by 
\[
(a,a')\in \mathcal E\lr{J^{(a)},\, a\in [k] } \quad \Longleftrightarrow\quad \exists i,i'\in [m],\, q\in [q_i^{(a)}],\, q'\in [q_{i'}^{(a')}]\text{ s.t. } j_{q;i}^{(a)} = j_{q';i'}^{(a')}.
\]
The key point is that in light of the vanishing property \eqref{E:indep-cumulants} of cumulants  and the fact that neurons at layer $1$ are independent
\[
\mathcal G\lr{J^{(a)},\, a\in [k]}\text{ disconnected }\quad \Longrightarrow\quad \kappa\lr{f_{J^{(1)}}^{(1)},\ldots, f_{J^{(k)}}^{(1)}} = 0.
\]
Hence, 
\begin{align*}
     \kappa\lr{p_1\lr{\amO_{F}^{(1)}},\ldots, p_k\lr{\amO_{F}^{(1)}}}= \mathcal E\left[ {\bf 1}_{\set{\mG\lr{J^{(a)},\, a\in [k]}\text{ connected}}}\kappa\lr{f_{J^{(1)}}^{(1)},\ldots, f_{J^{(k)}}^{(1)}}\right].
\end{align*}
Since $f_i$ are assumed to be polynomially bounded and the distribution of the neuron pre-activations $z_{i;\alpha}^{_{(1)}}$ is that of centered Gaussians with mean $0$ and covariance
\[
\Cov\lr{z_{i_1;\alpha_1}^{_{(1)}},z_{i_2;\alpha_2}^{_{(1)}}} = \delta_{i_1i_2}\lr{C_b+\frac{C_W}{n_0}\sum_{j=1^{n_0}}x_{j;\alpha_1}x_{j;\alpha_2}},
\]
we have for any fixed $k$ that 
\[
\sup_{n\geq 1}\sup_{J^{(1)},\ldots, J^{(a)}}\abs{\kappa\lr{f_{J^{(1)}}^{(1)},\ldots, f_{J^{(k)}}^{(1)}}}<\infty.
\]
Hence, 
\begin{align*}
    \kappa\lr{p_1\lr{\amO_{F}^{(1)}},\ldots, p_k\lr{\amO_{F}^{(1)}}} = O\lr{\mathcal P\lr{\mG\lr{J^{(a)},\, a\in [k]}\text{ connected}}},
\end{align*}
where $\mathcal P$ is the probability measure associated to our random draw of $J^{(1)},\ldots, J^{(k)}$. To complete the proof, note that since $m,q_i^{(a)}$ are fixed, by a simple union bound, we obtain
\begin{align*}
    &\mathcal P\lr{\mG\lr{J^{(a)},\, a\in [k']}\text{ connected}~\bigg|~\mG\lr{J^{(a)},\, a\in [k'-1]}\text{ connected}}= O(n^{-1}). 
\end{align*}
Hence, 
\begin{align}
\notag    &\mathcal P\lr{\mG\lr{J^{(a)},\, a\in [k]}\text{ connected}}\\
\notag    &\qquad  = \prod_{k'=2}^k \mathcal P\lr{\mG\lr{J^{(a)},\, a\in [k']}\text{ connected}~\bigg|~\mG\lr{J^{(a)},\, a\in [k'-1]}\text{ connected}}\\
\label{E:con-prob}    &\qquad =O(n^{-k+1}).
\end{align}
Thus, 
\[
\kappa\lr{p_1\lr{\amO_{F}^{(1)}},\ldots, p_k\lr{\amO_{F}^{(1)}}} = O(n^{-k+1}),
\]
as desired. 
\end{proof}

Propositions \ref{P:poly-to-gen} and \ref{P:base-case-cumulants} together show that the conclusion \eqref{E:smooth-cumulants} of Theorem \ref{T:smooth-cumulants} holds at layer $1$. In conjunction with Proposition \ref{P:poly-to-gen}, the following result establishes that if the conclusion \eqref{E:smooth-cumulants} of Theorem \ref{T:smooth-cumulants} holds at some layer $\ell\geq 1$ then it also holds at layer $\ell+1$. This will complete the proof by inductive of Theorem \ref{T:smooth-cumulants}.

\begin{proposition}[Inductive Step: Reducing polynomial cumulants in layer $\ell+1$ to smooth cumulants in layer $\ell$]\label{P:inductive-cumulants}
Fix $\ell \geq 1$. \\

\noindent \underline{Case 1}: Suppose that $\sigma$ is smooth. Assume that for any collection 
\[
F'=\lr{f_i':\R^{\abs{\mA}\times N(n_0,r)}\gives \R,\, i=1,\ldots,m}
\] 
of smooth and polynomially bounded functions and any $g_j$ as in the statement of Theorem \ref{T:smooth-cumulants} the conclusion \eqref{E:smooth-cumulants} of Theorem \ref{T:smooth-cumulants} holds at layer $\ell$:
\[
\sup_{n\geq 1}\abs{n^{k-1}\kappa\lr{g_1\lr{\amO_{F'}^{(\ell)}}, \ldots, g_k\lr{\amO_{F'}^{(\ell)}}}}<\infty.
\]
Then, if $p_1,\ldots, p_k$ are any polynomials in $m$ variables, and $F=\lr{f_i,\, i=1,\ldots, m}$ is an arbitrary collection of smooth and polynomially bounded functions $f_i:\R^{\abs{\mA}\times N(n_0,r)}\gives \R$, then 
\[
\sup_{n\geq 1}\abs{n^{k-1}\kappa\lr{p_1\lr{\amO_{F}^{(\ell+1)}}, \ldots, p_k\lr{\amO_{F}^{(\ell+1)}}}}<\infty.
\]\\

\noindent \underline{Case 2}: Suppose $\sigma$ is not smooth but satisfies Assumption \ref{A:sigma-prop} and that $(D_{\alpha}^{J}z_{i;\alpha}^{(\ell)},\, \alpha\in \mA,\, \abs{J}\leq r)$ is non-degenerate in the infinite width in the sense of \eqref{E:pd-cov}. Assume that for any collection 
\[
F'=\lr{f_i':\R^{\abs{\mA}\times N(n_0,r)}\gives \R,\, i=1,\ldots,m}
\] 
of measurable and polynomially bounded functions and any $g_j$ as in the statement of Theorem \ref{T:smooth-cumulants} the conclusion \eqref{E:smooth-cumulants} of Theorem \ref{T:smooth-cumulants} holds at layer $\ell$:
\[
\sup_{n\geq 1}\abs{n^{k-1}\kappa\lr{g_1\lr{\amO_{F'}^{(\ell)}}, \ldots, g_k\lr{\amO_{F'}^{(\ell)}}}}<\infty.
\]
Then, if $p_1,\ldots, p_k$ are any polynomials in $m$ variables, and $F=\lr{f_i,\, i=1,\ldots, m}$ is an arbitrary collection of measurable and polynomially bounded functions $f_i:\R^{\abs{\mA}\times N(n_0,d)}\gives \R$, then 
\[
\sup_{n\geq 1}\abs{n^{k-1}\kappa\lr{p_1\lr{\amO_{F}^{(\ell+1)}}, \ldots, p_k\lr{\amO_{F}^{(\ell+1)}}}}<\infty.
\]
\end{proposition}
\begin{proof}
 The proof of Proposition \ref{P:inductive-cumulants} is similar but somewhat more involved than that of Proposition \ref{P:base-case-cumulants}. Moreover, the two cases are proved in essentially the same way, except that we will employ the different cases in Lemma \ref{L:smooth-gauss}. We give the details in the case when $\sigma$ is smooth and indicate where the proof is modified slightly to handle the non-smooth case.

To start, as in the proof of Proposition \ref{P:base-case-cumulants}, note that since cumulants are multi-linear (see \eqref{E:linear-cumulants}), it is enough to assume that $p_j$ are monomials. Thus, borrowing the notation from the proof of Proposition \ref{P:base-case-cumulants} (see starting \eqref{E:gi-mono-def}), we find
\begin{align*}
    \kappa\lr{p_1\lr{\amO_{F}^{(\ell+1)}},\ldots, p_k\lr{\amO_{F}^{(\ell+1)}}} = n_{\ell+1}^{-(q^{(1)}+\cdots+q^{(a)})} \sum_{\substack{J^{(1)},\ldots, J^{(k)}}}\kappa\lr{f_{J^{(1)}}^{(\ell+1)},\ldots, f_{J^{(k)}}^{(\ell+1)}},
\end{align*}
where
\[
f_{J^{(a)}}^{(\ell+1)}:=\prod_{i=1}^m\prod_{q=1}^{q_i^{(a)}} f_{j_\alpha^{(a)}}^{(\ell+1)},\qquad f_{j}^{(\ell+1)}:=f\lr{z_{j;\mA}^{(\ell+1)}}.
\]
Note that, as in Proposition \ref{P:iwl}, the polynomially bounded assumption on $f_j$ and the non-linearity $\sigma$ together with the Gaussianity of weights and biases show that
\begin{equation}\label{E:fin-cumulants}
\sup_{n\geq 1}\abs{\kappa\lr{f_{J^{(1)}}^{(\ell+1)},\ldots, f_{J^{(k)}}^{(\ell+1)}}}<\infty.
\end{equation}
Using the law of total cumulance \eqref{E:conditional-cumulants}, we find that $\kappa\lr{p_1(\amO_{F}^{(\ell+1)}),\ldots, p_k(\amO_{F}^{(\ell+1)})}$ equals
\begin{align*}
    \sum_{\substack{\pi=\lr{\pi_1,\ldots, \pi_B}}}n_{\ell+1}^{-(q^{(1)}+\cdots+q^{(a)})} \sum_{\substack{J^{(1)},\ldots, J^{(k)}}}\kappa\lr{\kappa\lr{f_{J^{(\pi_1)}}^{(\ell+1)}~|~\mF^{(\ell)}},\ldots,\kappa\lr{ f_{J^{(\pi_B)}}^{(\ell+1)}~|~\mF^{(\ell)}}},
\end{align*}
where $\pi$ is any partition of $[k]$ and 
\[
f_{J^{(\pi_b)}} := \lr{f_{J^{(a)}},\, a \in \pi_b}.
\]
Just as in the proof of Proposition \ref{P:base-case-cumulants}, we may interpret the sum over $J^{(1)},\ldots, J^{(k)}$ as an average over the distribution in which $j_{q;i}^{(a)}$ are drawn iid uniformly on $[n_{\ell+1}]$. Writing $\mathcal E$ for averages with respect to this distribution yields 
\[
\kappa\lr{p_1(\amO_{F}^{(\ell+1)}),\ldots, p_k(\amO_{F}^{(\ell+1)})} =  \sum_{\substack{\pi=\lr{\pi_1,\ldots, \pi_b}}}\mathcal E\left[\kappa\lr{\kappa\lr{f_{J^{(\pi_1)}}^{(\ell+1)}~|~\mF^{(\ell)}},\ldots,\kappa\lr{ f_{J^{(\pi_b)}}^{(\ell+1)}~|~\mF^{(\ell)}}}\right]
\]
As in \eqref{E:graph-def}, we may associate to each collection $J^{(\pi_t)}$ the graph $\mathcal G\lr{J^{(\pi_t)}}$. Recall that by Lemma \ref{L:cond-indep}, the neurons pre-activations $D_\alpha^J z_{i;\alpha}^{(\ell+1)}$ in layer $\ell+1$ are independent for different $i$ conditional on $\mF^{(\ell)}$. Hence, in view of the vanishing property \eqref{E:indep-cumulants} of cumulants, we obtain that $\kappa\lr{p_1(\amO_{F}^{(\ell+1)}),\ldots, p_k(\amO_{F}^{(\ell+1)})}$ equals 
\begin{align*}
   \sum_{\substack{\pi=\lr{\pi_1,\ldots, \pi_B}}}\mathcal E\left[ {\bf 1}_{\set{\mG\lr{J^{(\pi_b)}}\text{ connected }\forall b\in [B]}}\kappa\lr{\kappa\lr{f_{J^{(\pi_1)}}^{(\ell+1)}~|~\mF^{(\ell)}},\ldots,\kappa\lr{ f_{J^{(\pi_B)}}^{(\ell+1)}~|~\mF^{(\ell)}}}\right]
\end{align*}
Since we've assumed that $f_i$ are smooth and polynomially bounded, Lemma \ref{L:smooth-gauss} shows that for each $b\in [B]=\set{1,\ldots, B}$ the conditional cumulant $\kappa\lr{ f_{J^{(\pi_b)}}^{(\ell+1)}~|~z^{(\ell)}}$ is a smooth function of the centered entries $D_{\alpha_1}^{J_1}D_{\alpha_2}^{J_2}\Delta_{\alpha_1\alpha_2}^{(\ell)}$ of the conditional covariance of $\lr{D_\alpha^J z_{i;\mA}^{(\ell+1)},\,\alpha\in \mA, \abs{J}\leq r}$ given $\mF^{(\ell)}$. Thus, since these entries are collective observables at layer $\ell$ we may apply the inductive hypothesis of Case 1 to find that 
\begin{align*}
   \kappa\lr{p_1(\amO_{F}^{(\ell+1)}),\ldots, p_k(\amO_{F}^{(\ell+1)})}=\sum_{\substack{\pi=\lr{\pi_1,\ldots, \pi_B}}}\mathcal P\left[ \mG\lr{J^{(\pi_b)}}\text{ connected }\forall b\in [B]\right] O(n^{-B+1}).
\end{align*}
Combining this with the estimate \eqref{E:con-prob} shows
\begin{align*}
&\kappa\lr{p_1\lr{\amO_{F}^{(\ell+1)}},\ldots, p_k\lr{\amO_{F}^{(\ell+1)}}}&=\sum_{\substack{\pi=\lr{\pi_1,\ldots, \pi_B}}} O(n^{-B+1})\prod_{b=1}^B O(n^{-\abs{\pi_b}+1}) = O(n^{-k+1}),
\end{align*}
as desired. The proof in Case 2 is almost identical. The only difference is that, we must introduce the event
\begin{equation}\label{E:Sn-def}
S_n = \set{\abs{\Delta_{\alpha_1\alpha_2}^{J_1J_2,(\ell)}}<n^{-1/4}}.     
\end{equation}
Precisely as in the proof of Lemma \ref{L:moment-bounds} we find that
\[
\mathbb P(S_n^c)=O(n^{-\infty}).
\]
Hence, 
\[
\kappa\lr{p_1\lr{\amO_{F}^{(\ell+1)}},\ldots, p_k\lr{\amO_{F}^{(\ell+1)}}} = \kappa\lr{p_1\lr{\amO_{F}^{(\ell+1)}},\ldots, p_k\lr{\amO_{F}^{(\ell+1)}}~|~S_n}+O(n^{-\infty}),
\]
where we've implicitly used \eqref{E:fin-cumulants}. Moreover, since in Case 2 we assume that the vector $\lr{D_\alpha^Jz_{i;\alpha}^{_{(\ell+1)}},\, \abs{J}\leq r,\, \alpha\in \mA}$ is non-degenerate in the infinite width limit in the sense of \eqref{E:pd-cov}, we see that for $n$ sufficiently large the covariance of  $\lr{D_\alpha^Jz_{i;\alpha}^{_{(\ell+1)}},\, \abs{J}\leq r,\, \alpha\in \mA}$ given $\mF^{(\ell)},$ which is the matrix with entries 
\[
\E{D_{\alpha_1}^{J_1}D_{\alpha_2}^{J_2}\Sigma_{\alpha_1\alpha_2}^{(\ell)}},\qquad \alpha_1,\alpha_2\in \mA,\, \abs{J_1},\abs{J_2}\leq r
\]
is also non-degenerate. On the event $S_n$, the conditional covariance of  $\lr{D_\alpha^Jz_{i;\alpha}^{_{(\ell+1)}},\, \abs{J}\leq r,\, \alpha\in \mA}$ given $\mF^{(\ell)},$ which is a matrix with entries $D_{\alpha_1}^{J_1}D_{\alpha_2}^{J_2}\Sigma_{\alpha_1\alpha_2}^{(\ell)}$, is also non-degenerate for all $n$ sufficiently large. Hence, we again conclude by Lemma \ref{L:smooth-gauss} that conditional on $S_n$ (which is measurable with respect to $\mF^{(\ell)}$) for each $b\in [B]$ the conditional cumulant $\kappa\lr{ f_{J^{(\pi_b)}}^{(\ell+1)}~|~\mF^{(\ell)}}$ is a smooth function of  $D_{\alpha_1}^{J_1}D_{\alpha_2}^{J_2}\Sigma_{\alpha_1\alpha_2}^{(\ell)}$, which are collective observables at layer $\ell$. The remainder of the proof now proceeds in the same way as for Case 1. 
\end{proof}
\end{proof}

\section{Proof of Theorem \ref{T:pert-exp}}
Let us recall the notation. We consider a random depth $L$ neural network with layer widths $n_0,\ldots, n_{L+1}$ with
\[
\exists c,C>0\text{  s.t.  }\qquad cn\leq n_1,\ldots, n_L\leq Cn,
\]
and a non-linearity $\sigma$ that satisfies \eqref{A:sigma-prop} for some $r\geq 1.$ We also fix $p\geq 1$ directional derivatives $d_1,\ldots, d_p$ as in \eqref{E:der-deriv-def} and the corresponding vectors of iterated directional derivatives
\[
D^{\leq r}z_{i,\mA}^{(\ell+1)}:=\lr{D_\alpha^{J}z_{i;\alpha}^{(\ell+1)},\, \alpha \in \mA,\, J=(j_1,\ldots, j_p)\in \N^{p},\, \abs{J}\leq r}. 
\]
Theorem \ref{T:pert-exp} concerns, for each fixed $m,\ell\geq 1$, the expectation of a function $f$ of of the form
\[
f\lr{D^{\leq r}z_{1,\mA}^{(\ell+1)},\ldots, D^{\leq r}z_{m,\mA}^{(\ell+1)}},
\]
which depends on all directional derivatives in $d_i$ of order at most $r$ in any $m$ neuron pre-activations at layer $\ell+1.$ We seek to show that if $f$ is both continuous and a tempered distribution, then for all $q_*\geq 1$ we have
\begin{align}
\label{E:pert-exp-pf}  &  \E{f\lr{D^{\leq r}z_{1,\mA}^{(\ell+1)},\ldots, D^{\leq r}z_{m,\mA}^{(\ell+1)}}}= O(n^{-q_*-1})+\\ 
\notag  &+ \sum_{q=0}^{2q_*}\frac{(-1)^q}{2^q q!} \mathbb E \bigg[ \bigg\langle\bigg(\sum_{\substack{\abs{J},\abs{J'}\leq r\\ \alpha,\alpha'\in \mA}}\Delta_{\alpha\alpha'}^{JJ',(\ell)} \sum_{j=1}^m \partial_{D_{\alpha}^{J}z_{j;\alpha}}\partial_{D_{\alpha'}^{J'}z_{j;\alpha'}}\bigg)^qf\lr{D_{\mA}^{\leq r}z_{1},\ldots, D_{\mA}^{\leq r}z_{m}}\bigg\rangle_{\kappa^{(\ell)}} \bigg].
\end{align}
We remind the reader the notation in this formula. First, we continue to denoted by $\bk{\cdot}_{\kappa^{(\ell)}}$ the expectation with respect to a collection of centered jointly Gaussian random vectors 
\[
D_{\mA}^{\leq r} z_{i} = \lr{D_{\alpha}^Jz_{i;\alpha},\,\alpha\in \mA,\, \abs{J}\leq r}
\]
with the same covariance
\[
\Cov\lr{D_{\alpha_1}^{J_1} z_{i_1;\alpha_1},\, D_{\alpha_2}^{J_2} z_{i_2;\alpha_2}}= \Cov\lr{D_{\alpha_1}^{J_1} z_{i_1;\alpha_1}^{(\ell)},\, D_{\alpha_2}^{J_2} z_{i_2;\alpha_2}^{(\ell)}}=\delta_{i_1i_2} \kappa_{\alpha_1\alpha_2}^{J_1J_2,(\ell)}
\]
as the true vectors of derivatives $D_{\mA}^{\leq r} z_{i;\mA}^{_{(\ell)}}$ in each component separately but zero covariance for different $i$. Second,
\begin{equation}\label{E:G-def}
\kappa^{(\ell)} = \E{ \Sigma^{\leq r, (\ell)}},\qquad  \Sigma^{\leq r, (\ell)}= \lr{D_{\alpha}^{J}D_{\alpha'}^{J'}\Sigma_{\alpha\alpha'}^{(\ell)}}_{\substack{\abs{J},\abs{J'}\leq r\\ \alpha,\alpha'\in \mA}},
\end{equation}
is an average of the conditional covariances
\[
D_{\alpha}^{J}D_{\alpha'}^{J'}\Sigma_{\alpha\alpha'}^{(\ell)} := \Cov\lr{D_{\alpha}^{J}z_{1;\alpha}^{(\ell+1)},\,D_{\alpha'}^{J'} z_{1;\alpha'}^{(\ell+1)}~|~\mF^{(\ell)}}=D_{\alpha}^{J}D_{\alpha'}^{J'}\left\{ C_b+\frac{C_W}{n_\ell}\sum_{j=1}^{n_\ell} \sigma\lr{z_{j;\alpha}^{(\ell)}}\sigma\lr{z_{j;\alpha'}^{(\ell)}} \right\}.
\]
Finally, $\Delta_{\alpha\alpha'}^{JJ',(\ell)}$ measures the corresponding fluctuations:
\[
\Delta_{\alpha\alpha'}^{JJ',(\ell)}:= D_{\alpha}^{J}D_{\alpha'}^{J'}\Sigma_{\alpha\alpha'}^{(\ell)} - \E{D_{\alpha}^{J}D_{\alpha'}^{J'}\Sigma_{\alpha\alpha'}^{(\ell)}},
\]
and we collect $\Delta_{\alpha\alpha'}^{JJ',(\ell)}$ into a matrix as follows:
\[
\Delta^{\leq r,(\ell)}:=\lr{\Delta_{\alpha\alpha'}^{JJ',(\ell)}}_{\substack{\abs{J},\abs{J'}\leq r\\ \alpha,\alpha'\in \mA}}.
\]
Our first step is to note that since the weights and biases in layer $\ell+1$ are Gaussian, independent of one another, and independent of the sigma algebra $\mF^{(\ell)}$ generated by all prior weights and biases, we may write
\begin{align*}
    \E{f\lr{D^{\leq r}z_{1;\mA}^{(\ell+1)},\ldots, D^{\leq r}z_{m;\mA}^{(\ell+1)}}}=     \E{f\lr{\lr{\Sigma^{\leq r,(\ell)}}^{1/2}Z_1,\ldots, \lr{\Sigma^{\leq r,(\ell)}}^{1/2}Z_m}},  
\end{align*}
where $Z_1,\ldots, Z_m$ are standard Gaussians which are independent of one another and of $\Sigma^{\leq r,(\ell)}$. Moreover, because $\Sigma^{\leq r,(\ell)}$ is PSD the relation \eqref{E:G-def} ensures 
\[
\ker(\kappa^{(\ell)})\subseteq\ker( \Sigma^{\leq r,(\ell)})  \text{ a.s.} 
\]
By decomposing
\[
Z_i  = Z_{i;||} + Z_{i;\perp},\qquad  Z_{i;||}\in \ker(\kappa^{(\ell)}), \,  Z_{i;\perp}\in \ker(\kappa^{(\ell)})^\perp
\]
and writing $\Sigma_{\perp}^{\leq r,(\ell)}$ for the compression of $\Sigma^{\leq r,(\ell)}$ onto $\ker(\kappa^{(\ell)})^\perp$ we obtain by a slight abuse of notation that 
\begin{align*}
    \E{f\lr{D^{\leq r}z_{1;\mA}^{(\ell+1)},\ldots, D^{\leq r}z_{m;\mA}^{(\ell+1)}}} =     \E{f\lr{\lr{\Sigma_{\perp}^{\leq r,(\ell)}}^{1/2}Z_{1,\perp},\ldots, \lr{\Sigma_{\perp}^{\leq r,(\ell)}}^{1/2}Z_{m,\perp}}}.
\end{align*}
The key point is now that $Z_{i;\perp}$ are standard Gaussian vectors supported on a subspace on which $\kappa^{(\ell)}$ is strictly positive definite and that $\Sigma_{\perp}^{\leq r,(\ell)}$ maps this subspace into itself. Consider the event 
\[
S_n=\set{\abs{\Delta_{\alpha\alpha'}^{JJ',(\ell)}} < n^{-1/4},\, \alpha,\alpha'\in \mA,\, \abs{J},\abs{J'}\leq r} = \set{\norm{\kappa^{(\ell)}-\Sigma^{\leq r,(\ell)}}_\infty < n^{-1/4}}.
\]
Note that, by applying Theorem \ref{T:smooth-cumulants} and arguing exactly as in Lemma \ref{L:moment-bounds}, we find that since $\Delta_{\alpha\alpha'}^{JJ',(\ell)}$ are centered collective observables,
\[
\mathbb P(S_n^c)=O(n^{-\infty}).
\]
Since $f$ is a tempered distribution and a continuous function, its expectation against any Gaussian is finite and we therefore have 
\begin{align*}
    &\E{f\lr{D^{\leq r}z_{1;\mA}^{(\ell+1)},\ldots, D^{\leq r}z_{m;\mA}^{(\ell+1)}}}=\E{{\bf 1}_{S_n} f\lr{\lr{\Sigma_{\perp}^{\leq r,(\ell)}}^{1/2}Z_{1,\perp},\ldots, \lr{\Sigma_{\perp}^{\leq r,(\ell)}}^{1/2}Z_{m,\perp}}},
\end{align*}
plus an error of size $O(n^{-\infty})$. Let us denote by $\widehat{f}(\xi_1,\ldots, \xi_m)$ the Fourier transform of $f$ and abbreviate
\[
\xi=\lr{\xi_1,\ldots, \xi_m},\qquad  \norm{\xi}^2:=\sum_{i=1}^m\norm{\xi_i}^2,\qquad d\xi := d\xi_1\cdots d\xi_m
\]
For a $C>0$ that we will choose later let us write
\begin{align*}
 &\E{{\bf 1}_{S_n} f\lr{\lr{\Sigma_{\perp}^{\leq r,(\ell)}}^{1/2}Z_{1,\perp},\ldots, \lr{\Sigma_{\perp}^{\leq r,(\ell)}}^{1/2}Z_{m,\perp}}}\\ 
     &\quad =\int \widehat{f}(\xi) \E{{\bf 1}_{S_n} \exp\left[-\frac{1}{2}\sum_{i=1}^m \xi_i^T \Sigma_{\perp}^{\leq r, (\ell)}\xi_i\right]} d\xi\\
     &\quad=\int_{\substack{\norm{\xi}^2> C\log(n)}}\widehat{f}(\xi) \E{{\bf 1}_{S_n} \exp\left[-\frac{1}{2}\sum_{i=1}^m \xi_i^T \Sigma_{\perp}^{\leq r, (\ell)}\xi_i\right]} d\xi\\
     &\quad+\int_{\substack{\norm{\xi}^2\leq C\log(n)}}\widehat{f}(\xi) \E{{\bf 1}_{S_n} \exp\left[-\frac{1}{2}\sum_{i=1}^m \xi_i^T \Sigma_{\perp}^{\leq r, (\ell)}\xi_i\right]} d\xi\\
     &\quad =: I_C +II_C.
 \end{align*}
Let us now check that 
\begin{equation}\label{E:high-freq}
    \forall q\geq 1~\exists C=C(q)\text{ s.t. }I_C= O(n^{-q}).
\end{equation}
By the fundamental structure theorem of tempered distributions (see e.g. \cite{friedlander1998introduction}), there exist bounded continuous function $u_{I,J}$ and an integer $o(f)$, called the order of $f$, such that \begin{equation}\label{E:fhat-form}
\widehat{f}(\xi) = \sum_{\substack{I,J\\\abs{I,J}\leq o(f)}} \xi^I D^J u_{I,J}(\xi),    
\end{equation}
where where the derivatives $D^J$ with respect to $\xi_1,\ldots, \xi_m$ are defined in the weak sense and $\xi$ raised to a multi-index $I$ denotes the corresponding monomial. Thus, we may use \eqref{E:fhat-form} to write 
\begin{align*}
    \abs{I_C}&=\abs{ \sum_{\substack{I,J\\\abs{I},\abs{J}\leq o(f)}}\int_{\substack{\norm{\xi}^2> C\log(n)}}    u_{I,J}(\xi)D^J \lr{\xi^I\E{{\bf 1}_{S_n}  \exp\left[-\frac{1}{2}\sum_{i=1}^m \xi_i^T \Sigma_{\perp}^{\leq r, (\ell)}\xi_i\right]}} d\xi}
    \\
    &\leq  \sum_{\substack{I,J\\\abs{I},\abs{J}\leq o(f)}}\norm{u_{I,J}}_\infty \int_{\substack{ \norm{\xi}^2> C\log(n)}}  \E{{\bf 1}_{S_n} \abs{p_{o(f)}(\xi)} \exp\left[-\frac{1}{2}\sum_{i=1}^m \xi_i^T \Sigma_{\perp}^{\leq r, (\ell)}\xi_i\right]}d\xi,
\end{align*}
where $p_{o(f)}$ is some polynomial of degree at most $2o(f)$ in the variables $\xi_1,\ldots, \xi_m$ in which the coefficients are themselves polynomials the entries of $\Sigma_{\perp}^{\leq r, (\ell)}$. On the event $S_n$, entries of $\Sigma_{\perp}^{\leq r, (\ell)}$ are uniformly bounded in $n$ since by Theorem \ref{T:smooth-cumulants} the entries of $\kappa^{(\ell)}$ are uniformly bounded in $n$ and the event $S_n$ guarantees that the difference $\kappa^{(\ell)}-\Sigma^{\leq r, (\ell)}$ is small for all large $n$. In particular, for some $T>0$ we may write
\[
{\bf 1}_{S_n}\abs{p_{o(f)}\lr{\xi_1,\ldots, \xi_m}}\leq {\bf 1}_{S_n} T\lr{1+\norm{\xi}^2}^{o(f)}.
\]
Note moreover that for all $n$ sufficiently large, on the event $S_n$, we have that for some $\lambda_0>0$ and any $\xi\in \ker(\kappa^{(\ell)})^\perp$ that
\[
\frac{1}{2}\xi^T \Sigma_{\perp}^{\leq r, (\ell)}\xi \geq \lambda_0 \norm{\xi}^2.
\] 
Hence, passing to polar coordinates, we find that 
\begin{align*}
    I_C&\leq  T\sum_{\substack{I,J\\\abs{I},\abs{J}\leq o(f)}}\norm{u_{I,J}}_\infty \int_{\substack{r^2> C\log(n)}}  (1+r^2)^{o(f)+mN(r,p)\abs{\mA}-1} e^{-\lambda_0r^2} dr,
\end{align*}
where we recall that $N(r,p)$ is the number of derivatives of order at most $r$ in the $p$ vector fields $d_1,\ldots, d_p$. Thus, we conclude that that for any $q\geq 1$ there indeed exists $C=C(q),C'=C'(q)$ such that
\[
I_{C}\leq C'n^{-q},
\]
confirming \eqref{E:high-freq}. We therefore define $C:=C(q_*+1)$ and rewrite $II_C$ as follows:
 \begin{align*}
 II_C&=\int_{\substack{\norm{\xi}^2\leq C\log(n)}}\widehat{f}(\xi) \E{{\bf 1}_{S_n} \exp\left[-\frac{1}{2}\sum_{i=1}^m \xi_i^T\Sigma_{\perp}^{\leq r, (\ell)}\xi_i\right]} d\xi\\  &=\int_{\substack{\norm{\xi}^2\leq C\log(n)}}\widehat{f}(\xi) \exp\left[-\frac{1}{2}\sum_{i=1}^m \xi_i^T \kappa^{(\ell)}\xi_i\right]\E{{\bf 1}_{S_n} \exp\left[-\frac{1}{2}\sum_{i=1}^m \xi_i^T \Delta^{\leq r, (\ell)}\xi_i\right]} d\xi.
 \end{align*}
 Note that on the event $S_n$ there exists $T>0$ so that 
 \[
 \sup_{\norm{\xi}^2\leq C\log(n)} \sum_{i=1}^{m}\xi_i^T\Delta_{\mA}^{\leq r, (\ell)}\xi_i\leq CTm\abs{\mA}^2\frac{\log(n)}{n^{1/4}}.
 \]
Hence, we may choose $Q^*=Q^*(q^*, C, \abs{\mA})\geq 1$ so that 
 \begin{align}
\label{E:exp-expand}\E{{\bf 1}_{S_n} \exp\left[-\frac{1}{2}\sum_{i=1}^m \xi_i^T \Delta_{\mA}^{\leq r, (\ell)}\xi_i\right]} &= \E{{\bf 1}_{S_n} \sum_{q=0}^{Q_*} \frac{(-1)^q}{2^qq!}\lr{\sum_{i=1}^m \xi_i^T \Delta_{\mA}^{\leq r, (\ell)}\xi_i}^q}+O(n^{-q_*-1}).
 \end{align}
We thus conclude that $II_C$ equals
 \begin{align*}
      &\sum_{q=0}^{Q_*}\frac{(-1)^q}{2^qq!}\int_{\substack{ \norm{\xi}^2\leq C\log(n)}}\E{\lr{\sum_{i=1}^m \xi_i^T \Delta_{\mA}^{\leq r,(\ell)}\xi_i}^q}\widehat{f}(\xi)\exp\left[-\frac{1}{2}\sum_{i=1}^m \xi_i^T \kappa^{(\ell)}\xi_i\right] d\xi
 \end{align*}
plus an error of size $O(n^{-q_*-1})$. Note also that by applying Lemma \ref{L:moment-bounds} we have
 \[
\E{{\bf 1}_{S_n} \lr{\sum_{i=1}^m \xi_{i}^T \Delta_{\mA}^{(\ell)}\xi_{i}}^q} = O\lr{\norm{\xi}^{2q} n^{-\lceil \frac{q}{2}\rceil}}.
 \]
Hence, since $\widehat{f}$ is a tempered distribution and $\kappa^{(\ell)}$ is strictly positive definite on $\ker(\kappa^{(\ell)})^\perp$, the terms corresponding to $2q^*+1\leq q\leq Q^*$ in \eqref{E:exp-expand} are of size $O(n^{-q_*-1})$. Moreover, applying the same reasoning as we used to bound $I_C$, by incurring another error of order $O(n^{-q_*-1})$ we may drop the restriction in $II_C$ that $\norm{\xi}^2\leq C\sqrt{\log(n)}$. All together, $II_C$ therefore equals
 \begin{align*}
      &\sum_{q=0}^{2q_*}\frac{(-1)^q}{2^qq!}\int\E{\lr{\sum_{i=1}^m \xi_i^T \Delta_{\mA}^{\leq r,(\ell)}\xi_i}^q}\widehat{f}(\xi)\exp\left[-\frac{1}{2}\sum_{i=1}^m \xi_i^T \kappa^{(\ell)}\xi_i\right] d\xi.
 \end{align*}
plus an error of size $O(n^{-q_*-1})$. Using that multiplication by components of $\xi_i$ acting on the Fourier transform corresponds to differentiation of with respect to the variables $\{D_\alpha^J z_{i;\alpha},\, \alpha\in \mA,\, \abs{J}\leq r\}$, yields the desired expression \eqref{E:pert-exp-pf} and completes the proof of Theorem \ref{T:pert-exp}. \hfill $\square$


\section{Proof of Corollary \ref{C:cumulants-vals}}\label{S:cumulants-recs-pf}
The goal of this section is to derive recursions for 
\[
\kappa_{2k;\alpha}^{(\ell+1)} = \kappa_{k}\big(\underbrace{\Delta_{\alpha\alpha}^{(\ell)},\ldots, \Delta_{\alpha\alpha}^{(\ell)}}_{k\text{ times}}\big),
\]
where we defined $\Delta_{\alpha\alpha}^{_{(\ell)}}$ in \eqref{E:Delta-def}. Let us write
\[
X_j:=\sigma\lr{z_{j;\alpha}^{(\ell)}}^2 - \E{\sigma\lr{z_{j;\alpha}^{(\ell)}}^2}
\]
so that
\[
\Delta_{\alpha\alpha}^{(\ell)} = \frac{C_W}{n_{\ell}}\sum_{j=1}^{n_\ell} X_j.
\]
By symmetry, we then have
\begin{align}
\label{E:k4-form}\kappa_{4;\alpha}^{(\ell+1)} &= \E{\lr{\Delta_{\alpha\alpha}^{(\ell)}}^2}= \frac{C_W^2}{n_{\ell}}\E{X_1^2} + C_W^2\lr{1-n_\ell^{-1}}\E{X_1X_2} \\
\notag\kappa_{6;\alpha}^{(\ell+1)} &= \E{\lr{\Delta_{\alpha\alpha}^{(\ell)}}^3}\\  
\label{E:k6-form}&= \frac{C_W^3}{n_{\ell}^2}\E{X_1^3} +\frac{3C_W^3}{n_{\ell}}\lr{1-\frac{1}{n_\ell}}\E{X_1^2X_2}+ C_W^3\lr{1-\frac{1}{n_{\ell}}}\lr{1-\frac{2}{n_{\ell}}}\E{X_1X_2X_3}\\
\notag \kappa_{8;\alpha}^{(\ell+1)} &= \E{\lr{\Delta_{\alpha\alpha}^{(\ell)}}^4}-3\E{\lr{\Delta_{\alpha\alpha}^{(\ell)}}^2}^2  = \frac{C_W^4}{n_{\ell}^3}\lr{\E{X_1^4}-3\E{X_1^2}^2}\\
\notag &\qquad+ \frac{C_W^4}{n_\ell^2}\lr{1-\frac{1}{n_\ell}} \left[ \binom{4}{2}\left\{\E{X_1^2X_2}^2 - \E{X_1^2}\E{X_2^2}-2\E{X_1X_2}^2\right\}\right.\\
\notag &\qquad \qquad \qquad \qquad\qquad\left.+ \binom{4}{1}\bigg\{ \E{X_1^3X_2}-\E{X_1^3}\E{X_2} - 2\E{X_1^2}\E{X_1X_2}\bigg\}\right]\\
\notag &\qquad + \frac{C_W^4}{n_\ell}\lr{1-\frac{1}{n_{\ell}}}\lr{1-\frac{2}{n_\ell}}\binom{4}{2}\left\{\E{X_1^2X_2X_3}-\E{X_1^2}\E{X_1X_2}-2\E{X_1X_2}^2\right\}\\
\label{E:k8-form}&\qquad +C_W^4\lr{1-\frac{1}{n_\ell}} \lr{1-\frac{2}{n_\ell}} \lr{1-\frac{3}{n_\ell}}\E{X_1X_2X_3X_4}.
\end{align}
To evaluate the mixed moments of $X_{i}$ that appear in \eqref{E:k4-form}-\eqref{E:k8-form}, we use Theorem \ref{T:pert-exp} in the case $g\equiv 1, r=0,q_*=1$. In this setting, if $f$ is a continuous function and a tempered distribution, we find
\begin{align}
\label{E:f-expand-order-4}\E{f\lr{z_{1;\alpha}^{(\ell)},\ldots, z_{m;\alpha}^{(\ell)}}} &= \bk{f\lr{z_{1;\alpha},\ldots, z_{m;\alpha}}}_{\kappa^{(\ell)}}\\
\notag &+ \frac{\kappa_{4;\alpha}^{(\ell)}}{2^2\cdot 2!}\bk{\lr{\sum_{j=1}^m \partial_{z_{j;\alpha}}^2}^2 f \lr{z_{1;\alpha},\ldots, z_{m;\alpha}}}_{\kappa^{(\ell)}}\\
\notag &+\frac{\kappa_{6;\alpha}^{(\ell)}}{2^3\cdot 3!} \bk{\lr{\sum_{j=1}^m \partial_{z_{j;\alpha}}^2}^3 f \lr{z_{1;\alpha},\ldots, z_{m;\alpha}}}_{\kappa^{(\ell)}}\\
\notag &+\frac{\kappa_{8;\alpha}^{(\ell)}+3\lr{\kappa_{4;\alpha}^{(\ell)}}^2}{2^4\cdot 4!} \bk{\lr{\sum_{j=1}^m \partial_{z_{j;\alpha}}^2}^4 f \lr{z_{1;\alpha},\ldots, z_{m;\alpha}}}_{\kappa^{(\ell)}} + O(n^{-4}).
\end{align}
We remind the reader that, by definition, $z_{1;\alpha},\ldots, z_{m;\alpha}$ are iid centered Gaussians with variance $\kappa_{\alpha\alpha}^{_{(\ell)}}$. Since the derivations of \eqref{E:k4-rec}-\eqref{E:k8-rec} are very similar, let us give the details for only cases of $\kappa_{4;\alpha}^{(\ell)}$ and $\kappa_{6;\alpha}^{(\ell)}$. We have, using \eqref{E:f-expand-order-4}, that
\begin{align}
\notag \kappa_{\alpha\alpha}^{(\ell+1)} = \E{\lr{\Delta_{\alpha\alpha}^{(\ell)}}^2}&=\frac{C_W^2}{n_\ell}\lr{\bk{\sigma^4}_{\kappa_{\alpha\alpha}^{(\ell)}}-\bk{\sigma^2}_{\kappa_{\alpha\alpha}^{(\ell)}}^2}\\
\label{E:kappa4-formula}&+C_W^2\lr{1-n_\ell^{-1}} \lr{\bk{X_1}_{\kappa_{\alpha\alpha}^{(\ell)}}^2 + \frac{1}{4}\bk{\partial^2 \sigma^2}_{\kappa_{\alpha\alpha}^{(\ell)}}^2\kappa_{4;\alpha}^{(\ell)}}+O(n^{-2}).    
\end{align}
Next, we will show in \S \ref{S:S2-recs} below that
\[
\kappa_{\alpha\alpha}^{(\ell)} = K_{\alpha\alpha}^{(\ell)}+O(n^{-1}). 
\]
Moreover, note that since $x_\alpha\neq 0$, we have $K_{\alpha\alpha}^{(\ell)}$ is non-zero. Hence, Gaussian integration by parts yields that for any measureable polynomially bounded $f$ we have 
\begin{equation}\label{E:replace}
\bk{f}_{\kappa_{\alpha\alpha}^{(\ell)}} = \bk{f}_{K_{\alpha\alpha}^{(\ell)}} + O(n^{-1}).    
\end{equation}
Note also that for all $i$ by applying \eqref{E:f-expand-order-4} we have
\begin{equation}\label{E:X-mean}
\bkal{X_i} = - \frac{1}{8}\kappa_{4;\alpha}^{(\ell)}\bkkl{\partial^4X_1}+O(n^{-2}).    
\end{equation}
Thus, recalling the definition \eqref{E:chi-para-def} of $\chi_{||;\alpha}^{(\ell)}$ the estimate \eqref{E:kappa4-formula} immediately yields \eqref{E:k4-rec}. Next, using \eqref{E:f-expand-order-4} as well as \eqref{E:replace} and \eqref{E:X-mean} that
\begin{align*}
   C_W^3 \E{X_1^3} = C_W^3\bk{X_1^3}_{K_{\alpha\alpha}^{(\ell)}}+O(n^{-1}) = T_{0,3;\alpha}^{(\ell)} + O(n^{-1}).
\end{align*}
Further, we seek to evaluate $\E{X_1^2X_2}$ up to errors of size $O(n^{-2})$. We apply  \eqref{E:f-expand-order-4} as well as \eqref{E:replace} and \eqref{E:X-mean} to obtain
\begin{align*}
     C_W^3\E{X_1^2X_2}&=  C_W^3\bk{X_1^2}_{\kappa^{(\ell)}}\bkal{X_2}+ O(n^{-2})\\
    &+\frac{ C_W^3}{8}\kappa_{4;\alpha}^{(\ell)}\left[\bkkl{X_1^2}\bkkl{\partial^4 X_2} + 2 \bkkl{\partial^2 X_1^2}\bkkl{\partial^2 X_2}\right]\\
     &+\frac{1}{2}T_{2,2;\alpha}^{(\ell)}\chi_{||;\alpha}^{(\ell)}\kappa_{4;\alpha}^{(\ell)}+O(n^{-2}).
\end{align*}
Finally, we must evaluate $\E{X_1X_2X_3}$ up to errors of size $O(n^{-3})$. Again using \eqref{E:X-mean}, we find
\begin{align*}
    C_W^3\E{X_1X_2X_3} &= C_W^3\bkal{X_1}\bkal{X_2}\bkal{X_3}\\
    &+\frac{C_W^3}{8}\kappa_{4;\alpha}^{(\ell)}\left[6\bkal{X_1}\bkal{\partial^2X_2}^2+O(n^{-2})\right]\\
    &+\frac{C_W^3}{48}\kappa_{6;\alpha}^{(\ell)}\left[6\bkal{\partial^2 X_1}^3+O(n^{-1})\right]+O(n^{-3})\\
    &=-\frac{3}{8}T_{4,1;\alpha}^{(\ell)}\lr{\chi_{||;\alpha}^{(\ell)}\kappa_{4;\alpha}^{(\ell)}}^2\bkkl{\partial^4 X_1} + \lr{\chi_{||;\alpha}^{(\ell)}}^3 \kappa_{6;\alpha}^{(\ell)} + O(n^{-3}).
\end{align*}
This completes the derivation of the recursion for $\kappa_{6;\alpha}^{(\ell)}$. 
\hfill $\square$

\section{Recursions for $2^{nd},4^{th}$ Cumulants of $z_{i;\alpha}^{(\ell)}$ and its Derivatives}\label{S:2-4-deriv-recs}
In this section, we consider a random depth $L$ fully connected network with input dimension $n_0$, hidden layer widths, $n_1,\ldots, n_L$, and non-linearity $\sigma$ satisfying Assumption \eqref{A:sigma-prop}. We also fix a network input $x_\alpha\in \R^{n_0}$.
Our goal in this section is to give a recursive description of the second and fourth cumulants of $z_{i;\alpha}^{_{(\ell+1)}}$ and its derivatives in terms of those of $z_{i;\alpha}^{_{(\ell)}}$. The starting point is to establish recursions for 
\[
K_{(ij)}^{(\ell)}: = \partial_{x_{i;\alpha_1}}\partial_{x_{j;\alpha_2}}K_{\alpha_1\alpha_2}^{(\ell)}\bigg|_{x_{\alpha_1}=x_{\alpha_2}=x_\alpha}
\]
where $i,j=1,\ldots, n_0$ (note that $K_{(ij)}^{_{(\ell)}}=K_{\alpha\alpha}^{_{(\ell)}}$).  This is done in Lemma \ref{L:2pt-deriv-recs-leading} in \S \ref{S:K-recs}.
We then proceed in \S \ref{S:k4-recs} to detail a number of recursions for the fourth cumulants of $z_{i;\alpha}^{_{(\ell)}}$ and its first derivatives. Finally, \S \ref{S:S2-recs} provides recursions to first order in $1/n$ for the finite width correction of the variance of the covariance of $z_{i;\alpha}^{_{(\ell)}}$ and its first derivative. All these recursions hold for any $\sigma$ satisfying \eqref{A:sigma-prop}. In \S \ref{S:tanh-rec-solutions} we will solve these recursions for $\sigma$ belonging to the $K_*=0$ universality class defined in \S \ref{S:crit-univ} and will then use these solutions to prove Theorem \ref{T:evgp}.

\subsection{Notation}\label{S:rec-notation}
We will abbreviate 
\[
d_i z_{i;\alpha}^{(\ell)}:= \partial_{x_{i;\alpha}}z_{i;\alpha}^{(\ell)},\qquad i=1,\ldots, n_0.
\]
Moreover, to keep the notation as uniform as possible, we will also write
\begin{equation}\label{E:d0-def}
d_0 := \text{identity}    
\end{equation}
so that $d_0 z_{i;\alpha}^{_{(\ell)}}=z_{i;\alpha}^{_{(\ell)}}$. With this notation, let us agree to denote by 
\begin{equation}\label{E:kappa-deriv-def}
\kappa_{(j_1j_2)}^{(\ell)}:=\Cov\lr{d_{j_1}z_{i;\alpha}^{(\ell)},d_{j_2}z_{i;\alpha}^{(\ell)}}    
\end{equation}
the finite width covariance of $z_{i;\alpha}^{_{(\ell)}}$ and its derivatives and by 
\[
K_{(j_1j_2)}^{(\ell)}:= \lim_{n\gives \infty} \kappa_{(j_1j_2)}^{(\ell)}
\]
the corresponding infinite width covariance. Note that with this notation we have $K_{(00)}^{_{(\ell)}}=K_{\alpha\alpha}^{_{(\ell)}}$. We also write
\begin{equation}\label{E:S-def}
S_{(ij)}^{(\ell)}:=K_{(ij)}^{(\ell)} -\kappa_{(ij)}^{(\ell)}.
\end{equation}
Further, for a function $f$ we will write
\begin{equation}\label{E:f-avg-def}
\bkl{f}:=\bk{f(d_jz_i),\quad i=1,\ldots,m,\, j=0,1,\ldots, n_0}_{K^{(\ell)}},     
\end{equation}
for the expectation of a function $f$ with respect to the centered Gaussian satisfying 
\[
\Cov\lr{d_{j_1}z_{i_1}, d_{j_2}z_{i_2}} = \delta_{i_1i_2}K_{(j_1j_2)}^{(\ell)},\qquad j_1,j_2=0,1,2.
\]
For the sake of recording more manageable formulas, we will often drop the argument $z$ in expressions such as $\sigma(z)$ so that
\[
\bkl{\sigma\sigma' d_1z (d_2 z)^2} = \bkl{\sigma(z)\sigma'(z) d_1z (d_2 z)^2} 
\]
as so on. Note that in the expectation \eqref{E:f-avg-def} variables such as $z_{i_1},\, z_{i_2}$ with different neural indices are independent. Thus, for example
\[
\bkl{\sigma\sigma'} = \bkl{\sigma(z)\sigma'(z)}= \int_{\R} \sigma\lr{z\lr{K_{\alpha\alpha}^{_{(\ell)}}}^{1/2}}\sigma'\lr{z\lr{K_{\alpha\alpha}^{_{(\ell)}}}^{1/2}} e^{-z^2/2} \frac{dz}{\sqrt{2\pi}}.
\]
In analogy with \eqref{E:f-avg-def}, we will denote by $\bk{f(d_jz_i,\, i=1,\ldots m, \, j=0,1,\ldots,n_0)}_{\kappa^{(\ell)}}$ the expectation of $f$ with respect to a centered Gaussian satisfying 
\[
\Cov\lr{d_{j_1}z_{i_1}, d_{j_2}z_{i_2}} = \delta_{i_1i_2}\kappa_{(j_1j_2)}^{(\ell)},\qquad j_1,j_2=0,1,2.
\]
in which all neurons are again independent centered Gaussians but this time with the finite width covariance. Further, when $f$ depends only a single variable $z=d_0z_i$, we will denote by $\partial$ derivatives with respect to $z$, so that for any function $f(z)$ we have
\[
\partial^i f =\partial^i f(z) = \frac{\partial^i}{\partial z^i} f(z).
\]
Here and elsewhere all derivatives are in the weak sense if $f$ is not sufficiently smooth. Finally, as in Corollary \ref{C:cumulants-vals} and \S \ref{S:crit}, we also set
\begin{align*}
\chi_{||}^{(\ell)}&:= \chi_{||;\alpha}^{(\ell)}=\chi_{||}^{(\ell)}(K_{\alpha\alpha}^{(\ell)})= \frac{C_W}{2}\bkl{\partial^2 \sigma^2}\\    
\chi_{\perp}^{(\ell)}&:= \chi_{\perp;\alpha}^{(\ell)}=\chi_{\perp}^{(\ell)}(K_{\alpha\alpha}^{(\ell)})= C_W\bkl{(\sigma')^2}.   
\end{align*}

\subsection{Recursions at Infinite Width}\label{S:K-recs}
The following Lemma gives recursions with respect to $\ell$ for the infinite width covariance $K_{(ij)}^{(\ell)}$.
\begin{lemma}\label{L:K-deriv-recs} 
We have
\begin{align}
\label{E:K00-rec}    K_{(00)}^{(\ell+1)} &= C_b+C_W\bkl{\sigma^2}\\
\label{E:K10-rec}       K_{(10)}^{(\ell+1)} &=\chi_{||}^{(\ell)}K_{(10)}^{(\ell)}\\
\label{E:K11-rec}       K_{(11)}^{(\ell+1)} &=C_W\bk{\partial^2 (\sigma')^2}_{(\ell)}\lr{K_{(10)}^{(\ell)}}^2+\chi_{\perp}^{(\ell)}K_{(11)}^{(\ell)}\\
\label{E:K12-rec}       K_{(12)}^{(\ell+1)} &=C_W\bk{\partial^2 (\sigma')^2}_{(\ell)}K_{(10)}^{(\ell)}K_{(20)}^{(\ell)}+\chi_{\perp}^{(\ell)}K_{(12)}^{(\ell)}.
\end{align}
\end{lemma}
\begin{proof}
The proof of these recursions consists of integrating out the weights and biases in layer $\ell$ and sometimes using Gaussian integration by parts. For instance, 
\[
 K_{(00)}^{(\ell+1)} = \lim_{n\gives \infty}\E{\lr{z_{i;\alpha}^{(\ell+1)}}^2}=\lim_{n\gives \infty}\lr{C_b+C_W\E{\sigma\lr{z_{i;\alpha}^{(\ell)}}^2}} = C_b + C_W\bkl{\sigma^2},
\]
gives \eqref{E:K00-rec}, which is a restatement of the recursion \eqref{E:K-rec} we already say in Theorem \ref{T:iwl}. Further, using the definition of $K_{(10)}^{_{(\ell+1)}}$, Theorem \ref{T:iwl}, and integrating by parts we have 
\begin{align*}
 K_{(10)}^{(\ell+1)} &= \lim_{n\gives \infty}\E{z_{i;\alpha}^{(\ell+1)} d_1z_{i;\alpha}^{(\ell+1)}  }\\
 &=\lim_{n\gives \infty} C_W\E{\sigma\lr{z_{i;\alpha}^{(\ell)}}\sigma'\lr{z_{i;\alpha}^{(\ell)}} d_1z_{i;\alpha}^{(\ell)}}\\
 &= C_W\bkl{\sigma(z)\sigma'(z) d_1z}\\
 &=C_W K_{(10)}^{(\ell)}\bkl{\sigma(z)\sigma'(z)}\\
 &=\chi_{||}^{(\ell)} K_{(10)}^{(\ell)}.
\end{align*}
This proves \eqref{E:K10-rec}. Similar reasoning gives 
\begin{align*}
     K_{(11)}^{(\ell+1)}&= C_W\bkl{(\sigma'(z) d_1z)^2}\\
     &=C_W\left[K_{(11)}^{(\ell)}\bkl{(\sigma')^2}+\lr{K_{(10)}^{(\ell)}}^2 \bkl{\partial^2 (\sigma')^2}\right]\\
     &=C_W\bk{\partial^2 (\sigma')^2}_{(\ell)}\lr{K_{(10)}^{(\ell)}}^2+\chi_{\perp}^{(\ell)}K_{(11)}^{(\ell)},
\end{align*}
proving \eqref{E:K11-rec}. The derivation of \eqref{E:K12-rec} is analogous. 
\end{proof}

\subsection{4th Cumulant Recursions at Finite Width}\label{S:k4-recs}
In this section we obtain a recursive description for the fourth cumulants of $z_{i;\alpha}^{_{(\ell)}}$ and its first derivatives. In keep with our convention \eqref{E:d0-def}, we will denote by $\partial_{x_{0;\alpha}}$ the identity operator. By explicitly integrating out the weights and biases in layer $\ell+1$ we find
\begin{align}
\notag &\kappa\lr{\partial_{x_{j_1;\alpha}}z_{i_1;\alpha}^{(\ell+1)},\partial_{x_{j_2;\alpha}}z_{i_2;\alpha}^{(\ell+1)},\partial_{x_{j_3;\alpha}}z_{i_3;\alpha}^{(\ell+1)},\partial_{x_{j_4;\alpha}}z_{i_4;\alpha}^{(\ell+1)}}\\
\label{E:kappa-reduce}&\qquad \qquad = \delta_{i_1i_2}\delta_{i_3i_4}\kappa_{(j_1j_2)(j_3j_4)}^{(\ell)}+ \delta_{i_1i_3}\delta_{i_2i_4}\kappa_{(j_1j_3)(j_2j_4)}^{(\ell)}+\delta_{i_1i_4}\delta_{i_2i_3}\kappa_{(j_1j_4)(j_2j_3)}^{(\ell)},
\end{align}
where 
\begin{align}
\notag \kappa_{(j_1j_2)(j_1'j_2')}^{(\ell)}&= \kappa\lr{\partial_{x_{j_1;\alpha}}z_{1;\alpha}^{(\ell+1)},\partial_{x_{j_2;\alpha}}z_{1;\alpha}^{(\ell+1)},\partial_{x_{j_1';\alpha}}z_{2;\alpha}^{(\ell+1)},\partial_{x_{j_2';\alpha}}z_{2;\alpha}^{(\ell+1)}}\\
\label{E:kappa-sigma-def}&=\Cov\lr{\partial_{x_{j_1;\alpha_1}} \partial_{x_{j_2;\alpha_2}}\Sigma_{\alpha_1\alpha_2}^{(\ell)}\bigg|_{\alpha_1=\alpha_2=\alpha},\,
\partial_{x_{j_1';\alpha_1}} \partial_{x_{j_2';\alpha_2}}\Sigma_{\alpha_1\alpha_2}^{(\ell)}\bigg|_{\alpha_1=\alpha_2=\alpha}}.
\end{align}
We remind the reader that the quantity $\Sigma_{\alpha_1\alpha_2}^{_{(\ell)}}$ appearing in \eqref{E:kappa-sigma-def} is the conditional covariance
\[
\Cov\lr{z_{i_1;\alpha_1}^{(\ell+1)},\, z_{i_2;\alpha_2}^{(\ell+1)}~|~\mF^{(\ell)}} =\delta_{i_1i_2}\Sigma_{\alpha_1\alpha_2}^{_{(\ell)}}= \delta_{i_2i_2}\lr{C_b+\frac{C_W}{n_\ell}\sum_{j=1}^{n_\ell}\sigma\lr{z_{j;\alpha_1}^{(\ell)}}\sigma\lr{z_{j;\alpha_2}^{(\ell)}}}
\]
that we introduced in \eqref{E:big-sig-def}. Our next result gives recursions with respect to $\ell$ directly for the covariances $\kappa_{(j_1j_2)(j_1'j_2')}^{_{(\ell)}}$ and hence via \eqref{E:kappa-reduce} for the fourth cumulants of $z_{i;\alpha}^{(\ell)}$ and its first derivatives as well.

\begin{proposition}\label{P:k4-deriv-recs}
We have
\begin{align*}
    \kappa_{(00)(00)}^{(\ell+1)} = \frac{C_W^2}{n_\ell}\lr{\bk{\sigma^4}_{(\ell)}-\bk{\sigma^2}_{(\ell)}^2}+\lr{\chi_{||}^{(\ell)}}^2   \kappa_{(00)(00)}^{(\ell)}+O(n^{-2}).
\end{align*}
Further, 
\begin{align}
\label{E:k1000-rec}    \kappa_{(10)(00)}^{(\ell+1)} &= \frac{C_W^2}{n_\ell}\lr{\bk{\sigma^3\sigma'd_1z}_{(\ell)}-\bk{\sigma^2}_{(\ell)}\bk{\sigma\sigma'd_1z}_{(\ell)}}\\
\notag    &+\frac{C_W^2}{8}\left[2\kappa_{(00)(00)}^{(\ell)}\bk{\partial^2\sigma^2}_{(\ell)}\bk{\partial^2(\sigma\sigma')d_1z}_{(\ell)}+4\kappa_{(10)(00)}^{(\ell)}\bk{\partial^2\sigma^2}_{(\ell)}\bk{\partial(\sigma\sigma')}_{(\ell)}\right]+O(n^{-2}).
\end{align}
Moreover,  
\begin{align*}
    \kappa_{(10)(10)}^{(\ell+1)} &= \frac{C_W^2}{n_\ell}\left[\bk{(\sigma\sigma'd_1z)^2}_{(\ell)}-\bk{\sigma\sigma'd_1z}_{(\ell)}^2\right]\\
    &+\frac{C_W^2}{8}\left[2\kappa_{(00)(00)}^{(\ell)}\bk{\partial^2 (\sigma\sigma') d_1z}^2+8\kappal_{(10)(00)}\bkl{\partial(\sigma\sigma')}\bkl{\partial^2(\sigma\sigma')d_1z}\right.\\
    &\left.\qquad \qquad +8 \kappa_{(10)(10)}^{(\ell)} \bk{\partial(\sigma\sigma')}_{(\ell)}^2\right]+O(n^{-2}).
\end{align*}
Similarly, 
\begin{align*}
    \kappa_{(10)(20)}^{(\ell+1)}&=\frac{C_W^2}{n_\ell}\left[\bkl{(\sigma\sigma')^2 d_1zd_2z}-\bkl{\sigma\sigma' d_1z}\bkl{\sigma\sigma' d_2z}\right]\\
    &+\frac{C_W^2}{8}\left[2\kappal_{(00)(00)}\bkl{\partial^2(\sigma\sigma')d_1z}\bkl{\partial^2(\sigma\sigma')d_2z} \right.\\
    &\qquad \qquad + 4\kappal_{(10)(00)}\bkl{\partial(\sigma\sigma')}\bkl{\partial^2(\sigma\sigma') d_2z}+4\kappal_{(20)(00)}\bkl{\partial(\sigma\sigma')}\bkl{\partial^2(\sigma\sigma') d_1z}\\
    &\qquad \qquad + \left.8\kappal_{(10)(20)}\bkl{\partial(\sigma\sigma')}^2\right]+O(n^{-2}).
\end{align*}
In addition,
\begin{align*}
    \kappa_{(11)(00)}^{(\ell+1)} &= \frac{C_W^2}{n_\ell}\left[\bk{(\sigma\sigma' d_1z)^2}_{(\ell)} - \bk{\sigma^2}_{(\ell)}\bk{(\sigma'd_1z)^2}_{(\ell)}\right]\\
    &+\frac{C_W^2}{8}\left[2\kappa_{(00)(00)}^{(\ell)}\bk{\partial^2 \sigma^2}_{(\ell)}\bk{\partial^2 (\sigma')^2 (d_1z)^2}_{(\ell)} +8 \kappa_{(10)(00)}^{(\ell)}\bk{\partial (\sigma')^2d_1z}_{(\ell)}\bk{\partial^2 \sigma^2}_{(\ell)}\right.\\
    &\qquad \qquad \left.4\kappa_{(11)(00)}^{(\ell)} \bk{(\sigma')^2}\bk{\partial^2 \sigma^2}\right]+O(n^{-2}).
\end{align*}
Also,
\begin{align*}
    \kappa_{(12)(00)}^{(\ell+1)}&=\frac{C_W^2}{n_\ell}\left[\bkl{(\sigma')^2\sigma^2 d_1z d_2z}-\bkl{(\sigma')^2d_1zd_2z}\bkl{\sigma^2}\right]\\
    &=\frac{C_W^2}{8}\left[2\kappal_{(00)(00)}\bkl{\partial^2(\sigma')^2 d_1zd_2z}\bkl{\partial^2\sigma^2}\right.\\
    &\qquad \qquad+4\kappal_{(10)(00)}\bkl{\partial (\sigma')^2 d_2z}\bkl{\partial^2\sigma^2}+4\kappal_{(20)(00)}\bkl{\partial (\sigma')^2 d_1z}\bkl{\partial^2\sigma^2}\\
    &\qquad \qquad +\left.4\kappal_{(12)(00)}\bkl{(\sigma')^2}\bkl{\partial^2\sigma^2}\right]+O(n^{-2}).
\end{align*}
Additionally,
\begin{align*}
    \kappa_{(11)(10)}^{(\ell+1)} &= \frac{C_W^2}{n_\ell}\left[\bk{\sigma(\sigma'd_1z)^3}_{(\ell)}-\bk{\sigma\sigma'd_1z}_{(\ell)}\bk{(\sigma'd_1z)^2}_{(\ell)}\right]\\
    &+\frac{C_W^2}{8}\left[2\kappa_{(00)(00)}^{(\ell)}\bk{\partial^2(\sigma \sigma')d_1z}_{(\ell)}\bkl{\partial^2 (\sigma')^2 (d_1z)^2}\right.\\
    &\qquad \qquad +4\kappa_{(10)(00)}^{(\ell)}\lr{\bkl{\partial(\sigma\sigma')}\bkl{\partial^2 (\sigma')^2 (d_1z)^2}+2\bkl{\partial^2(\sigma\sigma')d_1z}\bkl{\partial(\sigma')^2d_1z}}\\
    &\qquad \qquad +16\kappa_{(10)(10)}^{(\ell)}\bkl{\partial(\sigma\sigma')}\bkl{\partial(\sigma')^2 d_1z}  +4\kappa_{(11)(00)}^{(\ell)}\bkl{\partial^2(\sigma\sigma')d_1z}\bkl{(\sigma')^2}\\
    &\qquad \qquad \left.+8\kappa_{(11)(10)}^{(\ell)}\bkl{(\sigma')^2}\bkl{\partial(\sigma\sigma')}\right]+O(n^{-2}).
\end{align*}
Also,
\begin{align*}
    \kappa_{(12)(10)}^{(\ell+1)}&= \frac{C_W^2}{n_\ell}\left[\bkl{(\sigma')^3\sigma (d_1z)^2d_2z}-\bkl{(\sigma')^2d_1zd_2z}\bkl{\sigma\sigma'd_1z}\right]\\
    &=\frac{C_W^2}{8}\left[2\kappal_{(00)(00)}\bkl{\partial^2 (\sigma')^2 d_1z d_2z}\bkl{\partial^2(\sigma\sigma')d_1z}\right.\\
    &\qquad \qquad + 4\kappal_{(10)(00)}\left\{\bkl{\partial (\sigma')^2d_2z}\bkl{\partial^2(\sigma\sigma')d_1z}+\bkl{\partial^2(\sigma')^2 d_1zd_2z}\bkl{\partial(\sigma\sigma')}\right\}\\
    &\qquad \qquad +4\kappal_{(20)(00)}\bkl{\partial(\sigma')^2 d_1z}\bkl{\partial^2(\sigma\sigma')d_1z}+4\kappal_{(12)(00)}\bkl{(\sigma')^2}\bkl{\partial^2(\sigma\sigma')d_1z}\\
    &\qquad \qquad+ 8\kappal_{(10)(10)}\bkl{\partial (\sigma')^2 d_2z}\bkl{\partial (\sigma\sigma')}+8\kappal_{(10)(20)}\bkl{\partial (\sigma')^2 d_1z}\bkl{\partial(\sigma\sigma')}\\
    &\qquad \qquad \left.+ 8\kappal_{(12)(10)}\bkl{(\sigma')^2}\bkl{\partial(\sigma\sigma')}\right]+O(n^{-2}).
\end{align*}
Further,
\begin{align*}
    \kappa_{(11)(20)}^{(\ell+1)} &= \frac{C_W^2}{n_\ell}\left[\bkl{(\sigma')^2\sigma (d_1z)^2 d_2z}-\bkl{(\sigma' d_1z)^2}\bkl{\sigma\sigma'd_2z}\right]\\
    &+\frac{C_W^2}{8}\left[2\kappa_{(00)(00)}^{(\ell)}\bk{\partial^2(\sigma')^2(d_1z)^2}_{(\ell)}\bkl{\partial^2 (\sigma\sigma') d_2z}\right.\\
    &\qquad \qquad +8\kappal_{(10)(00)}\bkl{\partial^2 (\sigma')^2 d_1z}\bk{\partial^2(\sigma\sigma')d_2z}+4\kappal_{(20)(00)}\bkl{\partial^2(\sigma')^2 (d_1z)^2}\bkl{\partial(\sigma\sigma')}\\
    &\qquad \qquad + 4\kappal_{(11)(00)}\bkl{(\sigma')^2}\bkl{\partial^2(\sigma\sigma')d_2z}+16\kappal_{(10)(20)}\bkl{\partial(\sigma')^2d_1z}\bkl{\partial(\sigma\sigma')}\\
    &\qquad \qquad +\left.8\kappal_{(11)(20)}\bkl{\partial(\sigma\sigma')}\bkl{(\sigma')^2}\right]+O(n^{-2}).
\end{align*}
Additionally,
\begin{align*}
    \kappa_{(11)(11)}^{(\ell+1)} &=\frac{C_W^2}{n_\ell}\left[\bkl{(\sigma'dz)^4}-\bkl{(\sigma'dz)^2}^2\right]\\
    &+\frac{C_W^2}{8}\left[2\kappal_{(00)(00)}\bkl{\partial^2 (\sigma')^2 (d_1z)^2}^2+16\kappal_{(10)(00)}\bkl{\partial (\sigma')^2 d_1z}\bkl{\partial^2 (\sigma')^2 (d_1z)^2}\right.\\
    &\qquad +8\kappal_{(11)(00)}\bkl{(\sigma')^2}\bkl{\partial^2 (\sigma')^2 (d_1z)^2} + 32\kappal_{(10)(10)}\bkl{\partial (\sigma')^2 d_1z}^2\\
    &\qquad \left.+32\kappal_{(11)(10)} \bkl{(\sigma')^2}\bkl{\partial (\sigma')^2 d_1z}+8\kappal_{(11)(11)} \bk{(\sigma')^2}^2\right].
\end{align*}
Also, 
\begin{align*}
    \kappa_{(11)(22)}^{(\ell+1)}&=\frac{C_W^2}{n_\ell}\left[\bkl{(\sigma')^4(d_1zd_2z)^2}-\bkl{(\sigma'd_1z)^2}\bkl{(\sigma'd_2z)^2}\right]\\
    &+\frac{C_W^2}{8}\left[2\kappal_{(00)(00)}\bkl{\partial^2 (\sigma')^2 (d_1z)^2}\bkl{\partial^2 (\sigma')^2 (d_2z)^2}\right.\\
    &\qquad \qquad + 8\kappal_{(10)(00)}\bkl{\partial^2 (\sigma')^2d_1z}\bkl{\partial^2(\sigma')^2 (d_2z)^2}+8\kappal_{(20)(00)}\bkl{\partial^2 (\sigma')^2d_2z}\bkl{\partial^2(\sigma')^2 (d_1z)^2}\\
    &\qquad \qquad + 4\kappal_{(11)(00)}\bkl{(\sigma')^2}\bkl{\partial^2 (\sigma')^2 (d_2z)^2} + 4\kappal_{(22)(00)} \bkl{ (\sigma')^2}\bkl{\partial^2(\sigma') (d_1z)^2}\\
    &\qquad \qquad + 32\kappal_{(10)(20)}\bkl{\partial(\sigma')^2 d_1z}\bkl{\partial(\sigma')^2 d_2z}\\
    &\qquad \qquad + 16\kappal_{(11)(20)}\bkl{(\sigma')^2}\bkl{\partial(\sigma')^2 d_2z}+16\kappal_{(22)(10)}\bkl{(\sigma')^2}\bkl{\partial(\sigma')^2 d_1z}\\
    &\qquad \qquad \left.+8\kappal_{(11)(22)}\bkl{(\sigma')^2}^2\right]+O(n^{-2}).
\end{align*}
Finally, 
\begin{align*}
    \kappa_{(12)(12)}^{(\ell+1)}&=\frac{C_W^2}{n_\ell}\left[\bkl{(\sigma')^4(d_1zd_2z)^2}-\bkl{(\sigma')^2d_1zd_2z}^2\right]\\
    &+\frac{C_W^2}{8}\left[2\kappal_{(00)(00)}\bkl{\partial^2 (\sigma')^2 d_1z d_2z}^2\right.\\
    &\qquad \qquad +8\kappal_{(10)(00)}\bkl{\partial(\sigma')^2 d_2z}\bkl{\partial^2(\sigma')^2 d_1zd_2z}+8\kappal_{(20)(00)}\bkl{\partial(\sigma')^2 d_1z}\bkl{\partial^2(\sigma')^2 d_1zd_2z}\\ 
    &\qquad \qquad +8\kappal_{(12)(00)}\bkl{(\sigma')^2}\bkl{\partial^2(\sigma')^2d_1zd_2z}+26\kappal_{(10)(20)}\bkl{\partial(\sigma')^2d_2z}\bkl{\partial (\sigma')^2 d_1z}\\
    &\qquad \qquad +8\kappal_{(10)(10)}\bkl{\partial (\sigma')^2 d_2z}^2+8\kappal_{(20)(20)}\bkl{\partial (\sigma')^2 d_1z}^2\\
    &\qquad \qquad +16\kappal_{(12)(10)}\bkl{(\sigma')^2}\bkl{\partial (\sigma')^2 d_2z} + 16\kappal_{(12)(20)}\bkl{(\sigma')^2}\bkl{\partial (\sigma')^2 d_1z}\\
    &\qquad \qquad \left. +8\kappal_{(12)(12)}\bkl{(\sigma')^2}^2\right] +O(n^{-2}).
\end{align*}
\end{proposition}
\begin{proof}
The derivations of the recursions in Proposition \ref{P:k4-deriv-recs} are all similar. The idea is to apply Theorem \ref{T:pert-exp} in the special case of $q_*=1$, $r=1$, a single network input $x_\alpha$, and directional derivatives 
\[
d_j = \partial_{x_{j;\alpha}},\qquad j=1,2.
\]
To keep the notation as uniform as possible, we will continue to write
\[
d_0 = \text{ identity}
\]
for the identity operator. We will have occasion to apply the result to compute expressions of the form
\begin{equation}\label{E:f-exp}
\E{f\lr{d_j z_{i;\alpha}^{(\ell)},\,\,  i=1,\ldots,m,\, j=0,1,2}},    
\end{equation}
where the function $f$ depends on the pre-activations $z_{i;\alpha}^{_{(\ell)}},\, i=1,\ldots,m$ at finitely many neurons and is a polynomial of total degree between $0$ and $4$ in the derivatives $d_jz_{i;\alpha}^{_{(\ell)}}$. Theorem \ref{T:pert-exp} allows us to expand \eqref{E:f-exp} into a perturbative series in powers of $1/n$ with coefficients that are Gaussian integrals of the form 
\[
\bk{ \mathcal P f(d_jz_{i}, i=1,\ldots, m,j=1,2)}_{\kappa^{(\ell)}},
\]
where $\mathcal P$ is a differential operator consisting of finite sums of partial derivatives with respect to the variables $d_jz_{i}$, and $\bk{\cdot}_{\kappa^{(\ell)}}$ denotes an expectation in which $d_jz_{i;\alpha}$ are centered Gaussian with covariance
\begin{align*}
\delta_{i_1i_2}\kappa_{(j_1j_2)}^{(\ell)}:=\Cov\lr{d_{j_1}z_{i_1},d_{j_2}z_{i_2}^{(\ell)}}:=\delta_{i_1i_2}\Cov\lr{d_{j_1}z_{i_1;\alpha}^{(\ell)},d_{j_2}z_{i_2;\alpha}^{(\ell)}},\qquad j_1,j_2=0,1,2.
\end{align*}
Specifically, taking $g\equiv 1$ in Theorem \ref{T:pert-exp} we find that \eqref{E:f-exp} equals
\begin{align*}
    \bk{f}_{\kappa^{(\ell)}} + \frac{1}{8}\sum_{i,i'=1}^{m} \sum_{j_1,j_2,j_1',j_2'=0}^2 \kappa_{(j_1j_2)(j_1'j_2')}^{(\ell)} \bkkl{\frac{\partial^2}{\partial d_{j_1}z_i\partial d_{j_2}z_i}\frac{\partial^2}{\partial d_{j_1'}z_{i'}\partial d_{j_2;}z_{i'}}f} + O(n^{-2})
\end{align*}
Note that in the expectations inside the sum the integrals are with respect to a centered Gaussian with the infinite width covariance $K^{(\ell)}$ defined in \S \ref{S:rec-notation} instead of the finite width covariance $\kappa^{(\ell)}$. The reason is that, as we'll show in \S \ref{S:S2-recs}, we have $\kappa^{(\ell)}=K^{(\ell)}+O(n^{-1})$ and since $\kappa_{(00)(00)}^{(\ell)}=O(n^{-1})$ by Theorem \ref{T:cumulants}, the replacement of $\kappa^{(\ell)}$ by $K^{(\ell)}$ incurs an error of size $O(n^{-2})$, which we are neglecting anyway.

With that said, note that if $f$ is a polynomial of degree $0$ in the derivatives $d_jz_{i;\alpha}^{_{(\ell)}}$ for $j=1,2$ (i.e. doesn't depends on them), then we find that \eqref{E:f-exp} equals
\begin{equation}\label{E:f-exp-0}
   \bk{f}_{\kappa^{(\ell)}} + \frac{1}{8}\sum_{i,i'=1}^{m}\kappa_{(00)(00)}^{(\ell)} \bkkl{\Disq{ z_{i}} \Disq{z_{i'}}f} + O(n^{-2}).
\end{equation}
If instead $f$ is a polynomial of degree $1$ in the derivatives $d_jz_{i;\alpha}^{_{(\ell)}}$ then \eqref{E:f-exp} equals \eqref{E:f-exp-0} plus the expression
\begin{align}
 \label{E:f-exp-1} \frac{1}{8}\sum_{i,i'=1}^{m} &2\kappa_{(10)(00)}^{(\ell)}\bkkl{\left\{\DDi{z_{i}}{d_1z_{i}}\Disq{z_{i'}}+\DDi{z_{i'}}{d_1z_{i'}}\Disq{z_{i}}\right\}f}+\lr{1\leftrightarrow 2},
\end{align}
where $\lr{1\leftrightarrow 2}$ means the same expression with $d_1$ exchanged with $d_2$, $\kappa_{(10)(00)}^{_{(\ell)}}$ replaced by $\kappa_{(20)(00)}^{_{(\ell)}}$, and so on. If $f$ is a polynomial of degree $2$ in the derivatives $d_jz_{i;\alpha}^{_{(\ell)}}$ then we must add to \eqref{E:f-exp-0} and \eqref{E:f-exp-1} in addition 
\begin{align}
\label{E:f-exp-2} \frac{1}{8}\sum_{i,i'=1}^{m} &\left[\kappa_{(11)(00)}^{(\ell)}\bkkl{\left\{\Disq{d_1z_{i}}\Disq{z_{i'}} + \Disq{d_1z_{i'}}\Disq{z_{i}} \right\}f} +\lr{1\leftrightarrow 2}\right.\\
\notag &+\\
\notag &+4\kappa_{(10)(10)}^{(\ell)}\bkkl{\DDi{z_{i}}{d_1z_{i}}\DDi{z_{i'}}{d_1z_{i'}}f} +\lr{1\leftrightarrow 2}\\
\notag &\left.+4\kappa_{(10)(20)}^{(\ell)}\bkkl{\left\{\DDi{z_{i}}{d_1z_{i}}\DDi{z_{i'}}{d_2z_{i'}}+\DDi{z_{i}}{d_2z_{i}}\DDi{z_{i'}}{d_1z_{i'}}\right\}f}\right]
\end{align}
Similarly, if $f$ is a polynomial of degree $3$ in the derivatives $d_jz_{i;\alpha}^{_{(\ell)}}$ then we must add to \eqref{E:f-exp-0}-\eqref{E:f-exp-2} in also 
\begin{align}
\label{E:f-exp-3} \frac{1}{8}\sum_{i,i'=1}^{m} &\left[2\kappa_{(11)(10)}^{(\ell)}\bkkl{\left\{\Disq{d_1z_{i}}\DDi{z_{i'}}{d_1z_{i'}} +\Disq{d_1z_{i'}}\DDi{z_{i}}{d_1z_{i}} \right\}f} +\lr{1\leftrightarrow 2}\right.\\
\notag &+2\kappa_{(11)(20)}^{(\ell)}\bkkl{\left\{\Disq{d_1z_{i}}\DDi{z_{i'}}{d_2z_{i'}}+\Disq{d_1z_{i'}}\DDi{z_{i}}{d_2z_{i}}\right\}f} +\lr{1\leftrightarrow 2}\\
\notag &\left.+4\kappa_{(12)(10)}^{(\ell)}\bkkl{\left\{\DDi{d_1z_{i}}{d_2z_{i}}\DDi{z_{i'}}{d_1z_{i'}}+\DDi{d_1z_{i'}}{d_2z_{i'}}\DDi{z_{i}}{d_1z_{i}}\right\}f}+\lr{1\leftrightarrow 2}\right].
\end{align}
Finally, if $f$ is a polynomial of degree $4$ in the derivatives $d_jz_{i;\alpha}^{_{(\ell)}}$ then we must add to \eqref{E:f-exp-0}-\eqref{E:f-exp-3} also 
\begin{align}
\label{E:f-exp-4} \frac{1}{8}\sum_{i,i'=1}^{m} &\left[\kappa_{(11)(11)}^{(\ell)}\bkkl{\left\{\Disq{d_1z_{i}}\Disq{d_1z_{i'}} \right\}f} +\lr{1\leftrightarrow 2}\right.\\
\notag &+2\kappa_{(11)(12)}^{(\ell)}\bkkl{\left\{\Disq{d_1z_{i}}\DDi{d_1z_{i'}}{d_2z_{i'}}+\Disq{d_1z_{i'}}\DDi{d_1z_{i}}{d_2z_{i}}\right\}f} +\lr{1\leftrightarrow 2}\\
\notag &+\kappa_{(11)(22)}^{(\ell)}\bkkl{\left\{\Disq{d_1z_{i}}\Disq{d_2z_{i'}}+\Disq{d_1z_{i'}}\Disq{d_2z_{i}}\right\}f}\\
\notag &\left.+4\kappa_{(12)(12)}^{(\ell)}\bkkl{\DDi{d_1z_{i}}{d_2z_{i}}\DDi{d_1z_{i'}}{d_2z_{i'}}f}\right].
\end{align}
With these formulas in hand, the derivations of the recursions in Proposition \ref{P:k4-deriv-recs} are straight forward and we provide the details for the case of $\kappa_{(10)(00)}^{_{(\ell+1)}}$, which already illustrates all the main ideas of the derivation. We have
\[
\kappa_{(10)(00)}^{(\ell+1)}=  \E{\lr{\Sigma_{\alpha\alpha}^{(\ell)}-\E{\Sigma_{\alpha\alpha}^{(\ell)}}}\cdot \lr{\partial_{x_{1;\alpha_1}}\Sigma_{\alpha_1\alpha_2}^{(\ell)}\bigg|_{\alpha_1=\alpha_2=\alpha}- \E{\partial_{x_{1;\alpha_1}}\Sigma_{\alpha_1\alpha_2}^{(\ell)}}\bigg|_{\alpha_1=\alpha_2=\alpha}}}. 
\]
Recall that 
\[
\Sigma_{\alpha_1\alpha_2}^{(\ell)} = \frac{C_W}{n_\ell}\sum_{j=1}^{n_\ell}\sigma\lr{z_{j;\alpha_1}^{(\ell)}}\sigma\lr{z_{j;\alpha_2}^{(\ell)}} 
\]
Hence, 
\[
\Sigma_{\alpha\alpha}^{(\ell)}-\E{\Sigma_{\alpha\alpha}^{(\ell)}} = \frac{C_W}{n_\ell}\sum_{j=1}^{n_\ell} X_{j;\alpha},\]
where
\[
X_{j;\alpha}:=\sigma\lr{z_{j;\alpha}^{(\ell)}}^2- \E{\sigma\lr{z_{j;\alpha}^{(\ell)}}^2}.
\]
Similarly  
\[
\partial_{x_{\alpha_1}}\Sigma_{\alpha_1\alpha_2}^{(\ell)}\bigg|_{\alpha_1=\alpha_2=\alpha}- \E{\partial_{x_{\alpha_1}}\Sigma_{\alpha_1\alpha_2}^{(\ell)}}\bigg|_{\alpha_1=\alpha_2=\alpha} = \frac{C_W}{n_\ell}\sum_{j=1}^{n_\ell} Y_{j;\alpha},\]
where
\[
Y_{j;\alpha} = \sigma\lr{z_{j;\alpha}^{(\ell)}}\sigma'\lr{z_{j;\alpha}^{(\ell)}} \partial_{x_{1;\alpha}} z_{j;\alpha}^{(\ell)} - \E{ \sigma\lr{z_{j;\alpha}^{(\ell)}}\sigma'\lr{z_{j;\alpha}^{(\ell)}} \partial_{x_{1;\alpha}} z_{j;\alpha}^{(\ell)}}.
\]
By symmetry, we obtain
\begin{align*}
\kappa_{(10)(00)}^{(\ell+1)} &=  \frac{C_W^2}{n_\ell}\E{X_{1;\alpha}Y_{1;\alpha}}+ C_W^2\lr{1-\frac{1}{n_\ell}} \E{X_{1;\alpha}Y_{2;\alpha}}.
\end{align*}
We have by Theorem \ref{T:pert-exp} that
\begin{align*}
   \frac{C_W^2}{n_\ell}\E{X_{1;\alpha}Y_{1;\alpha}} = \frac{C_W^2}{n_\ell}\left[\bkkl{\sigma^3\sigma' d_1z}-\bkkl{\sigma^2}\bkkl{\sigma\sigma'd_1z}\right] + O(n^{-2}),
\end{align*}
giving the first term in \eqref{E:k1000-rec}. Further, note that $X_{1;\alpha}Y_{2;\alpha}$ is a polynomial of total degree one in the derivative $d_1z_{i;\alpha}^{(\ell)}$. Hence, applying \eqref{E:f-exp-0} and \eqref{E:f-exp-1}, we find that $C_W^2 \E{X_{1;\alpha}Y_{2;\alpha}}$ equals
\begin{align*}
    C_W^2\bkal{X_{1;\alpha}}\bkal{Y_{2;\alpha\alpha}}
\end{align*}
plus
\begin{align*}
\frac{C_W^2}{4}\left[2\kappa_{(00)(00)}^{(\ell)}\bkkl{\Disq{ z_{1}} X_{1;\alpha}\Disq{z_{2}}Y_{2;\alpha}}+ 2\kappa_{(10)(00)}^{(\ell)}\bkkl{\Disq{z_1}X_{1;\alpha}\DDi{z_2}{d_1z_2} Y_{2;\alpha}}\right]
\end{align*} 
plus an error of size $O(n^{-2})$. Note that by Theorem \ref{T:pert-exp} we have
\begin{align*}
    \bkal{X_{1;\alpha}},\, \E{Y_{1;\alpha}} &= O(n^{-1}).
\end{align*}
Hence, 
\begin{align*}
     C_W^2\bkal{X_{1;\alpha}}\bkal{Y_{2;\alpha}} &= O(n^{-2}).
\end{align*}
Further, 
\[
\bkkl{\Disq{ z_{1}} X_{1;\alpha}\Disq{z_{2}}Y_{2;\alpha}} = \bkkl{\partial^2(\sigma^2)}\bkkl{\partial^2(\sigma\sigma')d_1z}
\]
and 
\[
\bkkl{\Disq{z_1}X_{1;\alpha}\DDi{z_2}{d_1z_2} Y_{2;\alpha}}= \bkkl{\partial^2 \sigma^2}\bkkl{\partial(\sigma\sigma')}.
\]
Thus, all together, we find
\begin{align*}
\kappa_{(10)(00)}^{(\ell+1)} &=  \frac{C_W^2}{n_\ell}\left[\bkkl{\sigma^3\sigma' d_1z}-\bkkl{\sigma^2}\bkkl{\sigma\sigma'd_1z}\right]\\
&+\frac{C_W^2}{4}\kappa_{(00)(00)}^{(\ell)}\bkkl{\partial^2(\sigma^2)}\bkkl{\partial^2(\sigma\sigma')d_1z}\\
&+\frac{C_W^2}{2} \kappa_{(10)(00)}^{(\ell)}\bkkl{\partial^2 (\sigma^2)}\bkkl{\partial(\sigma\sigma')}\\
&+O(n^{-2}),
\end{align*}
completing our derivation of \eqref{E:k1000-rec}.

\end{proof}

\subsection{Recursions for Finite Width Corrections to 2nd Cumulant}\label{S:S2-recs}
We remind the reader of our standing notation. Namely, we fix a network input $x_\alpha\neq 0\in \R^{n_0}$ and define 
\[
\kappa_{(ij)}^{(\ell)}:=\Cov\lr{\partial_{x_{i;\alpha}}z_{i;\alpha}^{(\ell)},\, \partial_{x_{j;\alpha}}z_{i;\alpha}^{(\ell)}} = \partial_{x_{i;\alpha_1}}\partial_{x_{j;\alpha_2}}\kappa_{\alpha_1\alpha_2}^{(\ell)}\bigg|_{x_{\alpha_1}=x_{\alpha_2}=x_\alpha},\qquad i,j=1,\ldots, n_0.
\]
We also recall that 
\[
K_{(ij)}^{(\ell)}:=\lim_{n\gives \infty}\kappa_{(ij)}^{(\ell)},\qquad 
S_{(ij)}^{(\ell)}:=\kappa_{(ij)}^{(\ell)}- K_{(ij)}^{(\ell)}. 
\]
The purpose of this section is to derive recursions for $S_{(ij)}^{(\ell)}$ to leading order in $1/n$. We have
\begin{proposition}\label{P:S-recs}
We have
\begin{align}
   \label{E:S00-rec}S_{(00)}^{(\ell+1)}&=\frac{C_W}{8}\kappa_{4;\alpha}^{(\ell)}\bk{\partial^4 \sigma^2}_{K_{(00)}^{(\ell)}} +\chi_{||;\alpha}^{(\ell)}S_{(00)}^{(\ell)} + O(n^{-2}).
\end{align}
Further, 
\begin{align}
 \label{E:S10-rec}     S_{(10)}^{(\ell+1)}&= \chi_{||}^{(\ell)}S_{(10)}^{(\ell)} +K_{(10)}^{(\ell)}\frac{C_W}{2}S_{(00)}^{(\ell)}\bk{\partial^3(\sigma\sigma')}_{K^{(\ell)}}\\
\notag &+\frac{C_W}{8}\left[\kappa_{(00)(00)}^{(\ell)}\bk{\partial^4(\sigma\sigma')d_1z}_{\kappa^{(\ell)}}+4\kappa_{(10)(00)}^{(\ell)} \bk{\partial^3(\sigma\sigma')}_{\kappa^{(\ell)}}\right] + O(n^{-2}).
\end{align}
Moreover, 
\begin{align}
 \label{E:S11-rec}      S_{(11)}^{(\ell+1)}&=\chi_{\perp}^{(\ell)} S_{(11)}^{(\ell)} +\frac{C_W}{2}K_{(11)}^{(\ell)} S_{(00)}^{(\ell)}\bkkl{\partial^2 (\sigma')^2}+O(n^{-2})\\
 \notag  &+\frac{C_W}{2}\lr{K_{(10)}^{(\ell)}}^2  S_{(00)}^{(\ell)}\bkkl{\partial^4(\sigma')^2}+2C_WK_{(10)}^{(\ell)}S_{(10)}^{(\ell)}\bkkl{\partial^2 (\sigma')^2}\\
    \notag&+\frac{C_W}{8}\bigg\{\kappa_{(00)(00)}^{(\ell)} \lr{K_{(11)}^{(\ell)}\bkkl{\partial^4(\sigma')^2} + \lr{K_{(10)}^{(\ell)}}^2\bkkl{\partial^6 (\sigma')^2}}\\
    \notag &\qquad \qquad +8\kappa_{(10)(00)}^{(\ell)}K_{(10)}^{(\ell)}\bkkl{\partial^4(\sigma')^2} + 4\lr{2\kappa_{(10)(10)}^{(\ell)}+\kappa_{(11)(00)}^{(\ell)}}\bkkl{\partial^2(\sigma')^2}\bigg\}.
\end{align}
Finally, 
\begin{align}
  \label{E:S12-rec}     S_{(12)}^{(\ell+1)} &= \chi_{\perp}^{(\ell)}S_{(12)}^{(\ell)} +\frac{C_W}{2}S_{(00)}^{(\ell)}\lr{K_{(12)}^{(\ell)}\bkl{\partial^2(\sigma')^2}+K_{(10)}^{(\ell)}K_{(20)}^{(\ell)}\bkl{\partial^4(\sigma')^2}}\\
 \notag   &+C_W\lr{K_{(10)}^{(\ell)}S_{(20)}^{(\ell)}+K_{(20)}^{(\ell)}S_{(10)}^{(\ell)}}\bkl{\partial^2(\sigma')^2}+O(n^{-2})\\
 \notag     &+\frac{1}{8}\bigg\{ \kappa_{(00)(00)}^{(\ell)}\lr{K_{(12)}^{(\ell)}\bkl{\partial^4(\sigma')^2}+K_{(10)}^{(\ell)}K_{(20)}^{(\ell)}\bkl{\partial^6 (\sigma')^2}}\\
  \notag    &\qquad +4\lr{\kappa_{(10)(00)}^{(\ell)}K_{(20)}^{(\ell)}+\kappa_{(20)(00)}^{(\ell)}K_{(10)}^{(\ell)}}\bkl{\partial^4(\sigma')^2}\\
  \notag    &\qquad + 4\lr{\kappa_{(12)(00)}^{(\ell)} + 2\kappa_{(10)(20)}^{(\ell)}}\bkkl{\partial^2(\sigma')^2}\bigg\}.
\end{align}
\end{proposition}
\begin{proof}
The derivation of these recursion is very similar, so we provide the details for the most involved case of $S_{(12)}^{_{(\ell+1)}}.$ By definition, we have
\begin{align*}
    \kappa_{(12)}^{(\ell+1)}&=\E{d_1z_{i;\alpha}^{(\ell+1)}d_2z_{i;\alpha}^{(\ell+1)}}=C_W\E{\sigma'(z_{i;\alpha}^{(\ell)})^2 d_1 z_{i;\alpha}^{(\ell)}d_2z_{i;\alpha}^{(\ell)}}.
\end{align*}
Applying Theorem \ref{T:pert-exp} (or more precisely equations \eqref{E:f-exp-0} -\eqref{E:f-exp-2}) we therefore find
\begin{align}
 \notag   \kappa_{(12)}^{(\ell+1)}&=C_W\bk{(\sigma')^2 d_1z d_2z}_{\kappa^{(\ell)}}+O(n^{-2})\\
 \notag   &+\frac{C_W}{8}\bigg\{\kappa_{(00)(00)}^{(\ell)}\bkkl{\partial^4(\sigma')^2 d_1z d_2z}\\
    \notag &\qquad \qquad+4\kappa_{(10)(00)}^{(\ell)}\bkkl{\partial^3 (\sigma')^2d_2 z}+4\kappa_{(20)(00)}^{(\ell)}\bkkl{\partial^3 (\sigma')^2d_1 z}\\
 \label{E:S12-deriv}   &\qquad \qquad +4\lr{\kappa_{(12)(00)}^{(\ell)} + 2\kappa_{(10)(20)}^{(\ell)}}\bkkl{\partial^2 (\sigma')^2}\bigg\}.
\end{align}
Next, integrating by parts yields
\begin{align*}
   \bk{(\sigma')^2 d_1z d_2z}_{\kappa^{(\ell)}}&=\kappa_{(12)}^{(\ell)}\bkal{(\sigma')^2} + \kappa_{(10)}^{(\ell)}\kappa_{(20)}^{(\ell)}\bkal{\partial^2(\sigma')^2}\\
   &=\kappa_{(12)}^{(\ell)}\lr{\bkkl{(\sigma')^2} +\frac{1}{2}S_{(00)}^{(\ell)}\bkkl{\partial^2 (\sigma')^2}}\\
   &+ \lr{K_{(10)}^{(\ell)}+S_{(10)}^{(\ell)}}\lr{K_{(20)}^{(\ell)}+S_{(20)}^{(\ell)}}\lr{\bkkl{\partial^2(\sigma')^2}+\frac{1}{2}S_{(00)}^{(\ell)}\bkkl{\partial^2(\sigma')^2}}.
\end{align*}
Similarly, for $j=1,2$ we find
\begin{align*}
   \bk{\partial^4 (\sigma')^2 d_1z d_2z}_{K^{(\ell)}}
   &=K_{(12)}^{(\ell)}\bkkl{\partial^4(\sigma')^2}+K_{(10)}^{(\ell)}K_{(20)}^{(\ell)}\bkkl{\partial^6(\sigma')^2} + O(n^{-1})\\
   \bkkl{\partial^3 (\sigma')^2d_j z}&=K_{(j0)}^{(\ell)}  \bkkl{\partial^4 (\sigma')^2}.
\end{align*}
Substituting this into \eqref{E:S12-deriv}, comparing with the recursion \eqref{E:K12-rec}, and recalling that $\kappa_{(j_1j_2)(j_1'j_2')}^{_{(\ell)}}=O(n^{-1})$ yields the recursion \eqref{E:S12-rec}.
\end{proof}

\section{Proof of Theorem \ref{T:evgp}}\label{S:evgp-pf}
In this section we prove Theorem \ref{T:evgp}. The key technical step will be to solve the recursions for all possible second and fourth cumulants of the value and partial derivatives of some component of the output of a random fully connected network with respect to two different components of the network input that were obtained in \S\ref{S:2-4-deriv-recs}. Before diving into this part of the proof in \S \ref{S:tanh-rec-solutions}, we give the proof of Theorem \ref{T:evgp} modulo solving the recursions in \S \ref{S:2-4-deriv-recs}. 

\subsection{Reduction of Theorem \ref{T:evgp} to Cumulant Recursions from \S\ref{S:2-4-deriv-recs}}\label{S:evgp-pf-start}
We first recall the setup of Theorem \ref{T:evgp}. Namely, we fix a non-linearity $\sigma$ satisfying \eqref{A:sigma-prop} that belongs to the $K_*=0$ universality class, as defined in \S \ref{S:tanh-univ}, and consider a depth $L$ random fully connected network with input dimension $n_0$, hidden layer widths $n_1=\cdots=n_L=n$, output dimension $n_{L+1}$ and non-linearity $\sigma$. We assume this network is tuned to criticality (see \S \ref{S:crit}). We then fix a network input $x_\alpha\neq 0$ and seek to compute the averages 
\[
\E{\mathrm{Grad~Mean}^{(1)}}\qquad \text{and}\qquad \E{\mathrm{Grad~Var}^{(1)}}
\]
of the empirical mean 
\begin{align}
\label{E:emp-mean-def-2}    \mathrm{Grad~Mean}^{(1)} = \frac{1}{n_0n_1}\sum_{j=1}^{n_0}\sum_{i=1}^{n_1} \lr{\frac{\partial z_{q;\alpha}^{(L+1)}}{\partial W_{ij}^{(1)}}}^2
\end{align}
and the empirical variance 
\begin{align}\label{E:emp-var-def-2}
    \mathrm{Grad~Var}^{(1)} = \frac{1}{n_{0}n_1}\sum_{j=1}^{n_{0}}\sum_{i=1}^{n_1} \lr{\frac{\partial z_{q;\alpha}^{(L+1)}}{\partial W_{ij}^{(1)}}}^4 - \lr{  \mathrm{Grad~Mean}^{(1)}}^2
\end{align}
of the squared gradients $
(\partial z_{q;\alpha}^{_{(L+1)}}/\partial W_{ij}^{_{(1)}})^2$ over all weights in layer $1$. To evaluate these quantities to leading order in $n,\ell$, note that by symmetry the mean of the expressions $  \mathrm{Grad~Mean}^{(1)}$ and $  \mathrm{Grad~Var}^{(1)}$ \eqref{E:emp-mean-def-2} and \eqref{E:emp-var-def-2} are independent of the choice of $q$. Next, note that by the chain rule we have
\[
\frac{\partial z_{q;\alpha}^{(L+1)}}{\partial W_{ij}^{(1)}} = x_{j;\alpha}\sigma'(z_{i;\alpha}^{(1)})\frac{\partial z_{q;\alpha}^{(L+1)}}{\partial \sigma(z_{i;\alpha}^{(1)})}.
\]
Hence, by symmetry, 
\[
\E{\mathrm{Grad~Mean}^{{(1)}}} =  \frac{1}{n_0}\norm{x_\alpha}^2 \E{\lr{ \sigma'\lr{z_{1;\alpha}^{(1)}}\frac{\partial z_{q;\alpha}^{(L+1)}}{\partial \sigma(z_{i;\alpha}^{(1)})}}^2}.
\]
Recalling that $\mF^{(\ell)}$ is the sigma algebra generated by all weights and biases in layers up to and including $\ell$, we may thus write
\begin{align}
   \label{E:grad-mean-form}\E{\mathrm{Grad~Mean}^{{(1)}}} &=\frac{1}{n_0}\norm{x_\alpha}^2 \E{ \lr{\sigma'\lr{z_{1;\alpha}^{(1)}}}^2\E{\lr{\frac{\partial z_{q;\alpha}^{(L+1)}}{\partial \sigma(z_{i;\alpha}^{(1)})}}^2~\bigg|~\mF^{(1)}}}.
\end{align}
Observe that the conditional expectation in the previous line is precisely the square of the partial derivative of one component of the output in a random depth $L-1$ network tuned to criticality with respect to one component of the input $\sigma_{\alpha}^{_{(\ell)}}$. In the notation \eqref{E:kappa-deriv-def} above, we therefore have
\[
\E{\lr{\frac{\partial z_{q;\alpha}^{(L+1)}}{\partial \sigma(z_{1;\alpha}^{(1)})}}^2~\bigg|~\mF^{(\ell)}} = \kappa_{(11)}^{(L-1)}.
\]
We will see by combining Lemma \ref{L:2pt-deriv-recs-leading} and Proposition \ref{P:2pt-deriv-recs-subleading} below that 
\begin{align*}
    \kappa_{(11)}^{(L-1)}&=\frac{C_We^{-\gamma}}{n_1 L}\lr{1 + O(L^{-1})} + O(n^{-2}),
\end{align*}
where the implicit constant in $O(L^{-1})$ is independent of $L,n$. Since at leading order in $n,L$ this expression is a (non-random) constant, we write by a slight abuse of notation
\begin{align}
   \notag \E{\mathrm{Grad~Mean}^{{(1)}}} &=\frac{1}{n_0}\norm{x_\alpha}^2\E{ \lr{\sigma'\lr{z_{1;\alpha}^{(\ell)}}}^2}\kappa_{(11)}^{(L-1)}\\
   \label{E:grad-mean-kappa} &=\frac{1}{n_0}\norm{x_\alpha}^2\bk{(\sigma')^2}_{K^{(1)}}\kappa_{(11)}^{(L-1)}
\end{align}
where in the last step we use that pre-activations in layer $1$ are iid Gaussian even at finite width. We now similarly simplify the mean of the empirical gradient variance. Namely, we first note that by symmetry we have the decomposition
\[
\E{\mathrm{Grad~Var}^{(1)}} = I-II
\]
where
\begin{align*}
 I&:= \lr{\frac{1}{n_0}\norm{x_\alpha}_4^4-\frac{1}{n_1}\lr{\frac{1}{n_0}\norm{x_\alpha}^2}^2 }\E{\lr{\sigma'(z_{1;\alpha}^{(1)})\frac{\partial z_{q;\alpha}^{(L+1)}}{\partial \sigma(z_{1;\alpha}^{(1)})}}^4}\\
 II&:=\lr{\frac{1}{n_0}\norm{x_\alpha}^2}^2\lr{1-\frac{1}{n_1}}\E{\lr{\sigma'(z_{1;\alpha}^{(1)})\sigma'(z_{2;\alpha}^{(1)})\lr{\frac{\partial z_{q;\alpha}^{(L+1)}}{\partial \sigma(z_1^{(1)})}\frac{\partial z_{q;\alpha}^{(L+1)}}{\partial \sigma(z_2^{(1)})}}^2}^2}.
\end{align*}
By conditioning on $\mF^{(1)}$ as before (and again using that at leading order in $n,L$ the cumulants $\kappa_{(ij)(i'j')}^{(\ell)}$ are independent of the initial condition at large $\ell$) we have
\begin{align}
\label{E:mean-var-kappa-1} I&:= 3\lr{\kappa_{(11)(11)}^{(L-1)}+\lr{\kappa_{(11)}^{(L-1)}}^2}\bk{(\sigma')^4}_{K^{(\ell)}}\lr{\frac{1}{n_0}\norm{x_\alpha}_4^4-\frac{1}{n_1}\lr{\frac{1}{n_0}\norm{x_\alpha}^2}^2 }\\
\notag II&:=\lr{2\kappa_{(12)(12)}^{(L-1)}+2\lr{\kappa_{(12)}^{(L-1)}}^2 + \kappa_{(11)(22)}^{(L-1)} + \kappa_{(11)}^{(L-1)}\kappa_{(22)}^{(L-1)}}\\
 \label{E:mean-var-kappa-2} &\qquad \times \bk{(\sigma')^2}_{K^{(1)}}^2 \lr{\frac{1}{n_0}\norm{x_\alpha}^2}^2\lr{1-\frac{1}{n_1}}.
\end{align}
Combining the expressions \eqref{E:grad-mean-kappa}, \eqref{E:mean-var-kappa-1}, and \eqref{E:mean-var-kappa-2}, we find that
\begin{align*}
 \frac{\E{\mathrm{Grad~Var}^{{(1)}}}}{\E{\mathrm{Grad~Mean}^{{(1)}}}^2}   &= 3\frac{\bk{(\sigma')^4}_{K^{(1)}}}{\bk{(\sigma')^2}_{K^{(1)}}^2}\lr{\frac{n_0^{-1}\norm{x_\alpha}_4^4}{(n_{0}^{-1} \norm{x_\alpha}_2^2)^2}- \frac{1}{n_1}}\lr{\wkappa_{(11)(11)}^{(L-1)}+1}\\ &-\lr{1-\frac{1}{n_1}}\lr{2\wkappa_{(12)(12)}^{(L-1)}+2\lr{\wkappa_{(12)}^{(L-1)}}^2 + \wkappa_{(11)(22)}^{(L-1)} +1},
\end{align*}
plus errors of size $O(n^{-2})$, where we've abbreviated for $i,j\in \set{1,2}$
\begin{align*}
    \wkappa_{(ii)(jj)}^{(\ell)}:=\frac{\kappa_{(ii)(jj)}^{(\ell)}}{\kappa_{(ii)}^{(\ell)}\kappa_{(jj)}^{(\ell)}},\qquad \wkappa_{(12)(12)}^{(\ell)}:=\frac{\kappa_{(12)(12)}^{(\ell)}}{\kappa_{(11)}^{(\ell)}\kappa_{(22)}^{(\ell)}},\qquad \wkappa_{(12)}^{(\ell)}:=\frac{\kappa_{(12)}^{(\ell)}}{\lr{\kappa_{(11)}^{(\ell)}\kappa_{(22)}^{(\ell)}}^{1/2}}.
\end{align*}
Appealing to Propositions \ref{P:4pt-deriv-recs-tanh} and \ref{P:2pt-deriv-recs-subleading} we conclude 
\begin{align*}
    \wkappa_{(11)(11)}^{(L-1)}+1& =1+\frac{8L}{3n}+O(n^{-1})+O(n^{-2}) \\
    2\wkappa_{(12)(12)}^{(L-1)}+2\lr{\wkappa_{(12)}^{(L-1)}}^2 + \wkappa_{(11)(22)}^{(L-1)} +1& = 1+\frac{8L}{3n}+O(n^{-1})+O(n^{-2}),
\end{align*}
where the implicit constants in $O(n^{-1})$ depend only on $\sigma$ whereas the constant in $O(n^{-2})$ depends on both $\sigma$ and $L$.  Combining these expressions with the following estimates completes the derivation of Theorem \ref{T:evgp}, modulo solving the recursions for the $2nd$ cumulants $\kappa_{(ij)}^{_{(\ell)}}$ and the fourth cumulants $\kappa_{(i_1j_1)(i_2j_2)}^{_{(\ell)}}$, which we supply in the next section.  \hfill $\square$

\subsection{Completion of Proof of Theorem \ref{T:derivs-cumulants}: 2nd, 4th Cumulants at Large Depth for the $K_*=0$ Universality Class}\label{S:tanh-rec-solutions}
In this section, we complete the proof of Theorem \ref{T:evgp} by solving the 2nd and 4th cumulant recursion obtained in \S \ref{S:2-4-deriv-recs} when $\sigma$ belongs to the $K_*=0$ universality class. We start in \S \ref{S:k2-sol} by solving the infinite width variance recursions for $K_{(ij)}^{_{(\ell)}}$. We then use this for solving to leading order in $1/n$ the fourth cumulant recursions for $\kappa_{(ij)(i'j')}^{_{(\ell)}}$ in \S \ref{S:4pt-solns-tanh}. Finally, we combine in \S \ref{S:2pt-solns-tanh} the results of the preceding two sections to solve the recursions for $S_{(ij)}^{_{(\ell)}}$ again at leading order in $n, \ell$. 

\subsection{Two Point Recursions}\label{S:k2-sol}
We remind the reader of our standing notation. Namely, we let $\sigma$ be an non-linearity from the $K_*=0$ universality class (see \S \ref{S:tanh-univ}). We also fix a network input $x_\alpha\neq 0\in \R^{n_0}$ and define 
\[
K_{(ij)}^{(\ell)}:=\lim_{n\gives \infty}\Cov\lr{\partial_{x_{i;\alpha}}z_{1;\alpha}^{(\ell)},\, \partial_{x_{j;\alpha}}z_{1;\alpha}^{(\ell)}} = \partial_{x_{i;\alpha_1}}\partial_{x_{j;\alpha_2}}K_{\alpha_1\alpha_2}^{(\ell)}\bigg|_{x_{\alpha_1}=x_{\alpha_2}=x_\alpha},\qquad i,j=1,\ldots, n_0.
\]
The recursive description for $K_{(ij)}^{(\ell)}$ was obtained in Lemma \ref{L:K-deriv-recs}. These recursions held for any non-linearity and any values of $C_b,C_W$. The following Lemma records the large $\ell$ behavior of $K_{(ij)}^{(\ell)}$ when $\sigma$ belongs to the $K_*=0$ universality class and the network is tuned to criticality. 
\begin{lemma}\label{L:2pt-deriv-recs-leading} We have
\begin{align}
    \label{E:K00}K_{(00)}^{(\ell+1)} &= \frac{1}{a\ell}+O_\delta(\ell^{-2+\delta})\\
    \label{E:K10}K_{(10)}^{(\ell+1)} &=\frac{C_We^{-2\gamma}x_{1;\alpha}}{n_0\ell^2}\lr{1+O(\ell^{-1})}\\
    \label{E:K11}K_{(11)}^{(\ell+1)} & = \frac{C_We^{-\gamma}}{n_0\ell}\lr{1+O(\ell^{-1})} + O\lr{\frac{(\sigma_{1}^{(0)})^2}{n_0^2\ell^3}}\\
   \label{E:K12} K_{(12)}^{(\ell+1)} &=O\lr{\frac{x_{1;\alpha}x_{2;\alpha}}{n_0^2\ell^3}}
\end{align}
plus errors of size $O(n^{-1})$. 
\end{lemma}

\begin{proof}
From Lemma \ref{L:K-deriv-recs} we have, up to errors of size $O(n^{-1})$, that
\begin{align}
\label{E:K00-rec2}    K_{(00)}^{(\ell+1)} &= C_W\bkl{\sigma^2}\\
\label{E:K10-rec2}       K_{(10)}^{(\ell+1)} &=\chi_{||}^{(\ell)}K_{(10)}^{(\ell)}\\
\label{E:K11-rec2}       K_{(11)}^{(\ell+1)} &=C_W\bk{\partial^2 (\sigma')^2}_{(\ell)}\lr{K_{(10)}^{(\ell)}}^2+\chi_{\perp}^{(\ell)}K_{(11)}^{(\ell)}\\
\label{E:K12-rec2}       K_{(12)}^{(\ell+1)} &=C_W\bk{\partial^2 (\sigma')^2}_{(\ell)}K_{(10)}^{(\ell)}K_{(20)}^{(\ell)}+\chi_{\perp}^{(\ell)}K_{(12)}^{(\ell)}.
\end{align}
We already proved \eqref{E:K00} in \S \ref{S:iw-tanh}. Next, \eqref{E:K10-rec2} yields 
\[
K_{(10)}^{(\ell+1)} =K_{(10)}^{(1)}\prod_{\ell'=1}^\ell \chi_{||}^{(\ell')}.
\]
By direct computation we have
\[
K_{(10)}^{(1)} =\lim_{n\gives \infty} \E{z_{i;\alpha}^{(1)}d_1z_{i;\alpha}^{(1)}} = \frac{C_W}{n_0} x_{1;\alpha}.
\]
Moreover, 
\begin{align*}
\prod_{\ell'=1}^\ell \chi_{||}^{(\ell')} &= \prod_{\ell'=1}^\ell \lr{1-\frac{2}{\ell'}+O((\ell')^{-2})}\\
& \exp\left[\sum_{\ell'=1}^\ell \log\lr{1-\frac{2}{\ell'}+O((\ell')^{-2})}\right]\\
&=     \exp\left[O(\ell^{-1})+-2\sum_{\ell'=1}^\ell\frac{1}{\ell'}\right]\\
&=e^{-2\gamma}\ell^{-2}\lr{1+O(\ell^{-1})},
\end{align*}
where $\gamma$ is the Euler-Mascheroni constant. This yields \eqref{E:K10}. Further, 
\[
C_W\bk{\partial^2 (\sigma')^2}_{(\ell)}\lr{K_{(10)}^{(\ell)}}^2 = O\lr{\frac{(\sigma_{1}^{(0)})^2}{n_0^2\ell^4}}.
\]
Hence, using Lemma \ref{L:rec-lemma}, we find to leading order in $\ell, n$ that
\[
 K_{(11)}^{(\ell+1)} = K_{11}^{(1)}\prod_{\ell'=1}^\ell \chi_{\perp}^{(\ell)} +O\lr{\frac{(\sigma_{1}^{(0)})^2}{n_0^2\ell^3}}=K_{11}^{(1)}\prod_{\ell'=1}^\ell \lr{1-\ell^{-1}+O(\ell^{-2})} +O\lr{\frac{(\sigma_{1}^{(0)})^2}{n_0^2\ell^3}}
\]
Noting that
\[
K_{11}^{(1)} = \frac{C_W}{n_0}
\]
shows
 \[
K_{(11)}^{(\ell+1)} =K_{11}^{(1)}\prod_{\ell'=1}^\ell \chi_{\perp}^{(\ell)} +O\lr{\frac{(\sigma_{1}^{(0)})^2}{n_0^2\ell^3}} = \frac{C_We^{-\gamma}}{n_0\ell}\lr{1+O(\ell^{-1})} + O\lr{\frac{(\sigma_{1}^{(0)})^2}{n_0^2\ell^3}},
 \]
completing the derivation of \eqref{E:K11}. The derivation of \eqref{E:k12} is similar except that $K_{12}^{(1)} = 0.$
\end{proof}


\subsection{Four Point Recursions}\label{S:4pt-solns-tanh}
We remind the reader of our standing notation. Namely, we let $\sigma$ be an non-linearity from the $K_*=0$ universality class (see \S \ref{S:tanh-univ}). We also fix a network input $x_\alpha\neq 0\in \R^{n_0}$ and recall that 
\begin{align*}
\kappa_{(ij)(i'j')}^{(\ell)}&:=\Cov\lr{\partial_{x_{i;\alpha}}\partial_{x_{j;\alpha}}\Delta_{\alpha\alpha}^{(\ell)},\, \partial_{x_{i';\alpha}}\partial_{x_{j';\alpha}}\Delta_{\alpha\alpha}^{(\ell)}}\\
& = \kappa\lr{\partial_{x_{i;\alpha}}z_{1;\alpha}^{(\ell+1)},\,\partial_{x_{j;\alpha}}z_{1;\alpha}^{(\ell+1)} ,\, \partial_{x_{i';\alpha}}z_{2;\alpha}^{(\ell+1)}, \, \partial_{x_{j';\alpha}}z_{2;\alpha}^{(\ell+1)}},    
\end{align*}
where $\Delta_{\alpha\alpha}^{(\ell)}$ is defined in \eqref{E:Delta-def}.

The recursive description for $\kappa_{(ij)(i'j')}^{(\ell)}$ was obtained in Proposition \ref{P:k4-deriv-recs}. These recursions held for any non-linearity and any values of $C_b,C_W$. The following Lemma records the large $\ell$ behavior of $\kappa_{(ij)(i'j')}^{(\ell)}$ when $\sigma$ belongs to the $K_*=0$ universality class and the network is tuned to criticality. 


\begin{proposition}\label{P:4pt-deriv-recs-tanh}
We have
\begin{align}
 \label{E:k0000-form}   \kappa_{(00)(00)}^{(\ell)} &=\frac{2}{3n\ell a^2}(1+O(\ell^{-1})) = \frac{2\ell}{3n}\lr{K_{(00)}^{(\ell)}}^2(1+O(\ell^{-1})) \\
  \label{E:k1000-form}     \kappa_{(10)(00)}^{(\ell+1)}&=\frac{C_We^{-2\gamma}x_{1;\alpha}}{3ann_0\ell^2}\lr{1+O(\ell^{-1})} = \frac{\ell}{3n}K_{(10)}^{(\ell)}K_{(00)}^{(\ell)}\lr{1+O(\ell^{-1})}\\
 \label{E:k1010-form}      \kappa_{(10)(10)}^{(\ell+1)} &=\frac{C_W e^{-\gamma}}{3ann_0 \ell}(1+O(\ell^{-1}))=\frac{\ell}{3n} K_{(00)}^{(\ell)}K_{(11)}^{(\ell)}(1+O(\ell^{-1}))\\
  \label{E:k1020-form}    \kappa_{(10)(20)}^{(\ell+1)}&=O\lr{\frac{x_{1;\alpha}x_{2;\alpha}}{nn_0^2\ell^3}}\\
   \label{E:k1100-form}     \kappa_{(11)(00)}^{(\ell+1)} &=-\frac{C_W e^{-\gamma}}{3ann_0\ell}(1+O(\ell^{-1}))=-\frac{\ell}{3n}K_{(11)}^{(\ell)}K_{(00)}^{(\ell)}(1+O(\ell^{-1}))\\
  \label{E:k1200-form}    \kappa_{(12)(00)}^{(\ell+1)}&= O\lr{\frac{1}{nn_0^2\ell^3}}\\
    \label{E:k1110-form}    \kappa_{(11)(10)}^{(\ell+1)} &=O\lr{\frac{1}{nn_0^2\ell^{2}}}\\
 \label{E:k1210-form}   \kappa_{(12)(10)}^{(\ell+1)}&=O\lr{\frac{x_{2;\alpha}}{nn_0^2\ell}}\\
 \label{E:k1120-form}   \kappa_{(11)(20)}^{(\ell+1)}&=O\lr{\frac{x_{2;\alpha}}{nn_0^2\ell}}\\
\label{E:k1111-form}    \kappa_{(11)(11)}^{(\ell+1)} &= \frac{8C_W^2e^{-2\gamma}}{3nn_0^2\ell}(1+O(\ell^{-1})) = \frac{8\ell}{3n}\lr{K_{(11)}^{(\ell)}}^2(1+O(\ell^{-1})) \\
 \label{E:k1122-form}       \kappa_{(11)(22)}^{(\ell+1)}&=\frac{2C_W^2 e^{-2\gamma}}{3nn_0^2\ell}(1+O(\ell^{-1}))=\frac{2\ell}{3n}K_{(11)}^{(\ell)}K_{(22)}^{(\ell)}(1+O(\ell^{-1})) \\
   \label{E:k1212-form}     \kappa_{(12)(12)}^{(\ell+1)}&=\frac{C_W^2 e^{-2\gamma}}{nn_0^2\ell}(1+O(\ell^{-1}))=\frac{\ell}{n}K_{(11)}^{(\ell)}K_{(22)}^{(\ell)}(1+O(\ell^{-1}))  
\end{align}
plus errors of size $O(n^{-2})$. Here, $\gamma$ denotes the Euler-Mascheroni constant. 
\end{proposition}


\begin{proof}
This proof is a long but straightforward calculation using Theorem \ref{T:pert-exp}, Lemma \ref{L:rec-lemma}, and the estimates from Lemma  \ref{L:2pt-deriv-recs-leading}. We will spell this out for the first few recursions and then simply write the recursion and its solution. For example, Theorem  \ref{T:pert-exp} shows
\begin{align*}
    \kappa_{(00)(00)}^{(\ell+1)} = \frac{C_W^2}{n_\ell}\lr{\bk{\sigma^4}_{(\ell)}-\bk{\sigma^2}_{(\ell)}^2}+\lr{\chi_{||}^{(\ell)}}^2   \kappa_{(00)(00)}^{(\ell)}
\end{align*}
plus errors of size $O(n^{-2})$. A direction computation then yields
\[
\frac{C_W^2}{n_\ell}\lr{\bk{\sigma^4}_{(\ell)}-\bk{\sigma^2}_{(\ell)}^2} = \frac{2}{n_\ell}\lr{\kappal_{(00)}}^2\lr{1+O(\ell^{-1})}.
\]
Hence, setting $n_\ell=n$ we may apply Lemma \ref{L:rec-lemma} to obtain
\begin{align*}
    \kappa_{(00)(00)}^{(\ell)} =\frac{2}{3n\ell a^2}(1+O(\ell^{-1})).
\end{align*}
Similarly, we have 
\begin{align*}
    \kappa_{(10)(00)}^{(\ell+1)} &= \frac{C_W^2}{n_\ell}\lr{\bk{\sigma^3\sigma'd_1z}_{(\ell)}-\bk{\sigma^2}_{(\ell)}\bk{\sigma\sigma'd_1z}_{(\ell)}}\\
    &+\frac{C_W^2}{8}\left[2\kappa_{(00)(00)}^{(\ell)}\bk{\partial^2\sigma^2}_{(\ell)}\bk{\partial^2(\sigma\sigma')d_1z}_{(\ell)}+4\kappa_{(10)(00)}^{(\ell)}\bk{\partial^2\sigma^2}_{(\ell)}\bk{\partial(\sigma\sigma')}_{(\ell)}\right]
\end{align*}
plus errors of size $O(n^{-2})$. A direct computation shows that 
\[
\frac{C_W^2}{n_\ell}\lr{\bk{\sigma^3\sigma'd_1z}_{(\ell)}-\bk{\sigma^2}_{(\ell)}\bk{\sigma\sigma'd_1z}_{(\ell)}}= \frac{2\kappal_{00}\kappal_{10}}{n_\ell}\lr{1+O(\ell^{-1})}
\]
Hence, setting $n_\ell=n$ and  
\begin{align*}
    \kappa_{(10)(00)}^{(\ell+1)} &= -\frac{2C_Wx_{1;\alpha}e^{-2\gamma}}{3ann_0\ell^3}\lr{1+O(\ell^{-1})} + \lr{1-\frac{4}{\ell}+O(\ell^{-2})}\kappal_{(10)(00)}.
\end{align*}
Thus, Lemma \ref{L:rec-lemma} yields
\begin{align*}
    \kappa_{(10)(00)}^{(\ell+1)}&=-\frac{e^{-2\gamma}(x_{1;\alpha})^2}{3ann_0\ell^2} = O\lr{\frac{x_{1;\alpha}}{nn_0\ell^2}}.
\end{align*}
Further, applying Theorem \ref{T:pert-exp} gives
\begin{align*}
    \kappa_{(10)(10)}^{(\ell+1)} &= \frac{C_W^2}{n_\ell}\left[\bk{(\sigma\sigma'd_1z)^2}_{(\ell)}-\bk{\sigma\sigma'd_1z}_{(\ell)}^2\right]\\
    &+\frac{C_W^2}{8}\left[2\kappa_{(00)(00)}^{(\ell)}\bk{\partial^2 (\sigma\sigma') d_1z}^2+\kappal_{(10)(00)}8\bkl{\partial(\sigma\sigma')}\bkl{\partial^2(\sigma\sigma')d_1z}\right.\\
    &\left.\qquad \qquad +8 \kappa_{(10)(10)}^{(\ell)} \bk{\partial(\sigma\sigma')}_{(\ell)}^2\right].
\end{align*}
A direct computation shows 
\[
\frac{C_W^2}{n_\ell}\left[\bk{(\sigma\sigma'd_1z)^2}_{(\ell)}-\bk{\sigma\sigma'd_1z}_{(\ell)}^2\right] = \frac{\kappal_{11}\kappal_{00}}{n}\lr{1+O(\ell^{-1})}.
\]
Setting $n_\ell=n$ and using the formulas above shows
\begin{align*}
    \kappa_{(10)(10)}^{(\ell+1)} &= \frac{\kappal_{11}\kappal_{00}}{n}\lr{1+O(\ell^{-1})}+ \lr{1-\frac{4}{\ell}}\kappal_{(10)(10)}
\end{align*}
and hence by Lemma \ref{L:rec-lemma} we see that
\begin{align*}
    \kappa_{(10)(10)}^{(\ell+1)} &=\frac{C_W e^{-\gamma}}{3nn_0a\ell}\lr{1+O(\ell^{-1})}.
\end{align*}
Next, we obtain 
\begin{align*}
    \kappa_{(10)(20)}^{(\ell+1)}&=\frac{C_W^2}{n_\ell}\left[\bkl{(\sigma\sigma')^2 d_1zd_2z}-\bkl{\sigma\sigma' d_1z}\bkl{\sigma\sigma' d_2z}\right]\\
    &+\frac{C_W^2}{8}\left[\kappal_{(00)(00)}2\bkl{\partial^2(\sigma\sigma')d_1z}\bkl{\partial^2(\sigma\sigma')d_2z} \right.\\
    &\qquad \qquad + \kappal_{(10)(00)}4\bkl{\partial(\sigma\sigma')}\bkl{\partial^2(\sigma\sigma') d_2z}+\kappal_{(20)(00)}4\bkl{\partial(\sigma\sigma')}\bkl{\partial^2(\sigma\sigma') d_1z}\\
    &\qquad \qquad + \left.\kappal_{(10)(20)}8\bkl{\partial(\sigma\sigma')}^2\right].
\end{align*}
Keeping the leading order in $\ell$ we therefore have 
\begin{align*}
    \kappa_{(10)(20)}^{(\ell+1)}&=O\lr{\frac{x_{1;\alpha}x_{2;\alpha}}{nn_0^2\ell^4}} + \lr{1-\frac{4}{\ell}+O(\ell^{-2})}\kappal_{(10)(20)}.
\end{align*}
Thus, 
\begin{align*}
    \kappa_{(10)(20)}^{(\ell+1)}&=O\lr{\frac{x_{1;\alpha}x_{2;\alpha}}{nn_0^2\ell^3}}
\end{align*}
Next, 
\begin{align*}
    \kappa_{(11)(00)}^{(\ell+1)} &= \frac{C_W^2}{n_\ell}\left[\bk{(\sigma\sigma' d_1z)^2}_{(\ell)} - \bk{\sigma^2}_{(\ell)}\bk{(\sigma'd_1z)^2}_{(\ell)}\right]\\
    &+\frac{C_W^2}{8}\left[2\kappa_{(00)(00)}^{(\ell)}\bk{\partial^2 \sigma^2}_{(\ell)}\bk{\partial^2 (\sigma')^2 (d_1z)^2}_{(\ell)} +4 \kappa_{(10)(00)}^{(\ell)}\bk{\partial (\sigma')^2d_1z}_{(\ell)}\bk{\partial^2 \sigma^2}_{(\ell)}\right.\\
    &\qquad \qquad \left.\kappa_{(11)(00)}^{(\ell)} 4\bk{(\sigma')^2}\bk{\partial^2 \sigma^2}\right].
\end{align*}
We have
\[
\frac{C_W^2}{n_\ell}\left[\bk{(\sigma\sigma' d_1z)^2}_{(\ell)} - \bk{\sigma^2}_{(\ell)}\bk{(\sigma'd_1z)^2}_{(\ell)}\right] = O\lr{\frac{1}{nn_0\ell^4}}.
\]
Keeping the leading order in $\ell$ gives
\begin{align*}
    \kappa_{(11)(00)}^{(\ell+1)} &= \kappal_{11}\kappal_{(00)(00)} \frac{C_W}{2}\bkl{\partial^2 (\sigma')^2}\lr{1+O(\ell^{-1})}+\chi_{\perp}^{(\ell)}\chi_{||}^{(\ell)}\kappa_{(11)(00)}^{(\ell)}
\end{align*}
The solution is therefore
\begin{align*}
    \kappa_{(11)(00)}^{(\ell+1)} &=-\frac{2C_We^{-\gamma}}{3ann_0\ell}\lr{1+O(\ell^{-1})}.
\end{align*}
Next, 
\begin{align*}
    \kappa_{(12)(00)}^{(\ell+1)}&=\frac{C_W^2}{n_\ell}\left[\bkl{(\sigma')^2\sigma^2 d_1z d_2z}-\bkl{(\sigma')^2d_1zd_2z}\bkl{\sigma^2}\right]\\
    &=\frac{C_W^2}{8}\left[\kappal_{(00)(00)}2\bkl{\partial^2(\sigma')^2 d_1zd_2z}\bkl{\partial^2\sigma^2}\right.\\
    &\qquad \qquad+\kappal_{(10)(00)}4\bkl{\partial (\sigma')^2 d_2z}\bkl{\partial^2\sigma^2}+\kappal_{(20)(00)}4\bkl{\partial (\sigma')^2 d_1z}\bkl{\partial^2\sigma^2}\\
    &\qquad \qquad +\left.\kappal_{(12)(00)}4\bkl{(\sigma')^2}\bkl{\partial^2\sigma^2}\right].
\end{align*}
A direct computation shows that
\begin{align*}
    \frac{C_W^2}{n_\ell}\left[\bkl{(\sigma')^2\sigma^2 d_1z d_2z}-\bkl{(\sigma')^2d_1zd_2z}\bkl{\sigma^2}\right]= -a\kappal_{11}\lr{\kappal_{00}}^2 +O\lr{\frac{1}{n_0^2n\ell^4}}.
\end{align*}
Thus, to leading order in $\ell$ we have
\begin{align*}
    \kappa_{(12)(00)}^{(\ell+1)}&=-a\kappal_{11}\lr{\kappal_{00}}^2\lr{1+O(\ell^{-1})}+\chi_{||}^{(\ell)}\chi_{\perp}^{(\ell)}\kappal_{(12)(00)}
\end{align*}
and we conclude that
\begin{align*}
    \kappa_{(12)(00)}^{(\ell+1)}= O\lr{\frac{1}{nn_0\ell^2}}.
\end{align*}
Next, 
\begin{align*}
    \kappa_{(11)(10)}^{(\ell+1)} &= \frac{C_W^2}{n_\ell}\left[\bk{\sigma(\sigma'd_1z)^3}_{(\ell)}-\bk{\sigma\sigma'd_1z}_{(\ell)}\bk{(\sigma'd_1z)^2}_{(\ell)}\right]\\
    &+\frac{C_W^2}{8}\left[\kappa_{(00)(00)}^{(\ell)}2\bk{\partial^2(\sigma \sigma')d_1z}_{(\ell)}\bkl{\partial^2 (\sigma')^2 (d_1z)^2}\right.\\
    &\qquad \qquad +\kappa_{(10)(00)}^{(\ell)}4\lr{\bkl{\partial(\sigma\sigma')}\bkl{\partial^2 (\sigma')^2 (d_1z)^2}+2\bkl{\partial^2(\sigma\sigma')d_1z}\bkl{\partial(\sigma')^2d_1z}}\\
    &\qquad \qquad +\kappa_{(10)(10)}^{(\ell)}16\bkl{\partial(\sigma\sigma')}\bkl{\partial(\sigma')^2 d_1z}  +\kappa_{(11)(00)}^{(\ell)}2\bkl{\partial^2(\sigma\sigma')d_1z}\bkl{(\sigma')^2}\\
    &\qquad \qquad \left.+\kappa_{(11)(10)}^{(\ell)}8\bkl{(\sigma')^2}\bkl{\partial(\sigma\sigma')}\right].
\end{align*}
A direct computation shows that to leading order in $\ell$
\[
\frac{C_W^2}{n_\ell}\left[\bk{\sigma(\sigma'd_1z)^3}_{(\ell)}-\bk{\sigma\sigma'd_1z}_{(\ell)}\bk{(\sigma'd_1z)^2}_{(\ell)}\right] = \frac{\kappal_{10}\kappal_{11}}{n_\ell} =O\lr{\frac{1}{n_0^2n\ell^3}}.
\]
We since all other terms are at this order in $\ell$ or lower we obtain
\[
\kappa_{(11)(10)}^{(\ell+1)} =O\lr{\frac{1}{nn_0^2\ell^2}}.
\]
Next, 
\begin{align*}
    \kappa_{(12)(10)}^{(\ell+1)}&= \frac{C_W^2}{n_\ell}\left[\bkl{(\sigma')^3\sigma (d_1z)^2d_2z}-\bkl{(\sigma')^2d_1zd_2z}\bkl{\sigma\sigma'd_1z}\right]\\
    &=\frac{C_W^2}{8}\left[\kappal_{(00)(00)}2\bkl{\partial^2 (\sigma')^2 d_1z d_2z}\bkl{\partial^2(\sigma\sigma')d_1z}\right.\\
    &\qquad \qquad + \kappal_{(10)(00)}4\left\{\bkl{\partial (\sigma')^2d_2z}\bkl{\partial^2(\sigma\sigma')d_1z}+\bkl{\partial^2(\sigma')^2 d_1zd_2z}\bkl{\partial(\sigma\sigma')}\right\}\\
    &\qquad \qquad +\kappal_{(20)(00)}4\bkl{\partial(\sigma')^2 d_1z}\bkl{\partial^2(\sigma\sigma')d_1z}+\kappal_{(12)(00)}4\bkl{(\sigma')^2}\bkl{\partial^2(\sigma\sigma')d_1z}\\
    &\qquad \qquad+ \kappal_{(10)(10)}8\bkl{\partial (\sigma')^2 d_2z}\bkl{\partial (\sigma\sigma')}+\kappal_{(10)(20)}8\bkl{\partial (\sigma')^2 d_1z}\bkl{\partial(\sigma\sigma')}\\
    &\qquad \qquad \left.+ \kappal_{(12)(10)}8\bkl{(\sigma')^2}\bkl{\partial(\sigma\sigma')}\right].
\end{align*}
A direct computation shows that to leading order in $\ell$ we have
\[
\frac{C_W^2}{n_\ell}\left[\bkl{(\sigma')^3\sigma (d_1z)^2d_2z}-\bkl{(\sigma')^2d_1zd_2z}\bkl{\sigma\sigma'd_1z}\right] = O\lr{\frac{x_{2;\alpha}}{nn_0\ell^3}}
\]
and 
\begin{align*}
    \kappa_{(12)(10)}^{(\ell+1)}&= O\lr{\frac{x_{2;\alpha}}{nn_0^2\ell^3}}+ \lr{1-\frac{2}{\ell}+O(\ell^{-2})}\kappal_{(12)(10)}.
\end{align*}
Hence, 
\[
\kappa_{(12)(10)}^{(\ell+1)} = O\lr{\frac{x_{2;\alpha}}{nn_0^2\ell^2}}.
\]
Next, 
\begin{align*}
    \kappa_{(11)(20)}^{(\ell+1)} &= \frac{C_W^2}{n_\ell}\left[\bkl{(\sigma')^2\sigma (d_1z)^2 d_2z}-\bkl{(\sigma' d_1z)^2}\bkl{\sigma\sigma'd_2z}\right]\\
    &+\frac{C_W^2}{8}\left[\kappa_{(00)(00)}^{(\ell)}2\bk{\partial^2(\sigma')^2(d_1z)^2}_{(\ell)}\bkl{\partial^2 (\sigma\sigma') d_2z}\right.\\
    &\qquad \qquad +\kappal_{(10)(00)}8\bkl{\partial^2 (\sigma')^2 d_1z}\bk{\partial^2(\sigma\sigma')d_2z}+\kappal_{(20)(00)}4\bkl{\partial^2(\sigma')^2 (d_1z)^2}\bkl{\partial(\sigma\sigma')}\\
    &\qquad \qquad + \kappal_{(11)(00)}4\bkl{(\sigma')^2}\bkl{\partial^2(\sigma\sigma')d_2z}+\kappal_{(10)(20)}16\bkl{\partial(\sigma')^2d_1z}\bkl{\partial(\sigma\sigma')}\\
    &\qquad \qquad +\left.\kappal_{(11)(20)}8\bkl{\partial(\sigma\sigma')}\bkl{(\sigma')^2}\right].
\end{align*}
We therefore have
\begin{align*}
    \kappa_{(11)(20)}^{(\ell+1)} &= O\lr{\frac{x_{1;\alpha}x_{2;\alpha}}{nn_0^2\ell^3}} + \lr{1-\frac{4}{\ell}+O(\ell^{-2})} \kappal_{(11)(20)}.
\end{align*}
Hence,
\begin{align*}
    \kappa_{(11)(20)}^{(\ell+1)} &= O\lr{\frac{x_{1;\alpha}x_{2;\alpha}}{nn_0^2\ell^2}}.
\end{align*}
Next, 
\begin{align*}
    \kappa_{(11)(11)}^{(\ell+1)} &=\frac{C_W^2}{n_\ell}\left[\bkl{(\sigma'dz)^4}-\bkl{(\sigma'dz)^2}^2\right]\\
    &+\frac{C_W^2}{8}\left[\kappal_{(00)(00)}2\bkl{\partial^2 (\sigma')^2 (d_1z)^2}^2+\kappal_{(10)(00)}16\bkl{\partial (\sigma')^2 d_1z}\bkl{\partial^2 (\sigma')^2 (d_1z)^2}\right.\\
    &\qquad +\kappal_{(11)(00)}8\bkl{(\sigma')^2}\bkl{\partial^2 (\sigma')^2 (d_1z)^2} + \kappal_{(10)(10)}32\bkl{\partial (\sigma')^2 d_1z}^2\\
    &\qquad \left.+\kappal_{(11)(10)} 32 \bkl{(\sigma')^2}\bkl{\partial (\sigma')^2 d_1z}+\kappal_{(11)(11)} 8\bk{(\sigma')^2}^2\right]\\
    &=\frac{C_W^2}{n_\ell}\left[\bkl{(\sigma'dz)^4}-\bkl{(\sigma'dz)^2}^2\right]\\
    &+C_W^2\kappal_{(11)(00)}\bkl{(\sigma')^2}\bkl{\partial^2 (\sigma')^2 (d_1z)^2} +\kappal_{(11)(11)} \lr{\chi_{\perp}^{(\ell)}}^2.
\end{align*}
A direction computation shows that to leading order in $\ell$ 
\[
\frac{C_W^2}{n_\ell}\left[\bkl{(\sigma'dz)^4}-\bkl{(\sigma'dz)^2}^2\right] = \frac{2\lr{\kappal_{11}}^2}{n_\ell}.
\]
Hence, to leading order in $\ell$ we simply have
\begin{align*}
    \kappa_{(11)(11)}^{(\ell+1)} &=\lr{\frac{2}{n_\ell}\lr{\kappa_{11}^{(\ell)}}^2+(-2a)\kappal_{11}\kappal_{(11)(00)}}\lr{1+O(\ell^{-1})} +\lr{1-\frac{2}{\ell}}\kappal_{(11)(11)}.
\end{align*}
We thus find
\begin{align*}
    \kappa_{(11)(11)}^{(\ell+1)} &= \frac{7C_W^2e^{-2\gamma}}{3nn_0^2\ell}\lr{1+O(\ell^{-1})}.
\end{align*}
Next, 
\begin{align*}
    \kappa_{(11)(22)}^{(\ell+1)}&=\frac{C_W^2}{n_\ell}\left[\bkl{(\sigma')^4(d_1zd_2z)^2}-\bkl{(\sigma'd_1z)^2}\bkl{(\sigma'd_2z)^2}\right]\\
    &+\frac{C_W^2}{8}\left[\kappal_{(00)(00)}2\bkl{\partial^2 (\sigma')^2 (d_1z)^2}\bkl{\partial^2 (\sigma')^2 (d_2z)^2}\right.\\
    &\qquad \qquad + \lr{\kappal_{(10)(00)}+\kappal_{(20)(00)}}8\bkl{\partial^2 (\sigma')^2d_1z}\bkl{\partial^2(\sigma')^2 d_2z}\\
    &\qquad \qquad + \kappal_{(11)(00)}4\bkl{(\sigma')^2}\bkl{\partial^2 (\sigma')^2 (d_2z)^2} + \kappal_{(22)(00)} 4\bkl{ (\sigma')^2}\bkl{\partial^2(\sigma') (d_1z)^2}\\
    &\qquad \qquad + \kappal_{(10)(20)}32\bkl{\partial(\sigma')^2 d_1z}\bkl{\partial(\sigma')^2 d_2z}\\
    &\qquad \qquad + \kappal_{(11)(20)}16\bkl{(\sigma')^2}\bkl{\partial(\sigma')^2 d_2z}+\kappal_{(22)(10)}16\bkl{(\sigma')^2}\bkl{\partial(\sigma')^2 d_1z}\\
    &\qquad \qquad \left.+\kappal_{(11)(22)}8\bkl{(\sigma')^2}^2\right].
\end{align*}
Substituting the solutions above we find
\begin{align*}
    \kappa_{(11)(22)}^{(\ell+1)}&=\frac{4C_W^2 e^{-2\gamma}}{3nn_0^2\ell}\lr{1+O(\ell^{-1})}.
\end{align*}
Finally, 
\begin{align*}
    \kappa_{(12)(12)}^{(\ell+1)}&=\frac{C_W^2}{n_\ell}\left[\bkl{(\sigma')^4(d_1zd_2z)^2}-\bkl{(\sigma')^2d_1zd_2z}^2\right]\\
    &+\frac{C_W^2}{8}\left[\kappal_{(00)(00)}2\bkl{\partial (\sigma')^2 d_1z d_2z}^2\right.\\
    &\qquad \qquad +\kappal_{(10)(00)}8\bkl{\partial(\sigma')^2 d_2z}\bkl{\partial^2(\sigma')^2 d_1zd_2z}+\kappal_{(20)(00)}8\bkl{\partial(\sigma')^2 d_1z}\bkl{\partial^2(\sigma')^2 d_1zd_2z}\\ 
    &\qquad \qquad +\kappal_{(12)(00)}8\bkl{(\sigma')^2}\bkl{\partial^2(\sigma')^2d_1zd_2z}+\kappal_{(10)(20)}16\bkl{\partial(\sigma')^2d_2z}\bkl{\partial (\sigma')^2 d_1z}\\
    &\qquad \qquad +\kappal_{(10)(10)}8\bkl{\partial (\sigma')^2 d_2z}^2+\kappal_{(20)(20)}8\bkl{\partial (\sigma')^2 d_1z}^2\\
    &\qquad \qquad +\kappal_{(12)(10)}16\bkl{(\sigma')^2}\bkl{\partial (\sigma')^2 d_2z} + \kappal_{(12)(20)}16\bkl{(\sigma')^2}\bkl{\partial (\sigma')^2 d_1z}\\
    &\qquad \qquad \left. +\kappal_{(12)(12)}8\bkl{(\sigma')^2}^2\right].
\end{align*}
To leading order in $\ell$ we have
\begin{align*}
    \kappa_{(12)(12)}^{(\ell+1)}&=\frac{C_W^2}{n_\ell}\left[\bkl{(\sigma')^4(d_1zd_2z)^2}-\bkl{(\sigma')^2d_1zd_2z}^2\right]\lr{1+O(\ell^{-1})}+\lr{\chi_{\perp}^{(\ell)}}^2\kappal_{(12)(12)}.
\end{align*}
A direct computation shows
\[
\frac{C_W^2}{n_\ell}\left[\bkl{(\sigma')^4(d_1zd_2z)^2}-\bkl{(\sigma')^2d_1zd_2z}^2\right] = \frac{\kappal_{11}\kappal_{22}}{n}\lr{1+O(\ell^{-1})}.
\]
Hence, 
\begin{align*}
    \kappa_{(12)(12)}^{(\ell+1)}&=\frac{C_W^2e^{-2\gamma}}{nn_0^2\ell}\lr{1+O(\ell^{-1})}.
\end{align*}

\end{proof}

\subsection{Finite Width Corrections to 2nd Cumulant}\label{S:2pt-solns-tanh}
We remind the reader of our standing notation. Namely, we let $\sigma$ be an non-linearity from the $K_*=0$ universality class (see \S \ref{S:tanh-univ}). We also fix a network input $x_\alpha\neq 0\in \R^{n_0}$ and recall that 
\[
\kappa_{(ij)}^{(\ell)}:=\Cov\lr{\partial_{x_{i;\alpha}}z_{i;\alpha}^{(\ell)},\, \partial_{x_{j;\alpha}}z_{i;\alpha}^{(\ell)}} = \partial_{x_{i;\alpha_1}}\partial_{x_{j;\alpha_2}}\kappa_{\alpha_1\alpha_2}^{(\ell)}\bigg|_{x_{\alpha_1}=x_{\alpha_2}=x_\alpha},\qquad i,j=1,\ldots, n_0.
\]
The recursive description for $\kappa_{(ij)}^{(\ell)}$ and its correction
\[
S_{(ij)}^{(\ell)}:=\kappa_{(ij)}^{(\ell)}-K_{(ij)}^{(\ell)}
\]
relative to the infinite width limit was obtained in Proposition \ref{P:S-recs}. These recursions held for any non-linearity and any values of $C_b,C_W$. The following Lemma records the large $\ell$ behavior of $\kappa_{(ij)}^{(\ell)}$ when $\sigma$ belongs to the $K_*=0$ universality class and the network is tuned to criticality.

\begin{proposition}\label{P:2pt-deriv-recs-subleading}
Fix $\delta\in (0,1)$. We have
\begin{align}
    \label{E:k00}\kappa_{(00)}^{(\ell)} &= \left[1-\frac{\ell}{3n}+O(n^{-1})+O_\delta(\ell^{-1+\delta})+O(n^{-2})\right]K_{(00)}^{(\ell)}\\
    \label{E:k10}\kappa_{(10)}^{(\ell)} &=\left[1+O\lr{n^{-1}}+O(\ell^{-1})+O(n^{-2})\right]K_{(10)}^{(\ell)}\\
    \label{E:k11}\kappa_{(11)}^{(\ell)} & = \left[1+\frac{\ell}{3n}+O(n^{-1})+O(\ell^{-1}) + O\lr{\frac{(\sigma_{1}^{(0)})^2}{n_0\ell^2}}+O(n^{-2})\right]K_{(11)}^{(\ell)}\\
   \label{E:k12} \kappa_{(12)}^{(\ell)} &=O\lr{\left[1+\frac{\ell}{n}+  O(\ell^{-1})+O(n^{-2})\right]K_{(12)}^{(\ell)}},
\end{align}
where the implicit constants in the error terms $O(n^{-1}), O(\ell^{-1}),O((\sigma_{1}^{_{(0)}})^2/n_0\ell^2), O(\ell/n)$ depend only on $\sigma$, the implicit constant in $O_\delta(\ell^{-1+\delta})$ depends on $\delta$ and $\sigma$, and the implicit constants in $O(n^{-2})$ depend both on $\sigma$ and on $\ell$. 
\end{proposition}
\begin{remark}
The results of Proposition \ref{P:2pt-deriv-recs-subleading} make clear that the error terms $O(n^{-2})$ must depend on $\ell$. For instance, note that $\kappa_{(00)}^{(\ell)}$ must be non-negative and hence the $O(n^{-2})$ term is negligible only if $n\gg \ell$. 
\end{remark}
\begin{proof}
Recall that in Lemma \ref{L:2pt-deriv-recs-leading} we already the large $\ell$ behavior of
\begin{align*}
    K_{(ij)}^{(\ell)} &:=\lim_{n\gives \infty} \kappa_{(ij)}^{(\ell)},\quad i,j\in \set{0,1,2}.
\end{align*}
Define
\[
S_{(ij)}^{(\ell)}:=\kappa_{(ij)}^{(\ell)} - K_{(ij)}^{(\ell)},\qquad i,j\in \set{0,1,2}.
\]
By direct computation and Theorem \ref{T:pert-exp} we obtain
\begin{align*}
   \kappa_{(00)}^{(\ell+1)}& = C_W\E{\sigma(z_{i;\alpha}^{(\ell)})^2} =  C_W\left\{\bk{\sigma^2}_{\kappa_{(00)}^{(\ell)}}+\frac{1}{8}\kappa_{4;\alpha}^{(\ell)}\bk{\partial^4 \sigma^2}_{\kappa_{(00)}^{(\ell)}} + O(n^{-2})\right\}.
\end{align*}
Comparing this with the recursion \eqref{E:K00-rec} and using that by Theorem \ref{T:cumulants} we know $\kappa_{4;\alpha}^{(\ell)}=O(n^{-1})$ allows us to conclude
\begin{equation}\label{E:S-est}
S_{(00)}^{(\ell)} = O(n^{-1}),    
\end{equation}
where at the moment the implicit constant in the error term depends in an unknown way on $\ell$. However, substituting this estimate back into the recursion for $ \kappa_{(00)}^{_{(\ell+1)}}$ and integrating by parts yields
\begin{align*}
   S_{(00)}^{(\ell+1)}& = C_W\left\{\bk{\sigma^2}_{ \kappa_{(00)}^{(\ell)}} - \bk{\sigma^2}_{K_{(00)}^{(\ell)}} \right\} +\frac{C_W}{8}\kappa_{4;\alpha}^{(\ell)}\bk{\partial^4 \sigma^2}_{K_{(00)}^{(\ell)}} + O(n^{-2})\\
   &=\frac{C_W}{8}\kappa_{4;\alpha}^{(\ell)}\bk{\partial^4 \sigma^2}_{K_{(00)}^{(\ell)}} +\chi_{||;\alpha}^{(\ell)}S_{(00)}^{(\ell)} + O(n^{-2})
\end{align*}
Using \eqref{E:K00} we obtain  
\[
 S_{(00)}^{(\ell+1)}=-\frac{2}{3na\ell} (1+O(\ell^{-1})) + \lr{1-\frac{2}{\ell}}S_{(00)}^{(\ell)} + O(n^{-2}).
\]
Noting that $S_{(00)}^{(1)}=0$ and applying Lemma \ref{L:rec-lemma}, we conclude
\[
 S_{(00)}^{(\ell+1)}=-\frac{1}{3an} (1+O(\ell^{-1})) + O(n^{-2}),
\]
where as usual the implicit constant in the $O(\ell^{-1})$ error is independent of $n,\ell$ but the constant in $ O(n^{-2})$ may depend on $\ell$. Hence, by \eqref{E:K00} we find
\[
 S_{(00)}^{(\ell+1)} = \lr{-\frac{\ell}{3n} (1+O(\ell^{-1})) + O(n^{-2})}K_{(00)}^{(\ell)}.
\]
Combining this with the definition \eqref{E:S-def} of $S_{(00)}^{_{(\ell)}}$ yields \eqref{E:k00}. To obtain \eqref{E:k10}, we proceed similarly. Namely, Theorem \ref{T:pert-exp}, the estimate $S_{(00)}^{(\ell)}=O(n^{-1})$, and integration by parts yield
\begin{align*}
    \kappa_{(10)}^{(\ell+1)} &= C_W \E{\sigma(z_{i;\alpha}^{(\ell)})\sigma'(z_{i;\alpha}^{(\ell)})d_1z_{i;\alpha}^{(\ell)}}\\
    &= C_W \left\{\bk{\sigma\sigma' d_1z}_{\kappa^{(\ell)}} + \frac{1}{8}\left[\kappa_{(00)(00)}^{(\ell)}\bk{\partial^4(\sigma\sigma')d_1z}_{\kappa^{(\ell)}}+2\kappa_{(10)(00)}^{(\ell)} \bk{\partial^3(\sigma\sigma')}_{\kappa^{(\ell)}}\right]  \right\}\\
    &=\kappa_{(10)}^{(\ell)} \left[\chi_{||}^{(\ell)} +\frac{C_W}{2}S_{(00)}^{(\ell)} \bk{\partial^3(\sigma\sigma')}_{K^{(\ell)}} \right]+O(n^{-2})\\  &+\frac{C_W}{8}\left[\kappa_{(00)(00)}^{(\ell)}\bk{\partial^4(\sigma\sigma')d_1z}_{\kappa^{(\ell)}}+4\kappa_{(10)(00)}^{(\ell)} \bk{\partial^3(\sigma\sigma')}_{\kappa^{(\ell)}}\right].
\end{align*}
Hence, 
\begin{align*}
    S_{(10)}^{(\ell+1)}&= S_{(10)}^{(\ell)}\left[\chi_{||}^{(\ell)} +\frac{C_W}{2}S_{(00)}^{(\ell)} \bk{\partial^3(\sigma\sigma')}_{K^{(\ell)}} \right]+K_{(10)}^{(\ell)}\frac{C_W}{2}S_{(00)}^{(\ell)}\bk{\partial^3(\sigma\sigma')}_{K^{(\ell)}}\\
    &+\frac{C_W}{8}\left[\kappa_{(00)(00)}^{(\ell)}\bk{\partial^4(\sigma\sigma')d_1z}_{\kappa^{(\ell)}}+4\kappa_{(10)(00)}^{(\ell)} \bk{\partial^3(\sigma\sigma')}_{\kappa^{(\ell)}}\right] + O(n^{-2}).
\end{align*}
This recursion shows that $S_{(10)}^{(\ell)}=O(n^{-1})$ and hence we obtain 
\begin{align*}
    S_{(10)}^{(\ell+1)}&= \chi_{||}^{(\ell)}S_{(10)}^{(\ell)} +K_{(10)}^{(\ell)}\frac{C_W}{2}S_{(00)}^{(\ell)}\bk{\partial^3(\sigma\sigma')}_{K^{(\ell)}}\\
    &+\frac{C_W}{8}\left[\kappa_{(00)(00)}^{(\ell)}\bk{\partial^4(\sigma\sigma')d_1z}_{\kappa^{(\ell)}}+4\kappa_{(10)(00)}^{(\ell)} \bk{\partial^3(\sigma\sigma')}_{\kappa^{(\ell)}}\right] + O(n^{-2}).
\end{align*}
The asymptotics from the formula above for $S_{(00)}^{_{(\ell)}}$ and Proposition \ref{P:4pt-deriv-recs-tanh} we see that 
\[
\kappa_{(10)(00)}^{(\ell)} = - K_{(10)}^{(\ell)}\frac{C_W}{2}S_{(00)}^{(\ell)} + O\lr{\frac{1}{n\ell^3}}. 
\]
Hence, we find all together that 
\begin{align*}
    S_{(10)}^{(\ell+1)}&=\lr{1-\frac{2}{\ell}+O(\ell^{-2})}S_{(10)}^{(\ell)}+ O\lr{\frac{x_{1;\alpha}}{nn_0 \ell^3}} +O(n^{-2}) = O\lr{\frac{x_{1;\alpha}}{nn_0 \ell^2}} +O(n^{-2}).
\end{align*}
Combining this with \eqref{E:K10} confirms \eqref{E:k10}. We now proceed to deriving \eqref{E:k11}. By using Theorem \ref{T:pert-exp} and integrating by parts 
\begin{align*}
    \kappa_{(11)}^{(\ell+1)}& = \E{\lr{d_1 z_{i;\alpha}^{(\ell+1)}}^2}\\
    &=\E{\lr{\sigma'(z_{i;\alpha}^{(\ell)})d_1z_{i;\alpha}^{(\ell)}}^2}\\
    &=\kappa_{(11)}^{(\ell)}\left[\chi_{\perp}^{(\ell)} + \frac{C_W}{2}S_{(00)}^{(\ell)}\bkkl{\partial^2 (\sigma')^2}\right] + C_W\lr{\kappa_{(10)}^{(\ell)}}^2 \left[\bkkl{\partial^2 (\sigma')^2} + \frac{1}{2}S_{(00)}^{(\ell)}\bkkl{\partial^4(\sigma')^2}\right]\\
    &+\frac{C_W}{8}\left\{\kappa_{(00)(00)}^{(\ell)} \left[K_{(11)}^{(\ell)}\bkkl{\partial^4(\sigma')^2} + \lr{K_{(10)}^{(\ell)}}^2\bkkl{\partial^6 (\sigma')^2}\right]\right.\\
    &\qquad \qquad +\left.4\kappa_{(10)(00)}^{(\ell)}K_{(10)}^{(\ell)}\bkkl{\partial^4(\sigma')^2} + \lr{2\kappa_{(10)(10)}^{(\ell)}+4\kappa_{(11)(00)}^{(\ell)}}\bkkl{\partial^2(\sigma')^2}\right\} + O(n^{-2}).
\end{align*}
Comparing this with the recursion \eqref{E:K11-rec} for $K_{(11)}^{(\ell)}$ shows that $S_{(11)}^{(\ell)}=O(n^{-1})$ and yields the recursion 
\begin{align*}
    S_{(11)}^{(\ell+1)}&=\chi_{\perp}^{(\ell)} S_{(11)}^{(\ell)} +\frac{C_W}{2}K_{(11)}^{(\ell)} S_{(00)}^{(\ell)}\bkkl{\partial^2 (\sigma')^2}+O(n^{-2})\\
    &+C_W\lr{K_{(10)}^{(\ell)}}^2 \left[\bkkl{\partial^2 (\sigma')^2} + \frac{1}{2}S_{(00)}^{(\ell)}\bkkl{\partial^4(\sigma')^2}\right]+2C_WK_{(10)}^{(\ell)}S_{(10)}^{(\ell)}\bkkl{\partial^2 (\sigma')^2}\\
    &+\frac{C_W}{8}\left\{\kappa_{(00)(00)}^{(\ell)} \left[K_{(11)}^{(\ell)}\bkkl{\partial^4(\sigma')^2} + \lr{K_{(10)}^{(\ell)}}^2\bkkl{\partial^6 (\sigma')^2}\right]\right.\\
    &\qquad \qquad +\left.4\kappa_{(10)(00)}^{(\ell)}K_{(10)}^{(\ell)}\bkkl{\partial^4(\sigma')^2} + \lr{2\kappa_{(10)(10)}^{(\ell)}+4\kappa_{(11)(00)}^{(\ell)}}\bkkl{\partial^2(\sigma')^2}\right\}.
\end{align*}
Recalling the asymptotics from Lemma \ref{L:2pt-deriv-recs-leading} and Proposition \ref{P:4pt-deriv-recs-tanh} we find to leading order in $\ell$ and $n$ that
\begin{align*}
    S_{(11)}^{(\ell+1)}&=\chi_{\perp}^{(\ell)} S_{(11)}^{(\ell)} +\frac{C_W}{2}K_{(11)}^{(\ell)} S_{(00)}^{(\ell)}\bkkl{\partial^2 (\sigma')^2},
\end{align*}
and hence by \eqref{E:K11} we obtain
\[
S_{(11)}^{(\ell)} = \frac{C_We^{-\gamma}}{3nn_0}\lr{1+O(\ell^{-1})}+O(n^{-2})=\left[\frac{\ell}{3n}+O(n^{-1})+O(n^{-2})\right]K_{(11)}^{(\ell)}.
\]
Combining this with \eqref{E:K11} completes our derivation of \eqref{E:k11}. A very similar computation reveals
\begin{align*}
    S_{(12)}^{(\ell+1)} &= \chi_{\perp}^{(\ell)}S_{(12)}^{(\ell)} +C_W\kappa_{(10)}^{(\ell)}\kappa_{(20)}^{(\ell)}\left[\bkkl{\partial^2(\sigma')^2}+\frac{1}{2}S_{(00)}^{(\ell)}\bkkl{\partial^4(\sigma')^2}\right]\\
    &+\frac{C_W}{8}\left\{\kappa_{(00)(00)}^{(\ell)} \left[K_{(12)}^{(\ell)}\bkkl{\partial^4(\sigma')^2}+K_{(10)}^{(\ell)}K_{(20)}^{(\ell)}\bkkl{\partial^6(\sigma')^2}\right]\right.\\
    &\qquad \qquad + 2\lr{\kappa_{(10)(00)}^{(\ell)}K_{(20)}^{(\ell)}+\kappa_{(20)(00)}^{(\ell)}K_{(10)}^{(\ell)}}\bkkl{\partial^4(\sigma')^2}\\
    &\qquad \qquad \left.+ 2\lr{2\kappa_{(12)(00)}^{(\ell)}+\kappa_{(10)(20)}^{(\ell)}}\bkkl{\partial^2(\sigma')^2}\right\}+O(n^{-2}).
\end{align*}
Using the asymptotics from Proposition \ref{P:4pt-deriv-recs-tanh} and Lemma \ref{L:rec-lemma} we obtain 
\[
S_{(12)}^{(\ell)}=O\lr{\frac{x_{1;\alpha}x_{2;\alpha}}{nn_0^2\ell}}+O(n^{-2}).
\]
Combined with \eqref{E:K12} this proves \eqref{E:k12}.
\end{proof}

\section{Statements and Declarations}
As required by Springer, we provide here statements on funding and competing interests. This research was partially funded by the National Science Foundation through Grants  DMS-2143754, DMS-1855684, and DMS-2133806 as well support from an ONR MURI on Foundations of Deep Learning. Finally, the author declares that he has no conflicts of interest. 

\bibliographystyle{plain} 
\bibliography{bibliography}
\appendix

\section{Proof of Theorem \ref{T:iwl}} \label{A:iwl}
Our proof of Theorem \ref{T:iwl} closely follows the proof of Theorem 1.2 in \cite{hanin2021random}. Let us first recall the notation and assumptions. We fix $\sigma:\R\gives \R$ such that
\begin{itemize}
	\item There exists $r\geq 1$ so that the $r$-th derivative of $\sigma$ belongs to $L^\infty$.
	\item There exist $c,c'>0$ so that 
	\[
\norm{e^{-cx^{2-c'}}\frac{d^r}{dx^r}\sigma(x)}_{L^\infty}<\infty.
	\]
\end{itemize}
We take $C_b\geq$ and $C_W>0$ and consider a random depth $L$ neural network with input dimension $n_0$, output dimension $n_{L+1}$, hidden layer widths satisfying
\[
\exists c,C>0\text{ s.t. } cn \leq n_1,\ldots, n_L\leq Cn,\qquad n \gg 1,
\]
non-linearity $\sigma$ and random weights and biases as in \eqref{E:Wb-def}. We also fix a finite collection 
\[
x_{\mA}:=\set{x_\alpha,\quad \alpha \in \mA}
\]
of distinct network inputs as well as an integer $m$ and study for each $\ell$ the random vectors
\[
D^{\leq r}z_{\mA}^{(\ell)}:=\lr{D^{\leq r}z_{i;\mA},\quad i=1,\ldots, m},
\]
where
\[
D^{\leq r}z_{i;\mA} :=\lr{D_{\alpha}^J z_{i;\alpha}^{(\ell)},\quad \alpha \in \mA,\, i=1,\ldots, m,\abs{J}\leq r}
\]
are the derivatives of $z_{i;\mA}^{_{(\ell)}}$ of order at most $r$. Our goal is to show that, as $n\gives \infty$, the joint distribution of the random vectors $D^{\leq r}z_{i;\mA}^{_{(\ell)}}$ converges to that of of centered jointly Gaussian vectors that are independent for different $i$ and satisfy
\[
\lim_{n\gives \infty}\Cov\lr{D_{\alpha_1}^{J_1} z_{i;\alpha_1}^{(\ell)},\, D_{\alpha_2}^{J_2} z_{i;\alpha_2}^{(\ell)}} = D_{\alpha_1}^{J_1}D_{\alpha_2}^{J_2} K_{\alpha_1\alpha_2}^{(\ell)}, 
\]
where
\[
K_{\alpha_1\alpha_2}^{(\ell+1)} = C_b+C_W \bk{\sigma(z_\alpha)\sigma(z_\beta)}_{K^{(\ell)}},\quad K_{\alpha_1\alpha_2}^{(1)} = C_b + C_W \sum_{j=1}^{n_0}x_{j;\alpha_1}x_{j;\alpha_2}
\]
is the infinite width covariance from Theorem \ref{T:iwl}. To prove this, let us denote by $\mF^{(\ell)}$  the sigma algebra generated by the weigts and biases in layer up to and including $\ell$. Observe that, conditional on $\mF^{(\ell)}$, we have that $D^{\leq r}z_{i;\mA}^{_{(\ell)}}$ are already independent for different $i$ and that, since the weights and biases are Gaussian, each is a centered Gaussian with conditional covariance
\[
\Cov\lr{D_{\alpha_1}^{J_1} z_{i;\alpha_1}^{(\ell+1)},\, D_{\alpha_2}^{J_2} z_{i;\alpha_2}^{(\ell+1)} ~\bigg|~\mF^{(\ell)}}= D_{\alpha_1}^{J_1}D_{\alpha_2}^{J_2} \Sigma_{\alpha_1\alpha_2}^{(\ell)},
\]
where
\[
\Sigma_{\alpha_1\alpha_2}^{(\ell)} = C_b + \frac{C_W}{n_{\ell}}\sum_{j=1}^{n_\ell}\sigma\lr{z_{j;\alpha_1}^{(\ell)}} \sigma\lr{z_{j;\alpha_2}^{(\ell)}}.
\]
Thus, by the continuous mapping theorem, it suffices to show that for any multi-indices $J_1,J_2$ with $\abs{J_i}\leq r$ and any $\alpha_1,\alpha_2\in \mA$ we have
\begin{equation}\label{E:iwl-goal1}
\lim_{n\gives \infty}\E{D_{\alpha_1}^{J_1}D_{\alpha_2}^{J_2}\Sigma_{\alpha_1\alpha_2}^{(\ell)}} \quad \text{exists and is finite}
\end{equation}
and 
\begin{equation}\label{E:iwl-goal2}
\lim_{n\gives \infty} \Var\left[D_{\alpha_1}^{J_1}D_{\alpha_2}^{J_2}\Sigma_{\alpha_1\alpha_2}^{(\ell)}\right] =0.
\end{equation}
We establish \eqref{E:iwl-goal1} and \eqref{E:iwl-goal2} by induction on $\ell$ the following more general statement.


\begin{proposition}\label{P:iwl}
Denote by $N(n_0,r)$ the number of derivatives of order at most $r$ in $n_0$ variables. Consider any measureable function $f:\R^{N(n_0,r)\times \abs{\mA}}\gives \R$ that is polynomially bounded, and define
\[
\mO_f^{(\ell)}:=\frac{1}{n_\ell}\sum_{j=1}^{n_\ell} f\lr{D^{\leq r}z_{j;\mA}}.
\]
Then, 
\begin{equation}\label{E:iwl-goal-11}
\lim_{n\gives \infty} \E{\mO_f^{(\ell)}} \quad \text{exists and is finite}
\end{equation}
and 
\begin{equation}\label{E:iwl-goal-21}
\lim_{n\gives \infty} \Var\left[\mO_f^{(\ell)}\right]=0.
\end{equation}
\end{proposition}


\begin{proof}
We proceed by induction, starting with $\ell=1$. Since weights and biases are Gaussian, the vectors $D^{\leq r}z_{i;\mA}^{(1)}$ are independent for all $i$ and jointly Gaussian. The polynomial growth assumption on $f$ show the moments of $f(x)$ are finite if $x$ is Gaussian. This allows us to apply the SLLN to conclude both \eqref{E:iwl-goal-11} and \eqref{E:iwl-goal-21}. 

Let us now assume we have proved \eqref{E:iwl-goal1} and \eqref{E:iwl-goal2} for layers $1,\ldots, \ell$. We start by fixing any polynomially bounded $f$  and establishing \eqref{E:iwl-goal1} at layer $\ell+1.$ We have
\begin{align*}
\E{\mO_{f}^{(\ell+1)}}&= \E{f\lr{D^{\leq r}z_{1;\mA}}}.
\end{align*}
As above, conditionl on $\mF^{(\ell)}$, we have the following equality in distribution: 
\begin{equation}\label{E:iwl-cond-gauss}
D^{\leq r}z_{1;\mA} \stackrel{d}{=} \lr{\Sigma^{\leq r, (\ell)}}^{1/2} Z,\qquad Z\sim \mN\lr{0,I_{N(n_0,r)\times \mA}}
\end{equation}
where $Z$ is independent of the conditional covariance matrix
\[
\Sigma^{\leq r, (\ell)} = \lr{D_{\alpha_1}^{J_1}D_{\alpha_2}^{J_2}\Sigma_{\alpha_1\alpha_2}^{(\ell)}}_{\alpha_1,\alpha_2\in \mA}.
\]
The key observation is that each entry of $\Sigma^{\leq r, (\ell)}$ is of the form $\mO_{f}^{(\ell)}$ for polynomially bounded $f$. Hence, we may apply the inductive hypothesis to conclude that there exists a matrix $\overline{\Sigma}^{_{\leq r, (\ell)}}$ such that the following convergence in distribution holds
\[
\Sigma^{\leq r, (\ell)} \quad \stackrel{d}{\longrightarrow}\quad  \overline{\Sigma}^{\leq r, (\ell)}\qquad \text{ as }\quad n\gives \infty.
\]
The polynomial growth assumption on $f$ together with the Skorohod representation theorem and dominated convergence show that 
\[
\lim_{n\gives \infty} \E{f\lr{D^{\leq r}z_{1;\mA}}} = \E{f\lr{\lr{ \overline{\Sigma}^{\leq r, (\ell)}}^{1/2} Z}}=:\overline{\mO}_{f}^{(\ell+1)}\quad \text{exists and is finite}.
\]
This proves \eqref{E:iwl-goal-11}. To show \eqref{E:iwl-goal-21}, we proceed similarly. Namely, we have
\begin{align*}
\Var\left[\mO_f^{(\ell+1)}\right] &= \frac{1}{n_{\ell+1}}\Var\left[f\lr{D^{\leq r}z_{1;\mA} ^{(\ell+1)}}\right] + \lr{1-\frac{1}{n_{\ell+1}}} \Cov\lr{f\lr{D^{\leq r}z_{1;\mA}^{(\ell+1)}},\, f\lr{D^{\leq r}z_{2;\mA} ^{(\ell+1)}}}.
\end{align*}
Note that
\[
\Var\left[f\lr{D^{\leq r}z_{1;\mA}^{(\ell+1)}}\right] \leq \E{\frac{1}{n}\sum_{j=1}^{n_{\ell}} \left[f\lr{D^{\leq r}z_{1;\mA} ^{(\ell+1)}}\right]^2}.
\]
Hence, since $f^2$ is also polynomially bounded we have already shown that \eqref{E:iwl-goal-11} holds at layer $\ell+1$, we see that
\[
\Var\left[\mO_f^{(\ell+1)}\right]= \Cov\lr{f\lr{D^{\leq r}z_{1;\mA}^{(\ell+1)}},\, f\lr{D^{\leq r}z_{2;\mA} ^{(\ell+1)}}} + O(n^{-1}).
\]
Next, using that conditional on $\mF^{(\ell)}$ the vectors $D^{\leq r}z_{i;\mA}^{(\ell+1)}$ are independent for different $i$ we conclude from the law of total covariance that
\[
\Cov\lr{f\lr{D^{\leq r}z_{1;\mA}^{(\ell+1)}},\, f\lr{D^{\leq r}z_{2;\mA} ^{(\ell+1)}}} \leq \Var\left[\E{f\lr{D^{\leq r}z_{1;\mA}^{(\ell+1)}} ~\bigg|~ \mF^{(\ell)}}\right].
\]
Combining the equality in distribution \eqref{E:iwl-cond-gauss} with the polynomial growth condition on $f$ and the dominated convergence theorem we find 
\[
\lim_{n\gives \infty} \Var\left[\E{f\lr{D^{\leq r}z_{1;\mA}^{(\ell+1)}} ~\bigg|~ \mF^{(\ell)}}\right]  = \Var\left[\E{f\lr{\lr{\overline{\Sigma}^{\leq r, (\ell)}}^{1/2}Z}}\right] = 0.
\]
This completes the proof that \eqref{E:iwl-goal-21} holds at infinite width, establishing Proposition \ref{P:iwl}.

\end{proof}


\section{Infinite Width Analysis of Tanh-like Non-linearities}\label{S:iw-tanh}

The purpose of this section is to derive some basic properties of the infinite width variance recursion 
\begin{equation}\label{E:K-tanh-rec}
\kappa_{\alpha\alpha}^{(\ell+1)} = C_b+ C_W\bk{\sigma(z)^2}_{\kappa_{\alpha\alpha}^{(\ell)}}.
\end{equation}
We abbreviate
\[
\sigma_j:=\frac{1}{j!}\frac{d^j}{dx^j}\bigg|_{x=0}\sigma(x)
\]
and consider here the case when $\sigma$ that is a tanh-like non-linearity in the sense that $\sigma$ satisfies: 
\begin{itemize}
    \item $\sigma$ is smooth at $0$ with $\sigma_1\neq 0$
    \item $\sigma$ has the opposite sign as its second derivative 
\begin{equation}
\label{E:curvature} \text{for almost every } z,~ \mathrm{sgn}\lr{\sigma(z)\sigma''(z)}=-1.
\end{equation} 
    Note that this forces $\sigma_2=0$ and
\[
a:=-\frac{6\sigma_3}{\sigma_1}>0.
\]
\item $\sigma$ is sub-linear:
\begin{equation}
\label{E:sublinear}  \exists C>0\, \text{ s.t. }\forall z\in \R\, \abs{\sigma(z)}\leq \abs{\sigma_1z},
\end{equation}
\item $\sigma$ is controlled by its first few non-zero Taylor series coefficients at $0$:
\begin{equation}\label{E:taylor-zero}
\exists C\geq 0\text{ s.t. }\forall z\geq 0,\quad \sigma_1z +\sigma_3z^3\leq \sigma(z)\leq  \sigma_1z +\sigma_3z^3+ Cz^4
\end{equation}
\end{itemize}
We will be interested in understanding the recursion \eqref{E:K-tanh-rec} at criticality in the sense defined in \S \ref{S:crit}. Specifically, we remind the reader that this means we choose $C_b,C_W$ so that 
\begin{align*}
\exists K_*\geq 0\quad \text{ s.t. }\quad K_* &= C_b + C_W\bk{\sigma^2(z)}_{K_*}\\
\chi_{||}(K_*)&=\frac{C_W}{2}\bk{\partial^2 (\sigma(z)^2)}_{K_*}=1\\
\chi_{\perp}(K_*)&=C_W\bk{(\sigma'(z))^2}_{K_*}=1.
\end{align*}
Before stating our main result (Proposition \ref{P:tanh-crit}), let us explain intuitively what we expect. First of all, as we shall see in Proposition \ref{P:tanh-crit}, $\tanh$-like non-linearities requires $K_*=0$ for criticality. Second, by Taylor expanding the recursion \eqref{E:K-rec} around small values of $K_{\alpha\alpha}^{(\ell)}$ we find
\begin{align*}
K_{\alpha\alpha}^{(\ell+1)} &= K_{\alpha\alpha}^{(\ell)}-a \lr{K_{\alpha\alpha}^{(\ell)}}^2 + O\lr{\lr{K_{\alpha\alpha}^{(\ell)}}^3}.
\end{align*}
This is well-approximated by the ODE
\[
\frac{d}{dt}K(t) = -a K(t)^2,
\]
whose solution is
\[
K(t) = \lr{at + \frac{1}{K(0)}}^{-1}.
\]
This form for the solution has two important properties that we will check in Proposition \ref{P:tanh-crit} hold for the actual solution $K_{\alpha\alpha}^{(\ell)}$ to the discrete difference equation \eqref{E:K-rec}:
\begin{itemize}
	\item At large $t$, $K(t)$ tends to zero like $1/at$ plus an error of size roughly $O(t^{-2})$.
	\item The leading order behavior of $K(t)$ at large $t$ is independent of the initial condition. 
\end{itemize}


\begin{proposition}\label{P:tanh-crit}
If $\sigma$ is a tanh-like non-linearity in the sense defined above then criticality is achieved for $\sigma$ only with 
\begin{equation}\label{E:tanh-crit-cond}
    K_*=0,\qquad C_b=0,\qquad \text{and}\qquad C_W=\sigma_1^{-2}.
\end{equation}
Moreover, for every $\delta\in (0,1)$ we have
\begin{equation}\label{E:K-tanh}
K_{\alpha\alpha}^{(0)}>0\quad \Rightarrow \quad \sup_{\ell\geq 1} \ell^{2-\delta}\abs{K_{\alpha\alpha}^{(\ell)}-\frac{1}{a\ell}}<\infty.
\end{equation}
\end{proposition}
\begin{proof}
Note that for any $K\geq 0$ we have
\begin{equation}\label{E:para-perp-tanh}
\chi_{||}(K) = \chi_{\perp}(K) + C_W\bk{\sigma(z)\sigma''(z)}_{K}.
\end{equation}
Hence, at criticality, we must have
\[
\bk{\sigma(z)\sigma''(z)}_{K_*}=0.
\]
But due to assumption \eqref{E:curvature} we have
\[
K>0\quad \Longrightarrow \quad \bk{\sigma(z)\sigma''(z)}_{K}<0.
\]
Thus, we indeed find that we must have $K_*=0$ at criticality. Hence, in light of \eqref{E:para-perp-tanh} criticality is equivalent to the system of equations
\[
K_*=0=C_b+C_W\sigma(0)^2,\qquad \chi_{||}(0)=\chi_{\perp}(0) = C_W\bk{(\sigma'(z))^2}_0=C_W\sigma_1^2 = 1.
\]
This system has a unique solution:
\[
C_b=0,\quad C_W = \sigma_1^{-2},
\]
completing the proof of the criticality conditions \eqref{E:tanh-crit-cond}. Let us now establish \eqref{E:K-tanh}. First note that at criticality the sub-linearity condition \eqref{E:sublinear} guarantees that for all $\delta>0$ there exists $c_\delta\in (0,1)$ such that
\[
K>\delta \quad \Longrightarrow \quad C_W\bk{\sigma(z)^2}_{K} < (1-c_\delta) \bk{z^2}_{K} = (1-c_\delta )K.
\]
Hence, for all $K,\delta>0$ there exists $\ell_0\geq 1$ such that
\begin{equation}\label{E:K-decay}
K_{\alpha\alpha}^{(0)}\leq K \qquad \Longrightarrow\qquad K_{\alpha\alpha}^{(\ell)}\leq \delta\quad \forall \ell \geq \ell_0.    
\end{equation}
In particular, $K_{\alpha\alpha}^{_{(\ell)}}$ is monotonically decreasing and converges to $K_*=0$ as $\ell$ grows. Let us now define for each $\ell\geq 1$
\[
K_{\alpha\alpha}^{(\ell)} =:\frac{1}{a\ell} + \epsilon^{(\ell)},\qquad a:=-6 \frac{\sigma_3}{\sigma_1} >0,
\]
where $a$ is positive due to \eqref{E:curvature}. Note that since $K_{\alpha\alpha}^{(\ell)}$ tends to zero with $\ell$, so does $\epsilon^{(\ell)}$. Let us agree that for any $t\in \R$ the symbol $t^+$ (resp. $t^-$) means that for $\ell$ sufficiently large we may make the constant $t^+$ (resp. $t^{-}$) arbitrary close to $t$ from above (resp. below). In order to prove \eqref{E:K-tanh}, we start with the following elementary estimate. 


\begin{lemma}
For all $\ell \geq 1$, we have
\begin{align}
  \label{E:tanh-ep-est-lb}   \epsilon^{(\ell+1)} ~\geq~ -\frac{1}{a\ell^2(\ell+1)}+\epsilon^{(\ell)}\lr{1-\frac{2}{\ell} - a\epsilon^{(\ell)}}.
\end{align}
Further, there exists a constant $C>0$ depending only on $\sigma$ with the following property. For all $K>0$ there exists a constant $\ell_0\geq 1$ so that if $K_{\alpha\alpha}^{_{(0)}}\leq K$, then for all $\delta\in (0,1)$ we have
\begin{align}
  \label{E:tanh-ep-est-ub}   \epsilon^{(\ell+1)} ~\leq~ \frac{C}{\ell^3} + \epsilon^{(\ell)}\lr{1-\frac{2-\delta}{\ell}},\qquad \forall~\ell\geq \ell_\delta:=\max\set{ \frac{C}{\delta},\,\frac{2C}{a}, \, \ell_0}.
\end{align}
\end{lemma}


\begin{proof}
Plugging in the estimates \eqref{E:taylor-zero} into the recursion \eqref{E:K-tanh-rec} yields for some $C>0$ depending only on $\sigma$
\begin{align*}
    \epsilon^{(\ell+1)}&\leq \frac{C}{\ell^3}+\epsilon^{(\ell)}\left[1-\frac{2}{\ell} +\frac{C}{\ell^2} \right]+\lr{\epsilon^{(\ell)}}^2\lr{-a+\frac{C}{\ell}}+C(\epsilon^{(\ell)})^3
\end{align*}
Note that for all $\ell\geq 2C/a$ we have $-a+C/\ell\leq 0$. Hence, for all $\ell\geq \max\set{2C/a, C/\delta}$ we have
\begin{align*}
    \epsilon^{(\ell+1)}&\leq \frac{C}{\ell^3}+\epsilon^{(\ell)}\left[1-\frac{2-\delta}{\ell} \right]+C(\epsilon^{(\ell)})^3.
\end{align*}
Moreover, if $\epsilon^{(\ell)}\leq 0$, then $(\epsilon^{(\ell)})^3\leq 0$. If on the other hand $\epsilon^{(\ell)}\geq 0$, then from \eqref{E:K-decay} we find that there for all $K>0$ there exists $\ell_0$ so that $(\epsilon^{(\ell)})^3\leq a(\epsilon^{(\ell)})^2/4$ for all $\ell\geq \ell_0.$ Hence, in all cases, for each $\delta\in (0,1)$ if $\ell\geq \max\set{2C/a, C/\delta, \ell_0},$ we find
\begin{align*}
    \epsilon^{(\ell+1)}&\leq \frac{C}{\ell^3}+\epsilon^{(\ell)}\left[1-\frac{2-\delta}{\ell}  \right],
\end{align*}
as claimed. The lower bound follows from a similar but simpler computation. 
\end{proof}
\noindent Fix $\delta\in (0,1)$. The relation \eqref{E:tanh-ep-est-ub}, together with Lemma \ref{L:rec-lemma}, show that for all $K>0$ there exists some $C'>0$ depending on $\delta,\sigma,K$ such that if $K_{\alpha\alpha}^{_{(0)}}\leq K$ then 
\[
\epsilon^{(\ell+1)} \leq \sum_{\ell'=\ell_\delta}^{\ell} \frac{C}{\ell^3} \prod_{\ell''=\ell'+1}^\ell \lr{1-\frac{2^-}{\ell}} \leq C'\left[\frac{1}{\ell^2} + \epsilon^{(\ell_\delta)} \frac{1}{\ell^{2-\delta}}\right].
\]
This shows that 
\begin{equation}\label{E:eps-tanh-ub}
\forall \delta\in (0,1)~\exists \ell_\delta \geq 1~\text{ s.t. }\quad \epsilon^{(\ell)}\leq \frac{1}{\ell^{2-\delta}}\quad \forall \ell\geq \ell_\delta.    
\end{equation}
To conclude \eqref{E:K-tanh} it therefore remains to deduce that 
\begin{equation}\label{E:eps-tanh-goal}
\forall K_1,K_2>0,\, \delta\in (0,1)~\exists \ell_\delta \geq 1~\text{ s.t. }\quad K_1<K_{\alpha\alpha}^{(0)}<K_2 \quad \Longrightarrow\quad \epsilon^{(\ell)}\geq -\frac{1}{\ell^{2-\delta}}\quad \forall \ell\geq \ell_\delta.
\end{equation}
To aid with this, we will need the following


\begin{lemma}
For any $\delta\in (0,1)$ there exists $\ell_\delta\geq 1$ with the property that if $\ell\geq \ell_\delta$ then
\[
\epsilon^{(\ell)} \geq -\ell^{-2+\delta}\quad \Longrightarrow \quad \epsilon^{(\ell+1)}\geq -(\ell+1)^{-2+\delta}.
\]
\end{lemma}

\begin{proof}
Suppose $\epsilon^{(\ell)} \geq -\ell^{-2+\delta}$. The lower bound in \eqref{E:tanh-ep-est-lb} yields for some $C,C'>0$
\begin{align*}
    \epsilon^{(\ell+1)}+\lr{\ell+1}^{-2+\delta}& \geq \lr{\ell+1}^{-2+\delta}\left[1-\lr{1+\ell^{-1}}^{2-\delta}\right] - 2\ell^{-3+\delta} - C\ell^{-4+2\delta}\\
    &\geq \delta \ell^{-3+\delta} -C'(\ell^{-3}+\ell^{-4+2\delta}),
\end{align*}
which is non-negative for all $\ell$ sufficiently large. 
\end{proof}
We are now in a position to establish \eqref{E:eps-tanh-goal}. In light of the previous Lemma we need only consider the case when 
\[
\forall \delta\in (0,1)~\exists ~\ell_\delta\geq 1\quad \text{s.t.}\quad \epsilon^{(\ell_\delta)}<-\ell^{2-\delta}.
\]
Note that in light of the upper bound \eqref{E:eps-tanh-ub} we find that for all $\delta\in (0,1)$ there exists $\ell_\delta\geq 1$ and $C_\delta>0$ so that for all $\ell\geq \ell_\delta$ we have
\begin{align*}
    K_{\alpha\alpha}^{(\ell+1)}\geq K_{\alpha\alpha}^{(\ell)}\lr{1-aK_{\alpha\alpha}^{(\ell)}}\geq K_{\alpha\alpha}^{(\ell)}\lr{1-a\lr{-\frac{1}{a\ell}+C_\delta \ell^{-2+\delta}}}=K_{\alpha\alpha}^{(\ell)}\lr{1-\frac{1}{\ell}-aC_\delta \ell^{-2+\delta}}.
\end{align*}
Hence, assuming $K_2\geq K_{\alpha\alpha}^{(0)}\geq K_1 >0,$ we may iterate this inequality to find that there exists $c>0$ depending on $K_1,K_2$ and $\ell_0\geq 1$ so that 
\[
K_{\alpha\alpha}^{(\ell)}\geq \frac{c}{a\ell} \qquad \forall \ell\geq \ell_0.
\]
Hence, since $\epsilon^{(\ell)}<0$ for all $\ell\geq \ell_\delta$ we find for all $\ell\geq \max\set{\ell_0,\ell_\delta}$ that 
\[
-a(\epsilon^{(\ell)})^2 \geq \epsilon^{(\ell)}\frac{1-c}{\ell}
\]
Substituting this into the lower bound \eqref{E:tanh-ep-est-lb}, we find that for all $\ell\geq \max\set{\ell_0,\ell_\delta}$
\[
\epsilon^{(\ell+1)}\geq -\frac{C'}{\ell^3} + \epsilon^{(\ell)}\lr{1-\frac{1+c}{\ell}}.
\]
Since $\epsilon^{(\ell_\delta)}<0$, we see by applying Lemma \ref{L:rec-lemma} that there exists $C>0$ so that for all $\ell\geq \max\set{\ell_0,\ell_\delta}$
\[
\epsilon^{(\ell+1)}\geq -\frac{C}{\ell^{1+c}}.
\]
But now we can bootstrap this estimate. Indeed, for any $\delta\in (0,1)$ we substitute this into the lower bound \eqref{E:tanh-ep-est-lb} to find that for all $\ell$ sufficiently large
\[
\epsilon^{(\ell+1)}\geq -\frac{C'}{\ell^3} + \epsilon^{(\ell)}\lr{1-\frac{2-\delta}{\ell}}.
\]
Again applying Lemma \ref{L:rec-lemma} yields that for all $\ell$ sufficiently large
\[
\epsilon^{(\ell+1)}\geq -\frac{C}{\ell^{2-\delta}}.
\]
This completes the proof.

\end{proof}

\section{Exact Solutions for $1$-homogeneous activations}\label{A:relu-exact}
In this appendix, we collect several known computations related to the distribution of neuron activations in random fully connected networks with $1-$homogeneous activations. Specifically, we fix a one homogeneous non-linearity 
\[
\sigma(t) = (a_+{\bf 1}_{t>0} + a_- {\bf 1}_{t<0})t
\]
and consider a random fully connected neural network with input dimension $n_0$, output dimension $n_{L+1}$, hidden layer widths $n_1,\ldots, n_\ell$, and non-linearity $\sigma$ that is tuned to criticality in the sense that
\[
C_b = 0,\qquad C_W = \frac{2}{a_+^2+a_-^2}.
\]
Our first task is to derive in \S \ref{S:relu-angle} a known exact formula for the infinite width covariance $K_{\alpha\beta}^{(\ell+1)}$ as a function of $K_{\alpha\alpha}^{(\ell)},K_{\alpha\beta}^{(\ell)},K_{\beta\beta}^{(\ell)}$. Then, in Section \S \ref{S:1-homog-limit}, we sketch a derivation of the limiting distribution \eqref{E:1-homog-limit} of a neuron pre-activation in the double scaling limit $n,L\gives \infty, L/n\gives \gamma$.

\subsection{Covariance Propagation in Random Fulluy Connected $1$-homogeneous Networks}\label{S:relu-angle}

In this section, we consider two network inputs $x_\alpha,x_\beta$ of the same norm:
\begin{equation}\label{E:same-norm}
K_{\alpha\alpha}^{(0)}=\frac{1}{n_0}\norm{x_\alpha}^2 = K =\frac{1}{n_0}\norm{x_\beta}^2= K_{\beta\beta}^{(0)},\qquad K>0.    \end{equation}
Let us define
\[
\epsilon_{\alpha\beta}^{(\ell)}:= \frac{1-\mathrm{Corr}_{\alpha\beta}^{(\ell)}}{2},\qquad \mathrm{Corr}_{\alpha\beta}^{(\ell)}:=\frac{K_{\alpha\beta}^{(\ell)}}{\lr{K_{\alpha\alpha}^{(\ell)}K_{\beta\beta}^{(\ell)}}^{1/2}}  
\]
Our goal is to derive the following explicit recursion for $\epsilon_{\alpha\beta}^{(\ell+1)}$ in terms of $\epsilon_{\alpha\beta}^{(\ell)}$. This derivation follows the approach in \S 5.5 \cite{roberts2022principles}. To the author's knowledge, the following formula (or really something equivalent) was first derived in \cite{cho2009kernel}.
\begin{proposition}[Correlation propagation for $1-$homogeneous activation functions]\label{P:1homog-angle-prop}
At criticality, we have the following exact formula:
\begin{align}
  \notag  1-2\epsilon_{\alpha\beta}^{(\ell+1)}&=\frac{2C_W(a_+-a_-)^2}{\pi}\left[\frac{1}{2}\sqrt{\epsilon_{\alpha\beta}^{(\ell)}(1-\epsilon_{\alpha\beta}^{(\ell)})}+\lr{\frac{1}{2}-\epsilon_{\alpha\beta}^{(\ell)}}\cos^{-1}\lr{\sqrt{\epsilon_{\alpha\beta}^{(\ell)}}}\right]\\
  \label{E:ep-rec}  &+C_Wa_+a_-(1-2\epsilon_{\alpha\beta}^{(\ell)})
\end{align}
In particular, taking $\eps_{\alpha\beta}^{(\ell)}$ small we find
\[\eps_{\alpha\beta}^{(\ell+1)} = \eps_{\alpha\beta}^{(\ell)} - \frac{4}{3\pi}\lr{\eps_{\alpha\beta}^{(\ell)}}^{3/2} +O\lr{\lr{\eps_{\alpha\beta}^{(\ell)}}^{5/2}}.\]
Hence, as $\ell\gives \infty$, 
\[
\eps_{\alpha\beta}^{(\ell)} = \frac{2}{3\pi}\ell^{-2}(1+o(1)).
\]
\end{proposition}

\begin{proof}
We have from Theorem \ref{T:iwl} that
\begin{equation}\label{E:K-rec-1homog}
K_{\alpha\beta}^{(\ell+1)} = C_b + C_W \bk{\sigma(z_\alpha)\sigma(z_\beta)}_{K^{(\ell)}},    
\end{equation}
where we recall that the brackets above mean the average with respect to the Gaussian distribution
\[
\lr{\begin{array}{c}
     z_\alpha  \\
     z_\beta
\end{array}}~\sim~\mathcal N\lr{0, \lr{\begin{array}{cc}
    K_{\alpha\alpha}^{(\ell)} & K_{\alpha\beta}^{(\ell)} \\
    K_{\alpha\beta}^{(\ell)} & K_{\beta\beta}^{(\ell)} 
\end{array}}}.
\]
Since we are at criticality, we have
\[
C_b= 0,\qquad C_W = \frac{2}{a_+^2+a_-^2}
\]
and that moreover
\[
 K_{\alpha\alpha}^{(\ell)}  =  K_{\beta\beta}^{(\ell)} = K,
\]
where $K$ is the constant from \eqref{E:same-norm}. Our first step is to change from the Gaussian variables $z_\alpha,z_\beta$ to the new Gaussian variables
\[\xi = \frac{z_\alpha + z_\beta}{2\sqrt{K}},\qquad \eta = \frac{z_\alpha-z_\beta}{2\sqrt{K}}.\]
We have
\[z_\alpha = \sqrt{K}(\xi +\eta),\qquad z_\beta = \sqrt{K}\lr{\xi-\eta}.\]
Moreover, writing 
\[
\epsilon:=\epsilon_{\alpha\beta}^{(\ell)}= \frac{1}{2}\lr{1-\frac{K_{\alpha\beta}^{(\ell)}}{\lr{K_{\alpha\alpha}^{(\ell)}K_{\beta\beta}^{(\ell)}}^{1/2}}}
\]
we find
 \begin{align*}
\Var[\xi]  =1-\eps,\qquad \Var[\eta] = \eps,\qquad 
\Cov[\xi,\eta] = 0.
\end{align*}
Hence, the right hand side of the recursion \eqref{E:K-rec-1homog} reads
\begin{equation*}
C_WK\int_\R\int_\R \sigma\lr{(1-\eps)^{1/2}\xi+\eps^{1/2} \eta}\sigma\lr{(1-\eps)^{1/2}\xi-\eps^{1/2} \eta} \exp\left[-\frac{1}{2}\lr{\xi^2 + \eta^2}\right] \frac{d\xi d\eta}{2\pi}.
\end{equation*}
Using the definition of $\sigma$ yields
\begin{align*}
    &\sigma\lr{(1-\eps)^{1/2}\xi+\eps^{1/2} \eta}\sigma\lr{(1-\eps)^{1/2}\xi-\eps^{1/2} \eta}\\
    &\qquad =\lr{a_+{\bf 1}_{\xi + \eta >0}+a_-{\bf 1}_{\xi+\eta <0}}\lr{a_+{\bf 1}_{\xi - \eta >0}+a_-{\bf 1}_{\xi-\eta <0}} ((1-\epsilon)\xi^2-\epsilon\eta^2).
\end{align*}
Changing variables $(\xi,\eta)\rightarrow (-\xi,-\eta)$ inside the integral and averaging yields
\begin{align*}
    K_{\alpha\beta}^{(\ell+1)}&= C_WKa_+a_-(1-2\epsilon) \\
    &+ \frac{C_WK(a_+-a_-)^2}{2} \int_{\R^2} {\bf 1}_{(1-\epsilon)\xi^2-\epsilon \eta^2>0} ((1-\epsilon)\xi^2-\epsilon \eta^2) \exp\left[-\frac{1}{2}\lr{\xi^2 + \eta^2}\right]\frac{d\xi d\eta}{2\pi}.
\end{align*}
Passing to polar coordinates and explicitly computing the resulting integral is now straightforward and completes the derivation of \eqref{E:ep-rec}. 
\end{proof}

\subsection{Full Distribution of Neuron Pre-activations at a Single Input and the Derivation of \eqref{E:1-homog-limit}}\label{S:1-homog-limit}
Our purpose in this section is to briefly recall an exact formula for the full distribution of a neuron pre-activation $z_{i;\alpha}^{(L+1)}$. For this, note that since $x_\alpha\mapsto z_\alpha^{(L+1)}$ is piecewise linear and the event that the Jacobian $J_{x_\alpha}z_{\alpha}^{(L+1)}$ is not well-defined at $x_\alpha$ has probability zero, we may write
\[
z_{\alpha}^{(L+1)} = J_{x_\alpha}z_{\alpha}^{(L+1)} x_\alpha. 
\]
Next, 
\begin{equation}\label{E:Jac-prod}
J_{x_\alpha}z_{\alpha}^{(L+1)} = W^{(L+1)}D^{(L)}W^{(L)}\cdots D^{(1)}W^{(1)},    
\end{equation}
where $W^{(\ell)}$ are simply the weight matrices and 
\[
D^{(\ell)}:=\mathrm{Diag}\lr{\sigma'(z_{i;\alpha}^{(\ell)}),\, i=1,\ldots, n_\ell}.
\]
Arguing exactly as in Proposition 2 of \cite{hanin2020products}, we have the following equality in distribution:
\[
 D^{(L)}W^{(L)}\cdots D^{(1)}W^{(1)}\stackrel{d}{=}  A\widehat{D}^{(L)}W^{(L)}\cdots \widehat{D}^{(1)}W^{(1)},
\]
where $A$ is a diagonal matrix with iid $\pm 1$ entries on the diagonal that is independent of $W^{(\ell)}$ and the diagonal matrices
\[
\widehat{D}^{(\ell)}=\mathrm{Diag}\bigg(\underbrace{a_+\xi_i^{(\ell)} + a_-(1-\xi_{i}^{(\ell)})}_{=:d_{i}^{(\ell)}},\, i=1,\ldots, n_\ell\bigg),\qquad \xi_i^{(\ell)} \sim \mathrm{Bernoulli}(1/2)\,\, iid.
\]
Combining this with \eqref{E:Jac-prod} and recalling that the entries of $W^{(L+1)}$ are iid centered Gaussians with variance $C_W/n_L$ yields
\[
z_{i;\alpha}^{(L+1)} \stackrel{d}{=}Z_1\cdot \lr{\frac{C_W}{n_L}}^{1/2} \norm{\widehat{D}^{(L)}W^{(L)}\cdots \widehat{D}^{(1)}W^{(1)}x_\alpha},
\]
where $Z_1\sim \mN(0,1)$ is independent of $\widehat{D}^{(\ell)},W^{(\ell)},\, i=1,\ldots, L$. Further, due to the right orthogonal invariance of the Gaussian matrices $W^{(\ell)}$ and the normalization that the variance of the entries of $W^{(\ell)}$ is $C_W/n_{\ell-11}$, we have that 
\begin{align*}
&\log \left[\lr{\frac{C_W}{n_L}}^{1/2}  \norm{\widehat{D}^{(L)}W^{(L)}\cdots \widehat{D}^{(1)}W^{(1)}x_\alpha}\right]\\ &\qquad\stackrel{d}{=}\frac{1}{2}\log\left[\frac{C_W}{n_0}\norm{x_\alpha}^2 \right]+\sum_{\ell=1}^L \frac{1}{2}\log\left[\frac{C_W}{n_L}\norm{\widehat{D}^{(\ell)}\widehat{W}^{(\ell)}u^{(\ell)}}^2\right]    
\end{align*}
where $u^{(\ell)}\in \R^{n_{\ell-1}}$ is collection of deterministic unit vectors and $\widehat{W}^{(\ell)}$ are independent random matrices with iid standard Gaussian entries. The summands on the previous line are independent and are each distributed like the logarithm of a randomly weighted $\chi^2$ random viable:
\[
\frac{C_W}{n_\ell}\norm{\widehat{D}^{(\ell)}W^{(\ell)}u^{(\ell)}}^2 \stackrel{d}{=} \frac{C_W}{n_\ell}\sum_{j=1}^{n_\ell} \lr{d_i^{(\ell)}}^2 \lr{Z_i^{(\ell)}}^2,
\]
where $Z_i^{(\ell)}\sim \mN(0,1)$ are iid and independent of $d_i^{(\ell)}$. Putting this all together, we find that
\begin{align*}
    z_{i;\alpha}^{(L+1)}&\stackrel{d}{=} \lr{\frac{C_W}{n_0}\norm{x_\alpha}^2}^{1/2}\cdot Z_1\cdot  \prod_{\ell=1}^L \lr{\frac{C_W}{n_\ell}}^{1/2}\norm{\widehat{D}^{(\ell)}W^{(\ell)}u^{(\ell)}}
\end{align*} 
is a product of $L+1$ independent random variables. Moreover, a direct computation shows that
\begin{align*}
\E{\log\left[\frac{C_W}{n_\ell}\sum_{j=1}^{n_\ell} \lr{d_i^{(\ell)}}^2 \lr{Z_i^{(\ell)}}^2 \right]}&= -\frac{1}{2}\Var\left[\frac{C_W}{n_\ell}\sum_{j=1}^{n_\ell} \lr{d_i^{(\ell)}}^2 \lr{Z_i^{(\ell)}}^2\right] + O(n_\ell^{-2})\\
&=-\frac{1}{2n_\ell}\lr{6\frac{a_+^4+a_-^4}{(a_+^2+a_-^2)^2}-1} + O(n_\ell^{-2})
\end{align*}
and also that
\begin{align*}
\Var\left[\log\left[\frac{C_W}{n_\ell}\sum_{j=1}^{n_\ell} \lr{d_i^{(\ell)}}^2 \lr{Z_i^{(\ell)}}^2 \right]\right]&= \Var\left[\frac{C_W}{n_\ell}\sum_{j=1}^{n_\ell} \lr{d_i^{(\ell)}}^2 \lr{Z_i^{(\ell)}}^2\right] + O(n_\ell^{-2})\\
&=\frac{1}{n_\ell}\lr{6\frac{a_+^4+a_-^4}{(a_+^2+a_-^2)^2}-1 }+ O(n_\ell^{-2}).
\end{align*}
Combining the preceding two estimates, taking $n,L\gives \infty$ with $L/n\gives \gamma$ and applying the CLT yields
\[
\lim_{\substack{n,L\gives \infty \\ L/n \gives \gamma\in [0,\infty)}} z_{i;\alpha}^{(L)}\quad \stackrel{d}{\longrightarrow}\quad \lr{\frac{C_W}{n_0}\norm{x_\alpha}^2}^{1/2} Z_1  \exp\left[-\mu(\gamma,a_+,a_-)+\sigma(\gamma,a_+,a_-)Z_2\right],    
\]
where
\[
\mu(\gamma,a_+,a_-) = \sigma^2(\gamma,a_+,a_-) := \frac{\gamma}{4}\lr{6\frac{a_+^4+a_-^4}{(a_+^2+a_+^2)^2}-1},\quad Z_1,Z_2\sim \mN(0,1)\text{ iid}.
\]
This is precisely the statement of \eqref{E:1-homog-limit}.
\end{document}